\documentclass[final, onefignum,onetabnum]{siamart190516}
\usepackage[english]{babel}
\usepackage{amsfonts}
\usepackage{graphicx}
\usepackage{subcaption}
\usepackage{epstopdf}
\usepackage{algorithm}
\usepackage{bm}
\usepackage{amsmath, amssymb}
\numberwithin{equation}{section}
\usepackage{enumitem}
\usepackage{siunitx}
\usepackage{booktabs}
\usepackage[sort]{cite}

\newtheorem{remark}{Remark}[section]

\newcommand{\myR}{\mathcal{R}}
\newcommand{\myQ}{\mathcal{Q}}
\newcommand{\myI}{\mathcal{I}}
\renewcommand{\div}{\operatorname{div}}

\DeclareMathOperator{\Tr}{Tr}

\newcommand{\bu}{\bm{u}}
\newcommand{\bv}{\bm{v}}
\newcommand{\bx}{\bm{x}}
\newcommand{\bnu}{\bm{\nu}}
\newcommand{\bmu}{\bm{\mu}}
\newcommand{\bl}{\bm{\lambda}}
\newcommand{\bchi}{\bm{\chi}}
\newcommand{\RM}{{\bf{RM}}}

\title{
Flux-mortar mixed finite element methods
on non-matching grids\thanks{
Submitted to the editors DATE.
\funding{
We thank the Deutsche Forschungsgemeinschaft (DFG, German Research Foundation) for supporting this work by funding SFB 1313, Project Number 327154368.
			}
		}
	}

\headers{Flux-mortar mixed finite element methods on non-matching grids}{W. M. Boon, D. Gl\"aser, R. Helmig, and I. Yotov}
\author{Wietse M. Boon\thanks{
			Department of Mathematics, KTH Royal Institute of Technology, 114 28 Stockholm, Sweden;
			\nobreak{\email{wietse@kth.se}}.
			}
		\and Dennis Gl\"aser\thanks{
			Institute for Modelling Hydraulic and Environmental Systems, University of Stuttgart, 70569 Stuttgart, Germany; \nobreak{\email{dennis.glaeser@iws.uni-stuttgart.de}, \email{rainer.helmig@iws.uni-stuttgart.de}}.
			}
		\and Rainer Helmig\footnotemark[3]
		\and Ivan Yotov\thanks{
		  Department of Mathematics, University of Pittsburgh, Pittsburgh, PA 15260, USA;
                  \nobreak{\email{yotov@math.pitt.edu}}.
                  Supported in part by NSF grant DMS 1818775.
			}
		}

\begin{document}

\maketitle

\begin{abstract}
We investigate a mortar technique for mixed finite element
approximations of Darcy flow on non-matching grids in which the normal
flux is chosen as the coupling variable. It plays the role of a
Lagrange multiplier to impose weakly continuity of pressure.  In the
mixed formulation of the problem, the normal flux is an essential
boundary condition and it is incorporated with the use of suitable
extension operators. Two such extension operators are considered and
we analyze the resulting formulations with respect to stability and
convergence. We further generalize the theoretical results, showing
that the same domain decomposition technique is applicable to a class
of saddle point problems satisfying mild assumptions. An example of
coupled Stokes-Darcy flows is presented.
\end{abstract}



\section{Introduction}
\label{sec:introduction}

The mortar mixed finite element method \cite{ACWY,APWY} has proven to
be an efficient and flexible domain decomposition technique for
solving a wide range of single-physics or multiphysics problems
described by partial differential equations in mixed formulations
coupled through interfaces with non-matching grids. The key attribute
of this method is the introduction of a Lagrange multiplier, referred
to as the mortar variable, on the interface that enforces continuity
of the solution. The method can be implemented as an iterative
algorithm that requires only subdomain solves on each iteration.

We consider porous media flow in a mixed formulation as our leading
example. In this context, the two most natural choices for the mortar
variable are the pressure or the normal Darcy flux. In the case of
matching grids, domain decomposition methods with these two types of
Lagrange multipliers were introduced in \cite{GW}. In \cite{ACWY}, a
mortar mixed finite element method on non-matching grids with a pressure mortar was
developed. A multiscale version of the method, referred to as the
multiscale mortar mixed finite element method (MMMFEM), was developed
in \cite{APWY}. In this case pressure continuity is enforced by
construction and normal flux continuity is enforced in a weak
sense. This strategy has been successfully applied to more general
applications as well, including coupled single-phase and multiphase
flows in porous media \cite{PesWheYot}, nonlinear elliptic problems
\cite{Park-nonlinear-MMMFE}, coupled Stokes and Darcy flows
\cite{GirVasYot,LSY,galvis-sarkis}, and mixed formulations of
linear elasticity \cite{khattatov-yotov}.

In this work, we explore the mortar mixed finite element method in
which the normal flux across each interface acts as the mortar
variable. In this case, normal flux continuity is imposed by
construction and continuity of pressure is imposed weakly. Our
specific interest lies in deriving {\em a priori} error estimates in
the presence of non-matching grids. To the best of our knowledge, such
analysis has not been previously done.

A challenge that arises in this approach is that the normal flux is an
essential boundary condition for mixed Darcy formulations and needs to
be incorporated accordingly. We achieve this by introducing
appropriate extension operators in the definitions of the velocity
spaces in both the continuous and the discrete settings. This involves
solving Neumann problems for each subdomain. We employ a Lagrange
multiplier to remedy both the potential incompatibility of the data as
well as the uniqueness of the solution.

Let us highlight the four main contributions presented in this
work. First, our focus is on non-matching grids and we quantify the
role this non-conformity plays in the context of {\em a priori} error
analysis. Two projection operators are proposed that require separate
analyses and lead to slightly different error estimates.  Second, we
consider a reduction of the problem to a symmetric, positive definite
system that contains only the mortar variable. An iterative scheme is
then proposed to solve this reduced system such that the solution
conserves mass locally at each iteration.  Third, the theoretical
framework is presented in a general setting that is applicable to a
broad class of saddle point problems. Fourth, we explicitly consider
an important example, namely coupled Stokes-Darcy problems,
\cite{LSY,DMQ,galvis-sarkis,GirVasYot}. A key component is developing
a flux-mortar finite element method for Stokes. Previously, only
normal stress mortar methods for Stokes have been considered, with the
mortar variable being used to impose weakly continuity of the velocity
\cite{BenBelgacem-Stokes,GirVasYot,Kim-Lee}. While a velocity Lagrange
multiplier has been employed in domain decomposition methods for
Stokes with matching grids \cite{Pavarino-Widlund,Li-Widlund}, to the
best of our knowledge this is the first Stokes discretization on
non-matching grids with velocity mortar variable. Moreover, we develop
a new parallel domain decomposition method for the Stokes-Darcy
problem, which satisfies velocity or flux continuity at each
iteration. We refer the reader to
\cite{Disc-Quart-Valli-2007,VasWangYot,Galvis-Sarkis-DD} for some of
the previous works on domain decomposition methods for coupled Stokes
and Darcy flows. In these works, flux continuity is either relaxed via
the use of Robin transmission conditions \cite{Disc-Quart-Valli-2007}
or it is satisfied only at convergence using pressure and normal stress
mortars \cite{VasWangYot,Galvis-Sarkis-DD}.

For the sake of clarity of the presentation, we focus on the case of
subdomain and mortar grids being on the same scale. However, the
flux-mortar mixed finite element method can be formulated as a
multiscale method via the use of coarse scale mortar grids, as was
done in \cite{APWY}.  In this case, following the approach in
\cite{Ganis-Yotov}, the method can be implemented using an interface
multiscale pressure basis, which can be computed by solving local
subdomain problems with specified mortar flux boundary data.

We note several similarities and relationships between the flux-mortar
method and existing schemes.  First, as mentioned, our approach is
dual to the pressure mortar technique that is central to the MMMFEM
\cite{ACWY,APWY}. Second, the multiscale hybrid-mixed (MHM) method
\cite{MHM-JCP,MHM-SINUM} similarly introduces flux degrees of freedom
on the interfaces between elements to impose weakly continuity of
pressure. The difference is that the MHM is defined on a single global
grid and it is based on an elliptic formulation, rather than a mixed
formulation.  The MHM is related to a special case of the mixed
subgrid upscaling method proposed in \cite{arbogast}. The latter does
not involve a Lagrange multiplier, but incorporates global flux
continuity via a coarse scale mixed finite element velocity space,
which may include additional degrees of freedom internal to the
subdomains. In contrast, our method reduces to an interface problem
involving only mortar degrees of freedom. Furthermore, the analysis in
\cite{arbogast} does not allow for non-matching grids along the coarse
scale interfaces.  Finally, we note that the flux-mortar mixed finite
element method has been successfully applied in the context of
fracture flow \cite{boon2018robust,nordbotten2019unified} and coupled
Stokes-Darcy flow \cite{Boon2019StokesDarcy}. The analysis in
\cite{boon2018robust} exploits that there is tangential flow along the
fractures and does not cover the domain decomposition framework
considered in this work, while the analysis in
\cite{Boon2019StokesDarcy} focuses on robust preconditioning.

The article is structured as follows. Section~\ref{sec:model_problem}
introduces the model problem and its domain decomposition formulation.
The mixed finite element discretization is
introduced in Section~\ref{sec:discretization} and we present two
projection operators to handle the non-matching grids. The well
posedness of the method is established in
Section~\ref{sec:well-posed}. Section~\ref{sec:a_priori_analysis}
provides {\em a priori} error estimates for the proposed
discretization. The problem is then reduced to an interface
formulation that involves only the mortar variables in
Section~\ref{sec:reduction_to_an_interface_problem}. We generalize
these concepts and results to a more abstract setting in
Section~\ref{sec:generalization_to_elliptic_problems} and show how to
apply them to a general class of saddle point problems. The general
framework is applied to the Stokes-Darcy problem. Finally,
we verify the {\em a priori} error analysis with numerical experiments
in Section~\ref{sec:numerical_results}.

\section{The model problem}
\label{sec:model_problem}

We introduce the flux-mortar method using an accessible model problem
given by the mixed formulation of the Poisson problem.
Let $\Omega \subset \mathbb{R}^n$, $n = 2,3$ be a bounded polygonal
domain. The model problem is
\begin{equation}\label{model}
  \bu = - K \nabla p, \quad \nabla\cdot \bu = f \ \mbox{ in } \Omega,
  \quad p = 0 \ \mbox{ on } \partial\Omega.
  \end{equation}
We will use the terminology common to porous media flow
modeling. Hence, we refer to $\bm{u}$ as the Darcy velocity, $p$ is
the pressure, $K$ is a uniformly bounded symmetric positive-definite
conductivity tensor, and $f\in L^2(\Omega)$ is a source function. We assume
that there exist $0 < k_{min} < k_{max} < \infty$ such that $\forall \bx \in \Omega$,
\begin{equation}\label{K-spd}
  k_{min} \xi^T\xi \le \xi^T K (\bx)\xi \le k_{max} \xi^T\xi, \quad \forall
  \xi \in \mathbb{R}^n.
\end{equation}

We will use the following standard notation. For $G$ a domain in
$\mathbb{R}^n$, $n = 2,3$, or a manifold in $\mathbb{R}^{n-1}$,
the Sobolev spaces on $G$ are denoted by $W^{k, p}(G)$. Let
$H^k(G) := W^{k, 2}(G)$ and $L^2(G) := H^0(G)$. The $L^2(G)$-inner product or duality pairing is
denoted by $(\cdot,\cdot)_G$. For $G \subset \mathbb{R}^n$, let
$$
H(\div, G) = \{\bm{v} \in (L^2(G))^n:
\nabla\cdot \bm{v} \in L^2(G)\}.
$$
We use the following shorthand notation to denote the norms of these spaces:
	\begin{align*}
		\| f \|_{k, G} &:= \| f \|_{H^k(G)}, &
		\| f \|_{G} &:= \| f \|_{0, G}, &
		\| \bm{v} \|_{\div, G}^2 & := \| \bm{v} \|_{H(\div, G)}^2 
                = \| \bm{v} \|_{G}^2 + \| \nabla \cdot \bm{v} \|_{G}^2.
	\end{align*}
%
We use the binary relation $a \lesssim b$ to imply that a
constant $C > 0$ exists, independent of the mesh size $h$, such that
$a \le C b$. The relationship $\gtrsim$ is defined analogously.

The variational formulation of \eqref{model} is: Find
$(\bm{u}, p) \in H(\div, \Omega) \times L^2(\Omega)$ such that
\begin{subequations}\label{weak-model}
	\begin{align}
		(K^{-1} \bm{u}, \bm{v})_\Omega
		- (p, \nabla \cdot \bm{v})_\Omega
		&= 0,
		& \forall \bm{v} &\in H(\div, \Omega), \label{weak-model-1}\\
		(\nabla \cdot \bm{u}, w)_\Omega
		&= (f, w)_\Omega,
		& \forall w &\in L^2(\Omega). \label{weak-model-2}
	\end{align}
\end{subequations}
It is well known that \eqref{weak-model} has a unique solution \cite{boffi2013mixed}.

\subsection{Domain decomposition}
\label{sub:domain_decomposition}

The domain $\Omega$ is decomposed into disjoint polygonal subdomains $\Omega_i$
with $i \in I_\Omega = \{1, 2, \ldots, n_\Omega \}$. Let $\bm{\nu}_i$
denote the outward unit vector normal to the boundary $\partial
\Omega_i$. The $(n - 1)$-dimensional interface between two subdomains
$\Omega_i$ and $\Omega_j$ is denoted by $\Gamma_{ij} :=
\partial \Omega_i \cap \partial \Omega_j$. Each interface
$\Gamma_{ij}$ is assumed to be Lipschitz and endowed with a unique,
unit normal vector $\bm{\nu}$ such that
%
$
	\bm{\nu} := \bm{\nu}_i = -\bm{\nu}_j 
	\text{ on } \Gamma_{ij}, \ i < j.
        $
%
Let $\Gamma := \bigcup_{i < j} \Gamma_{ij}$ and $\Gamma_i := \Gamma
\cap \partial \Omega_i$.  We categorize $\Omega_i$ as an interior
subdomain if $\partial \Omega_i \subseteq \Gamma$, i.e. if none of its
boundaries coincide with the boundary of the domain $\Omega$.
Let $I_{int} := \{ i \in I_\Omega :\ \partial \Omega_i \subseteq \Gamma \}$.
For given $\Omega_i$, let the local velocity and pressure function spaces
$V_i$ and $W_i$, respectively, be defined as
\begin{align*}
V_i := H(\div, \Omega_i), \quad
W_i := L^2(\Omega_i).
\end{align*}
Let the composite function spaces be defined as
\begin{align*}
V := \bigoplus_i V_i, \quad W := \bigoplus_i W_i = L^2(\Omega).
\end{align*}
Let 
\begin{align*}
  V_i^0 := \{ \bm{v}_{h, i} \in V_i :\ (\bm{\nu}_i \cdot \bm{v})|_{\Gamma_i} = 0\},
  \quad V^0 := \bigoplus_i V_i^0.
\end{align*}
The normal flux $\bm{\nu}\cdot\bm{u}$ on $\Gamma$ will be modeled by a Lagrange
multiplier $\lambda \in \Lambda$, with
\begin{align*}
\Lambda &:=
L^2(\Gamma).
\end{align*}
We note that $\Lambda$ has more regularity than the normal trace of $V_i$.
For $\lambda \in \Lambda$, we use a subscript to indicate its relative
orientation with respect to the adjacent subdomains:
\begin{align*}
 	\lambda_i := \lambda, \ \
 	\lambda_j := -\lambda \ \
 	\text{ on } \Gamma_{ij},\ i<j.
\end{align*}
In particular, $\lambda_i$ models $\bnu_i\cdot \bu$ and $\lambda_j$ models
$\bnu_j\cdot \bu$ on $\Gamma_{ij}$.

Next, we associate appropriate norms to the function spaces. The spaces $W$
and $\Lambda$ are equipped with the standard $L^2(\Omega)$ and
$L^2(\Gamma)$ norms, respectively. The space $V$, which  
does not have any continuity imposed on the interfaces, hence $V \not
\subset H(\div, \Omega)$, is equipped with a broken $H(\div)$
norm. Letting $\bm{v}_i = \bm{v}|_{\Omega_i}$, we define
\begin{align*}
	\| \bm{v} \|_V &:= \sum_i \| \bm{v}_{i} \|_{\div, \Omega_i}, &
	\| w \|_W &:= \| w \|_{\Omega}, &
	\| \mu \|_\Lambda &:= \| \mu \|_{\Gamma}.
\end{align*}

\section{Discretization}
\label{sec:discretization}

In this section we describe the flux-mortar mixed finite element method for 
\eqref{weak-model}. For subdomain $\Omega_i$, let $\Omega_{h,i}$ be a shape-regular
tessellation with typical mesh size $h$ consisting of affine finite
elements. The grids $\Omega_{h,i}$ and $\Omega_{h,j}$ may be
non-matching along the interface $\Gamma_{ij}$. Let $V_{h, i} \times
W_{h, i} \subset V_i \times W_i$ be a pair of conforming finite
element spaces that is stable for the mixed formulation of the Poisson
problem, i.e.,
\begin{subequations}
\begin{align}
  & \nabla \cdot V_{h, i} = W_{h, i}, \label{eq: div V = W} \\
  &
  \forall w_{h,i} \in W_{h,i}, \ \exists \, 0 \ne \bm{v}_{h,i} \in V_{h,i}:
  (\nabla\cdot\bv_{h,i},w_{h,i})_{\Omega_i}
  \gtrsim \|\bv_{h,i}\|_{\div, \Omega_i} \|w_{h,i}\|_{\Omega_i}. \label{local-inf-sup}
\end{align}
\end{subequations}
Let $V_{h, i}^0$ denote the subspace of $V_{h, i}$ with zero normal trace on $\Gamma$:
\begin{align*}
	V_{h, i}^0 &:=
	V_{h, i} \cap V_i^0, &
	V_h^0 &:= \bigoplus_i V_{h, i}^0.
\end{align*}
On the other hand, let the normal trace space on $\Gamma_i$ be denoted by $V_{h, i}^\Gamma$:
\begin{align*}
	V_{h, i}^\Gamma
	&:= (\bm{\nu}_i \cdot V_{h, i})|_{\Gamma_i}, &
	V_h^\Gamma
	&:= \bigoplus_i V_{h, i}^\Gamma,
\end{align*}
and let $\mathcal{Q}_{h, i}^\flat: \Lambda \to V_{h, i}^\Gamma$ and
$\mathcal{Q}_h^\flat: \Lambda \to V_h^\Gamma$ be the associated
$L^2$-projections. The reason for the superscript $\flat$ will become
clear shortly.

For the interfaces, we introduce a shape-regular affine tessellation of
$\Gamma_{ij}$, denoted by $\Gamma_{h,ij}$, with a typical mesh size
$h_\Gamma$.  Let the discrete interface space $\Lambda_{h,ij}\subset
L^2(\Gamma_{ij})$ contain continuous or discontinuous piecewise
polynomials on $\Gamma_{h,ij}$. Let $\Gamma_h = \bigcup_{i<j} \Gamma_{h,ij}$ and
$\Lambda_h = \bigoplus_{i<j} \Lambda_{h,ij}$.

An important restriction on $\Gamma_h$ and $\Lambda_h$ is that for $\mu_h \in \Lambda_h$, we assume that
\begin{align} \label{eq: flat mortar condition}
	\| \mu_h \|_{\Gamma_{ij}}
	&\lesssim
	\| \mathcal{Q}_{h, i}^\flat \mu_h \|_{\Gamma_{ij}}
	+ \| \mathcal{Q}_{h, j}^\flat \mu_h \|_{\Gamma_{ij}}, \quad
	\forall \, \Gamma_{ij}.
\end{align}
We emphasize that this is the conventional mortar assumption (see
e.g. \cite{ACWY}) implying that the mortar variable is controlled on
each interface by one of the two neighboring subdomains. The assumption
is easy to satisfy in practice and it has been shown to hold for some
very general mesh configurations \cite{APWY,PenYot}.

Next, let $\myR_{h, i}: \Lambda \to V_{h, i}$ be a bounded extension operator chosen to satisfy one of two properties, which we distinguish using a superscript $\sharp$ or $\flat$. The first option ($\sharp$) is to introduce an extension such that its normal trace has zero jump with respect to the mortar space:
\begin{subequations} \label{eqs: projection properties}
\begin{align} \label{eq: projection to mortar}
	\sum_i (\bm{\nu}_i \cdot \myR_{h, i}^\sharp \lambda, \mu_h)_{\Gamma_i} &= 0, &
	\forall \mu_h &\in \Lambda_h.
\end{align}
On the other hand, a second type of extension operators ($\flat$) is defined using the $L^2$-projection to each trace space $V_{h, i}^\Gamma$ such that:
\begin{align} \label{eq: projection from mortar}
	(\lambda_i - \bm{\nu}_i \cdot \myR_{h, i}^\flat \lambda, \xi_{h, i})_{\Gamma_i} &= 0, &
	\forall \xi_{h, i} &\in V_{h, i}^\Gamma.
\end{align}
\end{subequations}

The construction of these extension operators is described in detail in the following three subsections.
The global extension operator is defined as
$\myR_h \lambda := \bigoplus_{i} \myR_{h,i}\lambda$. 
After choosing $\myR_h$, we continue by defining the composite spaces $V_h$ and $W_h$ as
\begin{align} \label{eq: definition V_h}
	V_h
	&:=
	\bigoplus_i \left(V_{h, i}^0 \oplus \myR_{h, i} \Lambda_h \right)
	= V_h^0 \oplus \myR_h \Lambda_h,
	&
	W_h
	&:=
	\bigoplus_i W_{h, i}.
\end{align}
The two variants of $V_h$ that arise due to the choice of extension
operator are denoted by $V_h^\sharp$ and $V_h^\flat$. We will, from
now on, present the results that concern both variants by omitting the
superscript.

The flux-mortar mixed finite element method is as follows:
Find $(\bm{u}_h^0, \lambda_h, p_h) \in V_h^0 \times \Lambda_h
\times W_h$ such that
\begin{subequations}\label{dd-formulation}
\begin{alignat}2
& \sum_i (K^{-1} (\bm{u}_{h,i}^0 + \myR_{h,i} \lambda_h), \bm{v}_{h, i}^0)_{\Omega_i}
	- (p_{h,i}, \nabla \cdot \bm{v}_{h, i}^0)_{\Omega_i}
	= 0,
	 &\forall \bm{v}_{h, i}^0 \in V_{h,i}^0, \label{dd-1} \\
& \sum_i (K^{-1} (\bm{u}_{h,i}^0 + \myR_{h,i} \lambda_h), \myR_{h,i} \mu_h)_{\Omega_i}
	- (p_{h,i}, \nabla \cdot \myR_{h,i} \mu_h)_{\Omega_i}
	= 0,
	 &\forall \mu_h \in \Lambda_h, \label{dd-2} \\
& \sum_i (\nabla \cdot (\bm{u}_{h,i}^0 + \myR_{h,i} \lambda_h), w_{h,i})_{\Omega_i}
	= \sum_i (f, w_{h,i})_{\Omega_i},
	 &\forall w_{h,i} \in W_{h,i}. \label{dd-3}
\end{alignat}
\end{subequations}
Here, we use a subscript $i$ on a variable to denote its restriction to $\Omega_i$.
Letting $\bm{u}_h := \bm{u}_h^0 + \myR_h \lambda_h$ and
$\bm{v}_h := \bm{v}_h^0 + \myR_h \mu_h$, \eqref{dd-formulation} can be equivalently
written as: Find $\bm{u}_h \in V_h$ and $p_h \in W_h$ such that
\begin{subequations} \label{eqs: discrete problem 1}
	\begin{align}
		(K^{-1} \bm{u}_h, \bm{v}_h)_\Omega
		- \sum_i (p_h, \nabla \cdot \bm{v}_h)_{\Omega_i}
		&= 0,
		& \forall \bm{v}_h &\in V_h,\\
		\sum_i (\nabla \cdot \bm{u}_h, w_h)_{\Omega_i}
		&= (f, w_h)_\Omega,
		& \forall w_h &\in W_h.
	\end{align}
\end{subequations}
Note that the flux-mortar mixed finite element method \eqref{eqs: discrete problem 1}
is a non-conforming discretization of the weak formulation \eqref{weak-model},
since $V_h \not\subset H(\div, \Omega)$. We further emphasize that the
discrete trial and test functions from $V_h$ are naturally decomposed
into internal and interface degrees of freedom using $\myR_h$. This
will be used in the reduction to an interface problem in
Section~\ref{sec:reduction_to_an_interface_problem}. 

We next focus on the two types of extension operators $\myR_h^\sharp$
and $\myR_h^\flat$.

\subsection{Projection to the space of weakly continuous functions}
\label{sub: disc projection_to_the_mortar_space}

Let us first consider the projection operator $\myR_h^\sharp$ that
satisfies \eqref{eq: projection to mortar}. In its construction, we
use the concept of weakly continuous functions, as introduced in
\cite{ACWY} in the pressure-mortar method. In particular, let the
space of weakly continuous fluxes $V_{h, c}$ and the associated trace
space $V_{h, c}^\Gamma$ be given by
\begin{subequations}
	\begin{align*}
		V_{h, c}
		&:= \left\{ \bm{v}_h \in \bigoplus_i V_{h, i} :\
		\sum_i (\bm{\nu}_i \cdot \bm{v}_{h, i}, \mu_h)_{\Gamma_i} = 0, \
		\forall \mu_h \in \Lambda_h \right\}, \\
		V_{h, c}^\Gamma
		&:= \left\{ \xi_h \in V_h^\Gamma :\
		\sum_i (\xi_{h, i}, \mu_h)_{\Gamma_i} = 0, \
		\forall \mu_h \in \Lambda_h \right\}.
	\end{align*}
\end{subequations}

Let $\mathcal{Q}_h^\sharp: \Lambda \to V_{h, c}^\Gamma$ denote the
$L^2$-projection to $V_{h, c}^\Gamma$ and let $\mathcal{Q}_{h,
  i}^\sharp: \Lambda \to V_{h, i}^\Gamma$ be its restriction to the
trace space $V_{h, i}^\Gamma$.

We construct an extension satisfying property \eqref{eq: projection to
  mortar} by introducing a two-step process. We first solve the
following auxiliary problem, obtained from \cite{ACWY}: Given $\lambda
\in \Lambda$, find $\psi_h^\sharp \in V_h^\Gamma$ and $\chi_h \in
\Lambda_h$ such that
\begin{subequations} \label{eq: to-projection problem}
\begin{align}
	\sum_i (\lambda_i - \psi_{h, i}^\sharp - \chi_h, \xi_{h, i})_{\Gamma_i} &= 0, &
	\forall \xi_h &\in V_h^\Gamma, \label{eq: eq1 of psi sharp}\\
	\sum_i (\psi_{h, i}^\sharp, \mu_h)_{\Gamma_i} &= 0, &
	\forall \mu_h &\in \Lambda_h. \label{eq: weak continuity of psi}
\end{align}
\end{subequations}

\begin{lemma} \label{lem: sharp projection well-posed}
	Problem \eqref{eq: to-projection problem} admits a unique solution under assumption \eqref{eq: flat mortar condition}.
\end{lemma}
\begin{proof}
	Since \eqref{eq: to-projection problem} corresponds to a square system of equations, it suffices to show uniqueness. Hence, we set $\lambda = 0$ and choose $\xi_{h, i} = \psi_{h, i}^\sharp$ and $\mu_h = \chi_h$. It follows after summation of the two equations that $\psi_{h, i}^\sharp = 0$. In turn, it follows from the first equation and assumption \eqref{eq: flat mortar condition} that $\chi_h = 0$.
\end{proof}

\begin{lemma} \label{lem psi sharp preserves mean}
	The solution $\psi_h^\sharp$ of \eqref{eq: to-projection problem} is the $L^2$-projection of $\lambda$ onto $V_{h, c}^\Gamma$:
	\begin{align}\label{psi-Qh-sharp}
		\psi_h^\sharp = \mathcal{Q}_h^\sharp \lambda.
	\end{align}
	Moreover, it satisfies $(\lambda_i - \psi_{h, i}^\sharp, 1)_{\Gamma_{ij}} = 0$
	for each $\Gamma_{ij}$.
\end{lemma}
\begin{proof}
First, we note that the solution $\psi_h^\sharp \in V_{h, c}^\Gamma$
due to \eqref{eq: weak continuity of psi}. By choosing $\xi_h$ in
\eqref{eq: eq1 of psi sharp} from $V_{h, c}^\Gamma \subset
V_h^\Gamma$, we obtain
	\begin{align*}
		\sum_i (\lambda_i - \psi_{h, i}^\sharp, \xi_{h, i})_{\Gamma_i} &= 0, &
		\forall \xi_h &\in V_{h, c}^\Gamma.
	\end{align*}
Hence, $\psi_h^\sharp$ is the $L^2$-projection of $\lambda$ onto
$V_{h, c}^\Gamma$ which we denote by $\mathcal{Q}_h^\sharp \lambda$.

For the second result, we consider a given $\Gamma_{ij}$ and note that
$1 \in V_{h, i}^\Gamma \cap V_{h, j}^\Gamma$. Taking $\xi_{h,i} = \xi_{h,j} = 1$ on
$\Gamma_{ij}$ in
\eqref{eq: eq1 of psi sharp} and using that $1 \in \Lambda_{h,ij}$ and 
$\psi_h^\sharp \in V_{h,c}^\Gamma$, we derive
\begin{align*}
		 2(\chi_h, 1)_{\Gamma_{ij}}
		 &= (\lambda_i - \psi_{h, i}^\sharp, 1)_{\Gamma_{ij}}
		 + (\lambda_j - \psi_{h, j}^\sharp, 1)_{\Gamma_{ij}} \\
		 &= (\lambda_i + \lambda_j, 1)_{\Gamma_{ij}}
		 - (\psi_{h, i}^\sharp + \psi_{h, j}^\sharp, 1)_{\Gamma_{ij}}
		 = 0.
\end{align*}
Thus, taking $\xi_{h,i} = 1$ and $\xi_{h,j} = 0$ on $\Gamma_{ij}$
in \eqref{eq: eq1 of psi sharp} gives
$(\lambda_i - \psi_{h, i}^\sharp, 1)_{\Gamma_{ij}} = 0$.
\end{proof}

The obtained $\psi_{h, i}^\sharp$ is in the trace space $V_{h, i}^\Gamma$. Hence, the second step in the definition of $\myR_{h, i}^\sharp$ is to choose a bounded extension to the discrete space $V_{h, i}$ such that $\bm{\nu}_i \cdot \myR_{h, i}^\sharp \lambda = \psi_{h, i}^\sharp$ on $\Gamma_i$. For an explicit example of such an extension, we refer to Section~\ref{sub:discrete_extension}.

Let $V_h^\sharp := V_h^0 \oplus \myR_h^\sharp \Lambda_h$ be the discrete function space defined by this choice of extension operator.
Due to Lemma~\ref{lem psi sharp preserves mean}, we note that $V_h^\sharp \subseteq V_{h, c}$. However, the converse inclusion does not hold in general since the projection $\mathcal{Q}_h^\sharp$ is not necessarily surjective on $V_{h, c}^\Gamma$ when acting on $\Lambda_h$. As a direct consequence, the problem we set up in this space is closely related, but not equivalent, to the one introduced in \cite{ACWY}, Section~3. To be specific, we have used $\myR_h^\sharp$ to generate a strict subspace of $V_{h, c}$ whereas the problem in \cite{ACWY} is posed on $V_{h, c}$.

We make one additional assumption for this choice of extension
operator in analogy with assumption \eqref{eq: flat mortar condition},
namely that for all $\mu_h \in \Lambda_h$,
\begin{align} \label{eq: sharp mortar condition}
	\| \mu_h \|_{\Gamma_{ij}}
	\lesssim
	\| \mathcal{Q}_{h, i}^\sharp \mu_h \|_{\Gamma_{ij}}
	+ \| \mathcal{Q}_{h, j}^\sharp \mu_h \|_{\Gamma_{ij}}, \quad
	\forall \, \Gamma_{ij}.
\end{align}

\subsection{Projection to the trace spaces}
\label{sub: disc_projection_from_the_mortar_space}

An alternative choice of extension operators ($\flat$) aims to satisfy
\eqref{eq: projection from mortar}. In this case, we project from the
space $L^2(\Gamma_i)$ onto the trace space of $V_{h, i}$ for each
$i$. We follow a similar two-step process as in the previous
subsection. In the first step we solve the problem:
Given $\lambda \in \Lambda$, find $\psi_{h, i}^\flat \in
V_{h, i}^\Gamma$ such that
\begin{align} \label{eq: from-projection problem}
	(\lambda_i - \psi_{h, i}^\flat, \xi_{h, i})_{\Gamma_i} &= 0, &
	\forall \xi_{h, i} &\in V_{h, i}^\Gamma.
\end{align}

\begin{lemma} \label{lem psi flat preserves mean}
	The solution $\psi_{h, i}^\flat$ of \eqref{eq: from-projection problem} is the $L^2$-projection of $\lambda$ onto $V_{h, i}^\Gamma$:
	\begin{align}\label{psi-Qh-flat}
		\psi_{h, i}^\flat = \mathcal{Q}_{h, i}^\flat \lambda.
	\end{align}
	Moreover, it satisfies $(\lambda_i - \psi_{h, i}^\flat, 1)_{\Gamma_{ij}} = 0$
	for each $\Gamma_{ij}$.
\end{lemma}
\begin{proof}
	The first claim follows by definition, whereas the second follows from the fact that the indicator function of $\Gamma_{ij}$ is in the space $V_{h, i}^\Gamma$.
\end{proof}

The second step in the definition of $\myR_{h, i}^\flat$ is to choose
a bounded extension to the discrete space $V_{h, i}$ such that
$\bm{\nu}_i \cdot \myR_{h, i}^\flat \lambda = \psi_{h, i}^\flat$ on
$\Gamma_i$. An explicit example is given in
Section~\ref{sub:discrete_extension}. We refer to the resulting
function space as $V_h^\flat := V_h^0 \oplus \myR_h^\flat \Lambda_h$.

\begin{remark}
The extension operator $\myR_{h, i}^\flat$ does not explicitly use
the mortar condition \eqref{eq: flat mortar condition} in its
construction. However, as shown later in Section~\ref{sec:well-posed},
this condition remains necessary to ensure unique solvability of the
model problem posed in $V_h^\flat$.

The spaces $V_h^\sharp$ and $V_h^\flat$ are different in general, with
none contained in the other. This can be seen by the fact that both
spaces have the same, finite dimensionality but the extension
$\myR_h^\flat$ does not satisfy \eqref{eq: projection to mortar} in
general.

\end{remark}

\subsection{A discrete extension operator}
\label{sub:discrete_extension}

We next present the second step in the construction of the two
extension operators, which is similar for both cases. It is denoted by
$\myR_{h, i}^\sharp$ or $\myR_{h, i}^\flat$ depending on the
associated projection operator $\mathcal{Q}_{h}^\sharp$ or
$\mathcal{Q}_{h}^\flat$ used in the first step of the construction.
We refer to results concerning both extension operators by omitting
the superscript.

The discrete extension operator on each subdomain $\Omega_i$ will be defined
using a subdomain problem with Neumann data on $\Gamma_i$. 
For interior subdomains, $i \in I_{int}$, this results in Neumann boundary
conditions on the entire boundary $\partial \Omega_i$. 
To deal with possibly singular subdomain problems, we define the space
\begin{align} \label{eq: definition S_Hi darcy}
	S_{H, i} &:= \begin{cases}
		\mathbb{R}, & i \in I_{int} \\
		0, & i \notin I_{int}
	\end{cases}, &
	S_H &:= \bigoplus_i S_{H, i}.
\end{align}
The subscript $H$ is the characteristic subdomain size.

We construct a discrete extension operator
$\myR_{h, i}$ by solving the following auxiliary problem for given $\lambda
\in \Lambda$: Find $(\myR_{h, i} \lambda, p_{h, i}^\lambda, r_i)
\in V_{h, i} \times W_{h, i} \times S_{H, i}$ such that
\begin{subequations} \label{eq: R_h problem}
\begin{align}
	(K^{-1} \myR_{h, i} \lambda, \bm{v}_{h, i}^0)_{\Omega_i}
	- (\nabla \cdot \bm{v}_{h, i}^0, p_{h, i}^\lambda)_{\Omega_i}
	&= 0
	, && \forall \bm{v}_{h, i}^0 \in V_{h, i}^0, \label{Rh-eq1}
	\\
	(\nabla \cdot \myR_{h, i} \lambda, w_{h, i})_{\Omega_i}
	- (r_i, w_{h, i})_{\Omega_i}
	&= 0
	, && \forall w_{h, i} \in W_{h, i}, \label{Rh-eq2}
	\\
	(p_{h, i}^\lambda, s_i)_{\Omega_i}
	&= 0
	, && \forall s_i \in S_{H, i}, \label{Rh-eq3}
        \\
	\bm{\nu}_i \cdot \myR_{h, i} \lambda
	&= \psi_{h, i}, && \text{ on } \Gamma_i. \label{Rh-bc}
\end{align}
\end{subequations}
We note that \eqref{Rh-bc} is an essential boundary condition and that,
for subdomains adjacent to $\partial \Omega$, the boundary condition
$p_i^\lambda = 0$ on $\partial \Omega_i \setminus \Gamma_i$ is natural and
has been incorporated in \eqref{Rh-eq1}.
We emphasize that the definition of $\psi_{h, i}$ depends on the
choice of projection operator from the previous subsections. In
particular, for $\myR_{h, i} = \myR_{h, i}^\sharp$, we have $\psi_{h,
  i} := \psi_{h, i}^\sharp$ from \eqref{eq: to-projection problem} and
we set $\psi_{h, i} := \psi_{h, i}^\flat$ from \eqref{eq:
  from-projection problem} for $\myR_{h, i} = \myR_{h, i}^\flat$.

\begin{lemma} \label{lem: R_h is well-posed}
Problem \eqref{eq: R_h problem} admits a unique solution with
\begin{align}\label{Rh-stab}
		\nabla \cdot \myR_{h, i} \lambda &= \overline{\lambda}_i &
		&\text{and} &
		\| \myR_{h, i} \lambda \|_{\div, \Omega_i}
		&\lesssim
		\| \psi_{h, i} \|_{\Gamma_i}
		\lesssim
		\| \lambda \|_{\Gamma_i}.
	\end{align}
\end{lemma}
\begin{proof}
We first show that
$$
r_i = \overline{\lambda}_i := \begin{cases}
| \Omega_i |^{-1} (\lambda_i, 1)_{\Gamma_i}, &i\in I_{int}, \\
0, &i \notin I_{int}.
\end{cases}
$$
If $i \notin I_{int}$, this follows by the
definition of $S_H$ since we then have $r_i = \overline{\lambda}_i =
0$. If $i \in I_{int}$, we set $w_{h, i} = 1$ in \eqref{Rh-eq2}:
\begin{align*}
		(r_i - \overline{\lambda}_i, 1)_{\Omega_i} &=
		(\nabla \cdot \myR_{h, i} \lambda, 1)_{\Omega_i}
		- (\overline{\lambda}_i, 1)_{\Omega_i}
		=
		(\psi_{h, i} - \lambda_i, 1)_{\Gamma_i}
		= 0,
		& \forall i \in I_{int}.
\end{align*}
The final equality follows for the two variants due to Lemmas~\ref{lem
  psi sharp preserves mean} and \ref{lem psi flat preserves mean}. Now
\eqref{Rh-eq2} implies that $\nabla \cdot \myR_{h, i} \lambda =
\overline{\lambda}_i$.

Since this is a square, finite-dimensional linear system, uniqueness
implies existence. Thus, we set $\lambda = 0$ and note that $r_i =
\overline{\lambda}_i = 0$. In addition, $\psi_{h,i} = 0$,
thus $\myR_{h, i} \lambda \in V_{h, i}^0$. Setting test functions
$(\myR_{h, i} \lambda, p_{h, i}^\lambda)$ in the first two equations and
summing them gives $\myR_{h, i} \lambda = 0$.  Finally, we use
\eqref{eq: div V = W} to derive that $W_{h, i} = \nabla \cdot V_{h,
  i}^0 \oplus S_{H, i}$ which implies $p_{h, i}^\lambda = 0$, using
\eqref{Rh-eq1} and \eqref{Rh-eq3}.

We continue with the stability estimate by first obtaining a bound on
the auxiliary variable $p_{h, i}^\lambda$. Recall that the discrete
pair $V_{h, i} \times W_{h, i}$ is stable, see \eqref{local-inf-sup},
and note that $p_{h, i}^\lambda$ has zero mean for $i \in I_{int}$.
Thus, there exists $\bm{v}_{h, p,i}^0 \in V_{h, i}^0$ such that
\begin{align*}
  \nabla \cdot \bm{v}_{h, p,i}^0 = p_{h, i}^\lambda \ \text{ in } \Omega_i,
  \quad \| \bm{v}_{h, p,i}^0 \|_{\div, \Omega_i} \lesssim \| p_{h, i}^\lambda \|_{\Omega_i}.
\end{align*}
Using $\bm{v}_{h, p,i}^0$ as a test function in \eqref{Rh-eq1}, we derive
	\begin{align*}
		\| p_{h, i}^\lambda \|_{\Omega_i}^2 =
		(K^{-1} \myR_{h, i} \lambda, \bm{v}_{h, p,i}^0)_{\Omega_i}
		\le
		\| K^{-1}\myR_{h, i} \lambda \|_{\Omega_i}
		\| \bm{v}_{h, p,i}^0 \|_{\Omega_i}
		\lesssim
		\| \myR_{h, i} \lambda \|_{\Omega_i}
		\| p_{h, i}^\lambda \|_{\Omega_i},
	\end{align*}
implying
\begin{align} \label{eq: p^lambda bound}
		\| p_{h, i}^\lambda \|_{\Omega_i}
		\lesssim
		\| \myR_{h, i} \lambda \|_{\Omega_i}.
	\end{align}

Second, we note that $\nabla\cdot \myR_{h, i} \lambda = 0$ for all $i
\notin I_{int}$ since $\overline{\lambda}_i = 0$. For the remaining
indexes, i.e. $i \in I_{int}$, we derive:
\begin{subequations} \label{eq: subeqs R_h well-posed}
	\begin{align}
		\| \nabla \cdot \myR_{h, i} \lambda \|_{\Omega_i}
		&=
		\| \overline{\lambda}_i \|_{\Omega_i}
		=
		| \Omega_i |^{-\frac12} (\lambda, 1)_{\Gamma_i}
		=
		| \Omega_i |^{-\frac12} (\psi_{h, i}, 1)_{\Gamma_i}
		\lesssim
		\| \psi_{h, i} \|_{\Gamma_i}.
	\end{align}

Third, we introduce the discrete $H(\div, \Omega_i)$--extension
operator from \cite[Sec. 4.1.2]{quarteroni1999domain}, and denote it
by $\myR_{h, i}^\star: V_{h, i}^\Gamma \to V_{h, i}$. This extension
has the properties:
\begin{align*}
\bm{\nu}_i \cdot \myR_{h, i}^\star \psi_{h, i} = \psi_{h, i} \text{ on }\Gamma_i, \ \
\bm{\nu}_i \cdot \myR_{h, i}^\star \psi_{h, i} = 0
\text{ on }\partial \Omega_i \setminus \Gamma_i, \ \
\| \myR_{h, i}^\star \psi_{h, i} \|_{\div, \Omega_i} \lesssim \| \psi_{h, i} \|_{\Gamma_i}.
\end{align*}
The next step is to set the test functions $\bm{v}_{h, i}^0 = \myR_{h,
  i} \lambda - \myR_{h, i}^\star \psi_{h, i}$, $w_{h, i} = p_{h,
  i}^\lambda$ and $s_i = r_i$ in \eqref{eq: R_h problem}. After
summation of the equations, we have
\begin{align*}
(K^{-1} \myR_{h, i} \lambda, \myR_{h, i} \lambda - \myR_{h, i}^\star \psi_{h, i})_{\Omega_i}
+ (\nabla \cdot \myR_{h, i}^\star \psi_{h, i}, p_{h, i}^\lambda)_{\Omega_i} = 0.
\end{align*}
Using bound \eqref{eq: p^lambda bound} and the continuity bound for
$\| \myR_{h, i}^\star \psi_{h, i} \|_{\div, \Omega_i}$, we obtain
\begin{align*}
\| \myR_{h, i} \lambda \|_{\Omega_i}^2 \lesssim
(\| \myR_{h, i} \lambda \|_{\Omega_i} + \| p_{h, i}^\lambda \|_{\Omega_i})
\|\myR_{h, i}^\star \psi_{h, i} \|_{\div, \Omega_i}
\lesssim \| \myR_{h, i} \lambda \|_{\Omega_i} \| \psi_{h, i} \|_{\Gamma_i},
\end{align*}
which implies
\begin{align}
		\| \myR_{h, i} \lambda \|_{\Omega_i}
		&\lesssim
		\| \psi_{h, i} \|_{\Gamma_i}.
\end{align}

Finally, recall that $\psi_{h, i}^\sharp = \mathcal{Q}_{h, i}^\sharp
\lambda$ and $\psi_{h, i}^\flat = \mathcal{Q}_{h, i}^\flat \lambda$,
i.e. both variants are generated using an $L^2$-projection. This
provides the bound:
\begin{align}
\| \psi_{h, i} \|_{\Gamma_i} \lesssim
\| \lambda \|_{\Gamma_i}.
\end{align}
\end{subequations}
Collecting \eqref{eq: subeqs R_h well-posed} proves the stability estimate.
\end{proof}

\section{Well posedness}
\label{sec:well-posed}

In this section, we establish existence, uniqueness, and stability of
the solution to the discrete problem \eqref{eqs: discrete problem 1}.

Let the bilinear forms $a$ and $b$ be defined as
\begin{align} \label{eqs: definition a and b}
	a(\bm{u}_h, \bm{v}_h) := (K^{-1} \bm{u}_h, \bm{v}_h)_\Omega, \quad 
	b(\bm{u}_h, w_h) := \sum_i (\nabla \cdot \bm{u}_{h, i}, w_{h, i})_{\Omega_i}.
\end{align}
Problem \eqref{eqs: discrete problem 1} can then be reformulated as:
Find $\bm{u}_h \in V_h$ and $p_h \in W_h$ such that
\begin{subequations}  \label{eq: discrete problem}
	\begin{align}
		a(\bm{u}_h, \bm{v}_h)
		- b(\bm{v}_h, p_h)
		&= 0,
		& \forall \bm{v}_h &\in V_h, \label{eqs: discrete problem eq1} \\
		b(\bm{u}_h, w_h)
		&= (f, w_h)_\Omega,
		& \forall w_h &\in W_h.
	\end{align}
\end{subequations}

In the next lemma we establish several properties of the bilinear forms
that will be used in the well posedness proof.

\begin{lemma} \label{lem: Brezzi conditions}
The bilinear forms $a(\cdot, \cdot)$ and $b(\cdot, \cdot)$ satisfy the
following bounds:
\begin{subequations} \label{ineqs: Brezzi conditions}
\begin{align}
	\forall & \bm{u}_h, \bm{v}_h \in V_h:
	&  a(\bm{u}_h, \bm{v}_h) &\lesssim \| \bm{u}_h \|_V \| \bm{v}_h \|_V.
	\label{ineq: a_cont}\\
	\forall & \bm{v}_h \in V_h \text{ and } w_h \in W_h:
	&  b(\bm{v}_h, w_h) &\lesssim \| \bm{v}_h \|_V \| w_h \|_W.
	\label{ineq: b_cont}\\
	\forall & \bm{v}_h \in V_h \text{ with } b(\bm{v}_h, w_h) = 0 \ \forall w_h \in W_h:
	&  a(\bm{v}_h, \bm{v}_h) &\gtrsim \| \bm{v}_h \|_V^2.
	\label{ineq: a_coercive}\\
	\forall & w_h \in W_h, \ \exists 0 \ne \bm{v}_h \in V_h \text{ such that}:
	&  b(\bm{v}_h, w_h) &\gtrsim \| \bm{v}_h \|_V \| w_h \|_W.
	\label{ineq: b_infsup}
\end{align}
\end{subequations}
\end{lemma}
\begin{proof}
Bounds \eqref{ineq: a_cont} and \eqref{ineq: b_cont} describe the
continuity of the bilinear forms. These follow directly from the
Cauchy-Schwarz inequality and the boundedness of $K$,
c.f. \eqref{K-spd}. Bound \eqref{ineq: a_coercive} concerns coercivity. Recall that
$\nabla \cdot V_{h, i} \subseteq W_{h, i}$ from \eqref{eq: div V = W}
and that $V_h \subseteq \bigoplus_i V_{h, i}$. In turn, the assumption
$b(\bm{v}_h, w_h) = 0$ for all $w_h \in W_h$ implies that $\nabla
\cdot \bm{v}_{h, i} = 0$ for all $i$. Using this in combination with
\eqref{K-spd} gives
\begin{align*}
		a(\bm{v}_h, \bm{v}_h)
		= \| K^{-1/2} \bm{v}_h \|_{\Omega}^2
		\gtrsim \| \bm{v}_h \|_{\Omega}^2
		= \| \bm{v}_h \|_V^2.
\end{align*}
Finally, inequality \eqref{ineq: b_infsup} describes the discrete inf-sup
condition. Let $w_h \in W_h$ be given.
Consider a global divergence problem on $\Omega$:
\begin{align}\label{global-div}
  \nabla \cdot \bm{v}^w &= w_h \text{ in } \Omega, \quad \bm{v}^w = \bm{g} \text{ on }
  \partial\Omega,
\end{align}
where $\bm{g} \in H^{\frac12}(\partial\Omega)$ is such that
$(\bm{\nu}\cdot\bm{g},1)_{\partial\Omega} = (w_h,1)_{\Omega}$ and
$\|\bm{g}\|_{\frac12,\partial\Omega} \lesssim \|w_h\|_{\Omega}$. This problem
has a solution $\bm{v}^w \in (H^1(\Omega))^n$ satisfying \cite{Galdi}:
\begin{align*}
  \| \bm{v}^w \|_{1, \Omega} \lesssim \| w_h \|_\Omega + \|\bm{g}\|_{\frac12,\partial\Omega}
  \lesssim \|w_h\|_{\Omega}.
\end{align*}
Let $\mu_h \in \Lambda_h$ be defined on each interface $\Gamma_{ij}$ as the mean
value of $\bm{\nu} \cdot \bm{v}^w$. We have
$$
\|\mu_h\|_{\Gamma} \lesssim \sum_{i \in I_\Omega} \| \mu_{h, i} \|_{\Gamma_i}
\lesssim \sum_{i \in I_\Omega}\| \bm{\nu}_i \cdot \bm{v}^w \|_{\Gamma_i}
\lesssim \sum_{i \in I_\Omega}\|\bm{v}^w \|_{1,\Omega_i}
\lesssim \| w_h \|_{\Omega}.
$$
Moreover, on each interior subdomain $\Omega_i$,
i.e., with $\Gamma_i = \partial\Omega_i$, we have that
\begin{align*} 
		(\mu_{h, i}, 1)_{\partial\Omega_i}
		= (\bm{\nu}_i \cdot \bm{v}^w, 1)_{\partial\Omega_i}
		= (\nabla \cdot \bm{v}^w, 1)_{\Omega_i}
		= (w_{h, i}, 1)_{\Omega_i}.
\end{align*}
Consider the extension $\myR_{h, i} \mu_h$ and note that on each interior subdomain
$\Omega_i$, 
\begin{align*}
(w_{h, i} - \nabla \cdot \myR_{h, i} \mu_h, 1)_{\Omega_i} & =
  (w_{h, i}, 1)_{\Omega_i} - (\bm{\nu}_i \cdot \myR_{h, i} \mu_h, 1)_{\partial \Omega_i} \\
  & = (w_{h, i}, 1)_{\Omega_i} - (\mu_{h, i}, 1)_{\partial\Omega_i} = 0.
\end{align*}
Then, the local discrete inf-sup condition \eqref{local-inf-sup} implies that
in each $\Omega_i$ there exists $\bm{v}_{h, i}^0 \in V_{h, i}^0$ such that
\begin{align*}
\nabla \cdot \bm{v}_{h, i}^0 = w_{h, i} - \nabla \cdot \myR_{h, i} \mu_h \quad
\text{ in }\Omega_i,
\end{align*}
and, using Lemma~\ref{lem: R_h is well-posed},
\begin{align*}
\sum_{i \in I_\Omega}\| \bm{v}_{h, i}^0 \|_{\div, \Omega_i}
& \lesssim \sum_{i \in I_\Omega} \| w_{h, i} - \nabla \cdot \myR_{h, i} \mu_h \|_{\Omega_i}
\le \sum_{i \in I_\Omega}
\left(\| w_{h, i} \|_{\Omega_i} + \| \nabla \cdot \myR_{h, i} \mu_h \|_{\Omega_i} \right)\\
&\lesssim \sum_{i \in I_\Omega} \left(\| w_{h, i} \|_{\Omega_i} + \| \mu_{h, i} \|_{\Gamma_i}\right)
\lesssim \| w_{h} \|_{\Omega}.
\end{align*}
The final step is to define $\bm{v}_h^0 \in V_h^0$ such that
$\bm{v}_h^0|_{\Omega_i}:= \bm{v}_{h, i}^0$,
set $\bm{v}_h := \bm{v}_h^0 + \myR_h \mu_h \in V_h$, and note that
\begin{subequations} \label{eq: result inf-sup}
\begin{align}
& b(\bm{v}_h, w_h)
= (\nabla \cdot (\bm{v}_h^0 + \myR_h \mu_h), w_h)_\Omega = \| w_h \|_W^2, \\
& \| \bm{v}_h \|_V \le \| \bm{v}_h^0 \|_V + \| \myR_h \mu_h \|_V
\lesssim \| \bm{v}_h^0 \|_V + \| \mu_h \|_{\Gamma} \lesssim \| w_h \|_W.
\end{align}
\end{subequations}
Combining equations \eqref{eq: result inf-sup} yields \eqref{ineq: b_infsup}.
\end{proof}

\begin{corollary} \label{cor: inf-sup S_H}
	The following inf-sup condition holds for the spaces $\Lambda_h \times S_H$:
	\begin{align*}
	\forall & s_H \in S_H, \ \exists 0 \ne \mu_h \in \Lambda_h \text{ such that }
	b(\myR_h \mu_h, s_H) \gtrsim \| \mu_h \|_\Lambda \| s_H \|_W.
			\label{ineq: b_infsup on S_H}
	\end{align*}
\end{corollary}
\begin{proof}
  Setting $w_h := s_H \in S_H \subseteq W_h$ in the above proof of \eqref{ineq: b_infsup}
leads to a pair $(\bm{v}_h^0, \mu_h)$ with $\bm{v}_h^0 = 0$, $\| \mu_h \|_\Lambda \lesssim \| s_H \|_W$, and $b(\myR_h \mu_h, s_H) = \| s_H \|_W^2$.
\end{proof}

We are now ready to establish the well posedness of the flux-mortar MFE method.

\begin{theorem} \label{thm: well-posedness discrete}
Problem~\eqref{eq: discrete problem} with $\myR_h = \myR_h^\flat$
admits a unique solution $(\bm{u}_h, p_h) \in V_h^\flat \times W_h$. If
\eqref{eq: flat mortar condition} holds, then \eqref{eq: discrete
  problem} with $\myR_h = \myR_h^\sharp$ has a unique solution
$(\bm{u}_h, p_h) \in V_h^\sharp \times W_h$. In both cases the solution
satisfies
\begin{align} \label{eq: stability estimate}
		\| \bm{u}_h \|_V + \| p_h \|_W \lesssim \| f \|_{\Omega}.
\end{align}
Moreover, if \eqref{eq: flat mortar condition} holds for $\myR_h = \myR_h^\flat$ and
if \eqref{eq: sharp mortar condition} holds for $\myR_h = \myR_h^\sharp$, then
the mortar solution $\lambda_h \in \Lambda_h$ is unique and satisfies
\begin{align}
\| \lambda_h \|_\Lambda \lesssim h^{-1/2} \| \bm{u}_h \|_V.
\end{align}
\end{theorem}
\begin{proof}
Continuity of the right-hand side of \eqref{eqs: discrete problem eq1} follows
from the Cauchy-Schwarz inequality. Together with the four
inequalities from Lemma~\ref{lem: Brezzi conditions}, we have
sufficient conditions to invoke the standard saddle point theory
\cite{boffi2013mixed} and obtain \eqref{eq: stability estimate}.

It remains to show the bound on $\lambda_h$ and therewith its uniqueness. 
For that, we use \eqref{eq: flat mortar condition} if $\myR_h = \myR_h^\flat$ and \eqref{eq: sharp mortar condition} if $\myR_h = \myR_h^\sharp$, combined with a discrete trace inequality:
	\begin{align*}
		\| \lambda_h \|_\Lambda =
		\| \lambda_h \|_\Gamma \lesssim
		\sum_i \| \mathcal{Q}_{h, i} \lambda_h \|_{\Gamma_i} \lesssim
		\sum_i h^{-1/2} \| \bm{u}_{h, i} \|_{\Omega_i} \lesssim
		h^{-1/2} \| \bm{u}_h \|_{V}.
	\end{align*}
\end{proof}

\section{A priori error analysis}
\label{sec:a_priori_analysis}

In this section, we present the error analysis of the discrete problem
\eqref{eqs: discrete problem 1}. Section~\ref{sub:interpolation_operators}
introduces the interpolation operators that form an important
tool in deriving the a priori error estimates in
Section~\ref{sub:error_estimates}.

\subsection{Interpolation operators}
\label{sub:interpolation_operators}

One of the main tools in deriving the error estimates is the
construction of an appropriate interpolant associated with the
discrete space $V_h$. The building blocks in our construction are the
canonical interpolation operators associated with the subdomain finite
element spaces $V_{h, i}$, namely $\Pi_i^V: V_i \cap
(H^\epsilon(\Omega_i))^n \to V_{h, i}$ with $\epsilon > 0$, with
the properties
\begin{align} 
(\nabla \cdot (\bm{v}_{i} - \Pi_i^V \bm{v}_{i}), w_{h, i})_{\Omega_i} &= 0, &
  \forall w_{h, i} &\in W_{h, i}, \label{eq: commutativity}\\
(\bnu_i\cdot(\bm{v}_{i} - \Pi_i^V \bm{v}_{i}), w_{h, i})_{\partial \Omega_i} &= 0, &
  \forall w_{h, i} &\in W_{h, i}. \label{eq: Pi-trace}
\end{align}
In addition, let $\Pi_i^W: L^2(\Omega_i) \to W_{h, i}$ and
$\Pi_{ij}^\Lambda: L^2(\Gamma_{ij}) \to \Lambda_{h, ij}$ denote the
$L^2$-projection operators onto $W_{h, i}$ and $\Lambda_{h, ij}$,
respectively. Together with the projection $\mathcal{Q}_{h, i}^\flat$
onto $V_{h, i}^\Gamma$ introduced earlier, we recall the
approximation properties \cite{boffi2013mixed}:
\begin{subequations} \label{eqs: approx props}
	\begin{align}
		\| \bm{v} - \Pi_i^V \bm{v} \|_{\Omega_i}
		&\lesssim h^{r_v} \| \bm{v} \|_{r_v, \Omega_i},
		& 0 &< r_v \le k_v + 1, \label{eq: approx prop v_i}\\
		\| \nabla \cdot (\bm{v} - \Pi_i^V \bm{v}) \|_{\Omega_i}
		&\lesssim h^{r_w} \| \nabla \cdot \bm{v} \|_{r_w, \Omega_i},
		& 0 &\le r_w \le k_w + 1,
		\label{eq: approx prop div v}\\
		\| w - \Pi_i^W w \|_{\Omega_i}
		&\lesssim h^{r_w} \| w \|_{r_w, \Omega_i},
		& 0 &\le r_w \le k_w + 1,
		\label{eq: approx prop w}\\
		\| \mu - \Pi_{ij}^\Lambda \mu \|_{\Gamma_{ij}}
		&\lesssim h_\Gamma^{r_\Lambda} \| \mu \|_{r_\Lambda, \Gamma_{ij}},
		& 0 &\le r_\Lambda \le k_\Lambda + 1,
		\label{eq: approx prop lambda}\\
		\| \mu - \mathcal{Q}_{h, i}^\flat \mu \|_{\Gamma_{ij}}
		& \lesssim
		h^{r_v} \| \mu \|_{r_v, \Gamma_{ij}}, &
		0 & \le r_v \le k_v + 1.
		\label{eq: approx prop Q}
	\end{align}
\end{subequations}
The constants $k_v$, $k_w$, and $k_\Lambda$ represent the polynomial order of the spaces $V_h$, $W_h$, and $\Lambda_h$, respectively, and $i, j \in I_\Omega$.
To exemplify, we present two choices of stable mixed finite element
pairs. For the pair of Raviart-Thomas of order $k_v$ and discontinuous
Lagrange elements of order $k_w$, we have $k_v = k_w$. On the other
hand, choosing the Brezzi-Douglas-Marini elements of order $k_v$ with
discontinuous Lagrange elements of polynomial order $k_w$, we obtain a
stable pair if $k_v = k_w + 1$. For more examples of stable finite
element pairs, we refer the reader to \cite{boffi2013mixed}

Let $\Pi^W: W \to W_h$ and $\Pi^\Lambda: \Lambda \to \Lambda_h$
be defined as the $L^2$-projections
$\Pi^W := \bigoplus_i \Pi_i^W$ and $\Pi^\Lambda:= \bigoplus_{i<j}\Pi^\Lambda_{ij}$.
The approximation properties of these
operators follow directly from \eqref{eqs: approx props}.

Next, we introduce the composite interpolant
$\Pi^V: \overline V \to V_h$, where
$\overline V = \{\bm{v} \in V: \bm{v}|_{\Omega_i} \in (H^\epsilon(\Omega_i))^n \text{ and }
  (\bm{\nu} \cdot
\bm{u})|_\Gamma \in \Lambda \}$.
Given $\bm{u} \in \overline V$ with normal trace $\lambda := (\bm{\nu} \cdot
\bm{u})|_\Gamma \in \Lambda$, we define $\Pi^V \bm{u} \in V_h$ as
\begin{subequations}\label{Pi-defn}
\begin{align}
  & \Pi_\flat^V \bm{u} := \myR_h^\flat \Pi^\Lambda\lambda
  + \bigoplus_i \Pi_i^V (\bm{u}_i - \myR_{h,i}^\flat \lambda)
= \myR_h^\flat (\Pi^\Lambda\lambda - \lambda) + \bigoplus_i \Pi_i^V \bm{u}_i, \label{Pi-flat}\\
  & \Pi_\sharp^V \bm{u} := \myR_h^\sharp \Pi^\Lambda\lambda
+ \bigoplus_i \Pi_i^V (\bm{u}_i - \myR_{h,i}^\flat \lambda)
= \Pi_\flat^V \bm{u} + \myR_h^\sharp \Pi^\Lambda\lambda - \myR_h^\flat \Pi^\Lambda\lambda. \label{Pi-sharp}
\end{align}
\end{subequations}
We note that, due to \eqref{psi-Qh-flat} and \eqref{Rh-bc},
$\bnu_i\cdot\myR_{h,i}^\flat \lambda = \mathcal{Q}_{h,i}^\flat\lambda$, which, combined
with \eqref{eq: Pi-trace}, implies
$ \bnu_i\cdot\Pi_i^V (\bm{u}_i - \myR_{h,i}^\flat \lambda) =
\mathcal{Q}_{h,i}^\flat\lambda - \mathcal{Q}_{h,i}^\flat\lambda = 0,
$
so \eqref{Pi-defn} gives $\Pi_\flat^V\bm{u} \in V_h^\flat$ and
$\Pi_\sharp^V\bm{u} \in V_h^\sharp$. In the following, the use of $\Pi^V$ indicates
that the result is valid for both choices. We emphasize that the
definitions of $\Pi^V\bm{u}$ and $\bm{u}_h$, combined with
\eqref{psi-Qh-sharp}, \eqref{psi-Qh-flat}, and \eqref{Rh-bc}, imply
\begin{equation}\label{trace-Pi-R-Q}
  \bnu_i\cdot \Pi^V \bm{u} = \bnu_i\cdot\myR_{h,i} \Pi^\Lambda\lambda = \mathcal{Q}_{h,i}\Pi^\Lambda\lambda, \quad \
  \bnu_i\cdot \bm{u}_h = \bnu_i\cdot\myR_{h,i} \lambda_h = \mathcal{Q}_{h,i}\lambda_h.
  \end{equation}

\begin{lemma} \label{lem: B-compatible}
The interpolation operator $\Pi^V$ has the property
\begin{align}\label{B-compatible}
b(\bm{u} - \Pi^V \bm{u}, w_h) = 0, \quad \forall \, w_h \in W_h.
\end{align}
\end{lemma}
\begin{proof}
In the case of $\Pi_\flat^V$, we first note that, due to \eqref{Rh-stab},
$\nabla\cdot\myR_{h,i}^\flat (\Pi^\Lambda\lambda - \lambda)
= \overline{\Pi^\Lambda\lambda_i} - \overline\lambda_i = 0$. Then
the statement of the lemma follows from
\eqref{eq: commutativity}. In the case of $\Pi_\sharp^V$, due to
\eqref{Rh-stab},
$\nabla\cdot(\myR_{h,i}^\sharp \Pi^\Lambda\lambda - \myR_{h,i}^\flat \Pi^\Lambda\lambda) = 0$,
and the result follows.
\end{proof}

We proceed with the approximation properties of the interpolants $\Pi_\flat^V$
  and $\Pi_\sharp^V$.
\begin{lemma} \label{lem: approximation prop Pi}
Assuming that $\bu$ is smooth enough and that \eqref{eq: flat mortar condition} holds in the case $\myR_h = \myR_h^\sharp$, then
\begin{subequations} \label{eq: approx prop v}
\begin{align}
\| \bm{u} - \Pi_\flat^V \bm{u} \|_V
&\lesssim h^{r_v} \sum_i \| \bm{u} \|_{r_v, \Omega_i}
+ h^{r_w} \sum_i \| \nabla \cdot \bm{u} \|_{r_w, \Omega_i}
+ h_\Gamma^{r_\Lambda} \sum_{i < j} \| \lambda \|_{r_\Lambda, \Gamma_{ij}}, \label{approx-Pi-flat} \\
\| \bm{u} - \Pi_\sharp^V \bm{u} \|_V
&\lesssim
h^{r_v} \sum_i \| \bm{u} \|_{r_v, \Omega_i}
+ h^{r_w} \sum_i \| \nabla \cdot \bm{u} \|_{r_w, \Omega_i}
+ h_\Gamma^{r_\Lambda} \sum_{i < j} \| \lambda \|_{r_\Lambda, \Gamma_{ij}} \label{approx-Pi-sharp} \\
& \qquad\qquad
+ h^{\tilde r_v} \sum_{i < j} \| \lambda \|_{\tilde r_v, \Gamma_{ij}}, \nonumber
\end{align}
\end{subequations}
for $0 < r_v \le k_v + 1$, $0 \le r_w \le k_w + 1$, $0 \le r_\Lambda \le k_\Lambda+1$, 
and $0 \le \tilde r_v \le k_v + 1$.
\end{lemma}
\begin{proof}
Using \eqref{Pi-flat}, bound \eqref{approx-Pi-flat} for $\Pi_\flat^V$
follows from \eqref{Rh-stab} and the approximation bounds \eqref{eq:
  approx prop v_i}, \eqref{eq: approx prop div v}, and \eqref{eq: approx
  prop lambda}. For $\Pi_\sharp^V$, using \eqref{Pi-sharp}, we need to bound $\|\myR_{h,
  i}^\sharp \Pi^\Lambda\lambda - \myR_{h, i}^\flat \Pi^\Lambda\lambda\|_{\div,
  \Omega_i}$. Since this is the extension that solves \eqref{eq: R_h
  problem} with boundary data
$\mathcal{Q}_{h, i}^\sharp \Pi^\Lambda\lambda - \mathcal{Q}_{h, i}^\flat \Pi^\Lambda\lambda$,
we have $\myR_{h, i}^\sharp \Pi^\Lambda\lambda - \myR_{h, i}^\flat
\Pi^\Lambda\lambda = \myR_{h, i}^\flat(\mathcal{Q}_{h, i}^\sharp \Pi^\Lambda\lambda - \mathcal{Q}_{h, i}^\flat \Pi^\Lambda\lambda)$.
We use this observation in combination with \eqref{Rh-stab} to obtain
the bound
\begin{align*}
& \| \myR_{h, i}^\sharp \Pi^\Lambda\lambda - \myR_{h, i}^\flat \Pi^\Lambda\lambda \|_{\div, \Omega_i}
\lesssim
\| \mathcal{Q}_{h, i}^\sharp \Pi^\Lambda\lambda - \mathcal{Q}_{h, i}^\flat \Pi^\Lambda\lambda \|_{\Gamma_i} \\
& \qquad \le \|\mathcal{Q}_{h, i}^\sharp \Pi^\Lambda\lambda - \mathcal{Q}_{h, i}^\sharp \lambda\|_{\Gamma_i}
+ \|\mathcal{Q}_{h, i}^\sharp \lambda - \mathcal{Q}_{h, i}^\flat \lambda\|_{\Gamma_i}
+ \|\mathcal{Q}_{h, i}^\flat \Pi^\Lambda\lambda - \mathcal{Q}_{h, i}^\flat \lambda\|_{\Gamma_i} \\
& \qquad \le 2\|\Pi^\Lambda\lambda - \lambda\|_{\Gamma_i} +
\|\mathcal{Q}_{h, i}^\sharp \lambda - \mathcal{Q}_{h, i}^\flat \lambda\|_{\Gamma_i}.
\end{align*}
The proof of \eqref{approx-Pi-sharp} is completed by using \eqref{eq: approx
  prop lambda} and 
Lemma~\ref{lem: similarity projections}, presented below.
\end{proof}

\begin{lemma} \cite[Lemma~3.2]{ACWY} \label{lem: similarity projections}
  If \eqref{eq: flat mortar condition} holds, then 
\begin{align}\label{eq:sim-proj}
\sum_{i < j} \| \mathcal{Q}_{h, i}^\sharp \lambda - \mathcal{Q}_{h, i}^\flat \lambda \|_{\Gamma_{ij}}
&\lesssim
h^{\tilde r_v} \sum_{i < j} \| \lambda \|_{\tilde r_v, \Gamma_{ij}}, &
0 &\le \tilde r_v \le k_v + 1.
\end{align}
\end{lemma}

\subsection{Error estimates}
\label{sub:error_estimates}
We now turn to the a priori error analysis. Using \eqref{eq: discrete problem} and 
\eqref{weak-model-2}, we obtain the error equations
\begin{subequations} \label{eqs: step1 errors}
\begin{align}
	a(\bm{u} - \bm{u}_h, \bm{v}_h)
	- b(\bm{v}_h, \Pi^W p - p_h)
	&=
	a(\bm{u}, \bm{v}_h)
	- b(\bm{v}_h, p) ,
	& \forall \, \bm{v}_h &\in V_h,\\
	b(\Pi^V \bm{u} - \bm{u}_h, w_h)
	&= 0,
	& \forall \, w_h &\in W_h, \label{eq: pi u - u_h div free}
\end{align}
\end{subequations}
where we used the orthogonality property of $\Pi^W$ in the first equation and
the b-compatibility \eqref{B-compatible} of $\Pi^V$
in the second equation. It is important to note that we did not
use the first equation in \eqref{weak-model}, which requires a test function in
$H(\div, \Omega)$. Instead, the expression on the right is the consistency error, which
will be controlled later with the use of \eqref{weak-model-1}. We set the test
functions as
\begin{align} \label{eq: test functions}
	\bm{v}_h := \Pi^V \bm{u} - \bm{u}_h - \delta \bm{v}_h^p, \quad
	w_h := \Pi^W p - p_h,
\end{align}
where, using the proof of \eqref{ineq: b_infsup} from
Lemma~\ref{lem: Brezzi conditions}, $\bm{v}_h^p \in V_h$ is constructed
to satisfy
\begin{align} \label{eq: properties v_h^p}
	b(\bm{v}_h^p, \Pi^W p - p_h) &= \| \Pi^W p - p_h \|_W^2, &
	\| \bm{v}_h^p \|_V &\lesssim \| \Pi^W p - p_h \|_W,
\end{align}
and $\delta > 0$ is a constant to be chosen later. Now
\eqref{eqs: step1 errors} leads to
\begin{align} \label{eq: definition terms}
	a(\Pi^V & \bm{u} - \bm{u}_h, \Pi^V \bm{u} - \bm{u}_h)
	+
	\delta \| \Pi^W p - p_h \|_W^2 \nonumber \\
	&=
	a(\Pi^V\bm{u} - \bm{u}, \Pi^V\bm{u} - \bm{u}_h)
	+
	a(\bm{u} - \bm{u}_h, \delta \bm{v}_h^p)
	+ \left[
	a(\bm{u}, \bm{v}_h) - b(\bm{v}_h, p)
	\right].
\end{align}
For the left-hand side of \eqref{eq: definition terms}, \eqref{eq: pi
  u - u_h div free} and \eqref{ineq: a_coercive} imply
\begin{subequations} \label{eqs: component bounds}
\begin{align}
	\| \Pi^V \bm{u} - \bm{u}_h \|_V^2
	\lesssim a(\Pi^V \bm{u} - \bm{u}_h, \Pi^V \bm{u} - \bm{u}_h).
\end{align}
For first term on the right in \eqref{eq: definition terms},
using \eqref{ineq: a_cont} and Young's inequality, we have
\begin{align}
a(\Pi^V\bm{u} - \bm{u}, \Pi^V\bm{u} - \bm{u}_h)
\lesssim
	\frac{1}{2\epsilon_1} \| \Pi^V\bm{u} - \bm{u} \|_V^2
	+ \frac{\epsilon_1}{2} \| \Pi^V\bm{u} - \bm{u}_h \|_V^2,
\end{align}
with $\epsilon_1 > 0$ to be determined later. Similarly, for the
second term on the right in \eqref{eq: definition terms}, using
$\epsilon_2 > 0$ and the bound on $\bm{v}_h^p$ from \eqref{eq:
  properties v_h^p}, we obtain
\begin{align}
	a(\bm{u} - \bm{u}_h, \delta \bm{v}_h^p)
	&\lesssim
	(\| \Pi^V\bm{u} - \bm{u} \|_V
	+ \| \Pi^V\bm{u} - \bm{u}_h \|_V )
	\| \delta \bm{v}_h^p \|_V \nonumber\\
	&\lesssim
	\frac{1}{2} \| \Pi^V\bm{u} - \bm{u} \|_V^2
	+ \frac{\epsilon_2}{2} \| \Pi^V\bm{u} - \bm{u}_h \|_V^2
	+ \left(\frac{1}{2} + \frac{1}{2 \epsilon_2}\right) \delta^2
	\| \Pi^W p - p_h \|_W^2.
\end{align}
Finally, for the last term in \eqref{eq: definition terms}
we introduce the consistency error
\begin{align}\label{consist-error}
	\mathcal{E}_c :=
	\sup_{\tilde{\bm{v}}_h \in V_h} \frac{a(\bm{u}, \tilde{\bm{v}}_h) - b(\tilde{\bm{v}}_h, p)}{\| \tilde{\bm{v}}_h \|_V}. 
\end{align}
Using the properties \eqref{eq: properties v_h^p} and Young's
inequality with $\epsilon_3 > 0$, we derive:
	\begin{align}
		a(\bm{u}, \bm{v}_h) - b(\bm{v}_h, p) &\le
		\| \bm{v}_h \|_V \mathcal{E}_c \nonumber\\
		&\lesssim (\| \Pi^V \bm{u} - \bm{u}_h \|_V + \delta \| \Pi^W p - p_h \|_W ) \mathcal{E}_c \nonumber\\
	&\lesssim \frac{\epsilon_3}{2} \| \Pi^V \bm{u} - \bm{u}_h \|_V^2
	+ \frac12 \delta^2 \| \Pi^W p - p_h \|_W^2
	+ \left( \frac{1}{2\epsilon_3} + \frac12 \right) \mathcal{E}_c^2.
	\end{align}
\end{subequations}
Collecting \eqref{eqs: component bounds} and setting all $\epsilon_i$
sufficiently small, it follows that
\begin{align*}
	\| \Pi^V \bm{u} - \bm{u}_h \|_V^2
	+
	\delta \| \Pi^W p - p_h \|_W^2
	\lesssim
	\| \Pi^V \bm{u} - \bm{u} \|_V^2
	+
	\delta^2 \| \Pi^W p - p_h \|_W^2
	+
	\mathcal{E}_c^2.
\end{align*}
Subsequently, we set $\delta$ sufficiently small to obtain
\begin{align} \label{eq: pre-intermediate}
	\| \Pi^V \bm{u} - \bm{u}_h \|_V
	+
	\| \Pi^W p - p_h \|_W
	\lesssim
	\| \Pi^V \bm{u} - \bm{u} \|_V
	+
	\mathcal{E}_c,
\end{align}
which, combined with the triangle inequality, implies
\begin{align}  \label{eq: intermediate}
	\| \bm{u} - \bm{u}_h \|_V
	+
	\| p - p_h \|_W
	\lesssim
	\| \Pi^V \bm{u} - \bm{u} \|_V
	+
	\| \Pi^W p - p \|_W
	+
	\mathcal{E}_c.
\end{align}
The next step is to derive a bound on the consistency error
$\mathcal{E}_c$. For that, we recall its definition
\eqref{consist-error} and apply integration by parts on each
$\Omega_i$ with $p = 0$ on $\partial \Omega$:
\begin{align} \label{eq: nonconformity term}
  \mathcal{E}_c
  	&= \sup_{\tilde{\bm{v}}_h \in V_h} \| \tilde{\bm{v}}_h \|_V^{-1}
        \Big((K^{-1} \bm{u}, \tilde{\bm{v}}_h)_\Omega
	- \sum_i (p, \nabla \cdot \tilde{\bm{v}}_h)_{\Omega_i}\Big) \nonumber\\
	&=
	\sup_{\tilde{\bm{v}}_h \in V_h} \| \tilde{\bm{v}}_h \|_V^{-1}
	\sum_i -(p, \bm{\nu}_i \cdot \tilde{\bm{v}}_{h, i})_{\Gamma_i}.
\end{align}
In the last equality we used that $K^{-1} \bm{u} = - \nabla p$, which follows
from the weak formulation $\eqref{weak-model-1}$ using integration by parts.

We continue the derivation using arguments that rely on the choice of
extension operator, as outlined in the following two subsections.

\subsubsection{Consistency error using \texorpdfstring{$\myR_h^\sharp$}{R sharp}}
\label{ssub: err_projection_sharp}

For this choice of extension operator,
c.f. Section~\ref{sub: disc projection_to_the_mortar_space}, we use the weak continuity
from Lemma~\ref{lem psi sharp preserves mean} to bound the consistency
error \eqref{eq: nonconformity term}.  Let the discrete subspace
consisting of continuous mortar functions be denoted by $\Lambda_{h,
  c} \subset \Lambda_h$. Next, let $\Pi_c^\Lambda: H^1(\Gamma) \to
\Lambda_{h, c}$ be the Scott-Zhang interpolant \cite{Scott-Zhang} into $\Lambda_{h,
  c}$. This interpolant has the approximation property
\begin{align} \label{eq: approx Scott Zhang}
	\| p - \Pi_c^\Lambda p \|_{s_\Lambda, \Gamma}
	\lesssim h_\Gamma^{r_\Lambda - s_\Lambda} \| p \|_{r_\Lambda, \Gamma}, \quad
	0 \le r_\Lambda \le k_\Lambda + 1, \,\,
        0 \le s_\Lambda \le \min\{r_\Lambda,1\}.
\end{align}
Importantly, the Scott-Zhang interpolant preserves traces on $\partial
\Gamma$. This allows us to extend the function $(I - \Pi_c^\Lambda) p$
continuously by zero on $\partial \Omega \setminus \Gamma$ and we
let $E(I - \Pi_c^\Lambda) p$ denote the extended function.

Recall that $\tilde{\bm{v}}_h \in V_h^\sharp$ is weakly continuous due to Lemma~\ref{lem psi sharp preserves mean}. Consequently, $\sum_i (\Pi_c^\Lambda p, \bm{\nu}_i \cdot \tilde{\bm{v}}_{h, i})_{\Gamma_i} = 0$ and we use this to derive:
\begin{align}
\sum_i (p, \bm{\nu}_i \cdot \tilde{\bm{v}}_{h, i})_{\Gamma_i}
&=
\sum_i
((I - \Pi_c^\Lambda) p, \bm{\nu}_i \cdot \tilde{\bm{v}}_{h, i})_{\Gamma_i}
     =
\sum_i
( E(I - \Pi_c^\Lambda) p, \bm{\nu}_i \cdot \tilde{\bm{v}}_{h, i})_{\partial \Omega_i} \nonumber\\
&\lesssim
\sum_i
\| E(I - \Pi_c^\Lambda) p \|_{\frac12, \partial \Omega_i}
\| \tilde{\bm{v}}_h \|_{\div, \Omega_i}
\lesssim
\| (I - \Pi_c^\Lambda) p \|_{\frac12, \Gamma}
\| \tilde{\bm{v}}_h \|_V,
\end{align}
where we used the normal trace inequality $\|\bm{\nu}_i \cdot
\tilde{\bm{v}}_{h, i}\|_{-\frac12,\partial \Omega_i} \lesssim \|
\tilde{\bm{v}}_h \|_{\div, \Omega_i}$ \cite{boffi2013mixed}.  This
gives a bound on the consistency error $\mathcal{E}_c$ from
\eqref{eq: nonconformity term}, so \eqref{eq: intermediate} implies
\begin{align*} \label{eq: bound sharp}
	\| \bm{u} - \bm{u}_h \|_V
	+
	\| p - p_h \|_W
	\lesssim& \
	\| \Pi_\sharp^V \bm{u} - \bm{u} \|_V
	+
	\| \Pi^W p - p \|_W
	+
	\| \Pi_c^\Lambda p - p \|_{\frac12, \Gamma}.
\end{align*}
This bound, combined with the approximation properties
\eqref{eqs: approx props}, \eqref{approx-Pi-flat},
and \eqref{eq: approx Scott Zhang},
leads us to the main result of this subsection, given by the following theorem.

\begin{theorem} \label{thm: Error estimate sharp}
  In the case of $\myR_h^\sharp$, if \eqref{eq: flat mortar condition}
  holds and assuming sufficient
  regularity of the solution, then
\begin{align*}
&\| \bm{u} - \bm{u}_h \|_V + \| p - p_h \|_W
\lesssim \ h^{k_v + 1} \left( \sum_i \| \bm{u} \|_{k_v + 1, \Omega_i}
+ \sum_{i < j} \| \lambda \|_{k_v + 1, \Gamma_{ij}} \right) \\
&+ h^{k_w + 1} \sum_i \left(\| \nabla \cdot \bm{u} \|_{k_w + 1, \Omega_i}
+ \| p \|_{k_w + 1, \Omega_i} \right) 
+ h_\Gamma^{k_\Lambda + 1} \sum_{i < j} \| \lambda \|_{k_\Lambda + 1, \Gamma_{ij}} 
+ h_\Gamma^{k_\Lambda + \frac12} \| p \|_{k_\Lambda + 1, \Gamma} 
.
\end{align*}
\end{theorem}

\subsubsection{Consistency error using \texorpdfstring{$\myR_h^\flat$}{R flat}}
\label{ssub: err_projection_flat}

For this choice of extension operator, c.f. 
Section~\ref{sub: disc_projection_from_the_mortar_space}, we require a
different strategy to bound the consistency error \eqref{eq:
  nonconformity term} since weak continuity of normal traces
in $V_h^\flat$ is not guaranteed in general. We note that
$\tilde{\bm{v}}_h \in V_h^\flat$ can be decomposed as 
$\tilde{\bm{v}}_h =: \tilde{\bm{v}}_h^0 + \myR_h^\flat \tilde{\mu}_h$,
with $(\tilde{\bm{v}}_h^0, \tilde{\mu}_h) \in V_h^0 \times \Lambda_h$.
Using that $\mathcal{Q}_{h, i}^\flat$ is the $L^2$-projection onto $V_{h,i}^\Gamma$
and that $p$ is single-valued on $\Gamma$, we derive:
\begin{align*}
	\sum_i
	(p, \bm{\nu}_i \cdot \tilde{\bm{v}}_h)_{\Gamma_i}
	& = \sum_i
	(p, \mathcal{Q}_{h, i}^\flat \tilde{\mu}_{h, i})_{\Gamma_i}
	= \sum_i
	(\mathcal{Q}_{h, i}^\flat p, \tilde{\mu}_{h, i})_{\Gamma_i} \\
	&= \sum_i
	(\mathcal{Q}_{h, i}^\flat p - p, \tilde{\mu}_{h, i})_{\Gamma_i}
	\le \sum_i
	\| \mathcal{Q}_{h, i}^\flat p - p \|_{\Gamma_i}
	\| \tilde{\mu}_{h, i} \|_{\Gamma_i}.
\end{align*}
We continue the bound using the mortar condition
\eqref{eq: flat mortar condition} and a discrete trace inequality:
\begin{align*}
\ldots & \lesssim \sum_i
	\| \mathcal{Q}_{h, i}^\flat p - p \|_{\Gamma_i}
	\| \mathcal{Q}_{h, i}^\flat \tilde{\mu}_{h, i} \|_{\Gamma_i}
	= \sum_i
	\| \mathcal{Q}_{h, i}^\flat p - p \|_{\Gamma_i}
	\| \bm{\nu}_i \cdot \tilde{\bm{v}}_{h, i} \|_{\Gamma_i}
	\\
	&\lesssim \sum_i
	\| \mathcal{Q}_{h, i}^\flat p - p \|_{\Gamma_i}
	\left( h^{-1/2} \| \tilde{\bm{v}}_{h, i} \|_{\Omega_i} \right)
	\lesssim \Big(h^{-1/2} \sum_i
		\| \mathcal{Q}_{h, i}^\flat p - p \|_{\Gamma_i} \Big)
	\| \tilde{\bm{v}}_h \|_V
\end{align*}
With this result, we bound $\mathcal{E}_c$ from \eqref{eq: nonconformity term}
and obtain from \eqref{eq: intermediate}:
\begin{align*} 
	\| \bm{u} - \bm{u}_h \|_V
	+
	\| p - p_h \|_W
	\lesssim& \
	\| \Pi_\flat^V \bm{u} - \bm{u}\|_V
	+
	\| \Pi^W p - p \|_W +
	h^{-1/2} \sum_i
	\| \mathcal{Q}_{h, i}^\flat p - p \|_{\Gamma_i},
\end{align*}
which, combined with \eqref{eqs: approx props} and \eqref{approx-Pi-flat},
results in the the following theorem.

\begin{theorem} \label{thm: Error estimate flat}
In the case of $\myR_h^\flat$, if \eqref{eq: flat mortar condition}
holds and the solution is sufficiently regular, then
\begin{align*}
\| \bm{u} - \bm{u}_h \|_V & + \| p - p_h \|_W
\lesssim \ h^{k_v + 1} \sum_i \| \bm{u} \|_{k_v + 1, \Omega_i}
+ h_\Gamma^{k_\Lambda+1}\sum_{i < j} \| \lambda \|_{k_\Lambda + 1, \Gamma_{ij}} 
\\
& 
+ h^{k_w + 1} \sum_i \left(\| \nabla \cdot \bm{u} \|_{k_w + 1, \Omega_i}
+ \| p \|_{k_w + 1, \Omega_i} \right) 
+ h^{k_v + \frac12} \sum_i \| p \|_{k_v + 1, \Gamma_i}
.
\end{align*}
\end{theorem}

\subsubsection{Comparison}
\label{ssub:comparison}

The previous two sections indicate that theoretically the choice of extension
operator affects the resulting discretization error.
Most importantly, the estimates from Theorems~\ref{thm:
  Error estimate sharp} and \ref{thm: Error estimate flat} differ 
in the suboptimal pressure term and thus a comparison of the choices
$\myR_h^\sharp$ and $\myR_h^\flat$ leads us to comparing the terms
\begin{align*}
 	&h_\Gamma^{k_\Lambda + 1/2} \| p \|_{k_\Lambda + 1, \Gamma} &
 	&\text{ versus }&
 	&h^{k_v + 1/2} \sum_i \| p \|_{k_v + 1, \Gamma_i}.
\end{align*}
It follows from these terms that both choices will lead to a
suboptimal convergence rate if $k_\Lambda = k_v$, i.e. if the
polynomial orders of $\Lambda_h$ and $V_h$ are equal. This loss is
inevitable in the case that the projection onto the trace spaces
($\flat$) is chosen. However, it can be remediated if the projection
is chosen onto the space of weakly continuous functions ($\sharp$) by
setting $k_\lambda > k_v$, i.e. we choose a higher-order mortar space
$\Lambda_h$ within the limit of the mortar condition \eqref{eq: flat
  mortar condition}. This
behavior is similar to the pressure-mortar method \cite{ACWY,APWY}. It
is important to note, however, that the convergence rates we observe
numerically are unaffected by this suboptimal term, as shown in
Section~\ref{sec:numerical_results}. Hence, increasing the polynomial
order of the mortar space may not be necessary in practice.

\subsubsection{The interface flux}
\label{ssub:error_estimate_of_the_mortar_variable}

The error estimates derived in the previous sections show convergence
of the subdomain variables $\bm{u}_h$ and $p_h$. However, convergence
of the mortar variable $\lambda_h$ itself is not guaranteed at this
point. We therefore devote this section to finding error estimates of
the mortar variable for both types of projection operators. The
results are presented in a general setting. However, we remind that
the discrete solution $(\bm{u}_h, p_h) \in V_h \times W_h$ implicitly
depends on the chosen the projection operator.

In the following lemma we consider two measures of the interface flux error,
comparing the true interface flux $\lambda$ to either the mortar flux $\lambda_h$
or to the normal trace of the velocity $\bm{u_h}$ on
$\Gamma$, $\bnu_i\cdot \bm{u}_h = \mathcal{Q}_{h,i}\lambda_h$.

\begin{lemma} \label{lem: mortar error}
  If \eqref{eq: flat mortar condition} holds for $\myR_h^\flat$
  and \eqref{eq: sharp mortar condition} holds additionally for $\myR_h^\sharp$, and the solution is sufficiently regular, then
\begin{align*}
\| \lambda - \lambda_h \|_{\Gamma}
&\lesssim h_\Gamma^{k_\Lambda + 1} \sum_{i < j} \| \lambda \|_{k_\Lambda + 1, \Gamma_{ij}}
+ h^{-1/2} (\| \Pi^V \bm{u} - \bm{u} \|_V + \mathcal{E}_c), \\
\sum_i\| \lambda - \mathcal{Q}_{h, i} \lambda_h \|_{\Gamma_i}
&\lesssim h_\Gamma^{k_\Lambda + 1} \sum_{i < j} \| \lambda \|_{k_\Lambda + 1, \Gamma_{ij}}
+ h^{-1/2} (\| \Pi^V \bm{u} - \bm{u} \|_V + \mathcal{E}_c) \\
& \qquad\qquad\qquad\qquad + h^{k_v + 1} \sum_{i < j} \| \lambda \|_{k_v + 1, \Gamma_{ij}}. 
\end{align*}
\end{lemma}
\begin{proof}
  We have
\begin{align*}
\| \lambda - \lambda_h \|_{\Gamma}
&\le
\| \lambda - \Pi^\Lambda \lambda \|_{\Gamma}
+ \| \Pi^\Lambda \lambda - \lambda_h \|_{\Gamma} \\
&\lesssim
\| \lambda - \Pi^\Lambda \lambda \|_{\Gamma}
+ \sum_i \| \mathcal{Q}_{h, i} (\Pi^\Lambda \lambda - \lambda_h) \|_{\Gamma_i} \\
& = \| \lambda - \Pi^\Lambda \lambda \|_{\Gamma}
+ \sum_i \| \bm{\nu}_i \cdot (\Pi^V \bm{u} - \bm{u}_h) \|_{\Gamma_i} \\
& \lesssim \| \lambda - \Pi^\Lambda \lambda \|_{\Gamma}
+ h^{-1/2}\| \Pi^V \bm{u} - \bm{u}_h \|_{V},
\end{align*}
where we used the mortar condition \eqref{eq: flat mortar condition}
or \eqref{eq: sharp mortar condition} corresponding to the choice of
projection operator, \eqref{trace-Pi-R-Q}, and a discrete trace
inequality. The approximation property \eqref{eq: approx prop lambda}
and inequality \eqref{eq: pre-intermediate} then give us the first bound.
The second bound follows from the triangle inequality,
$$
\| \lambda - \mathcal{Q}_{h,i} \lambda_h \|_{\Gamma_i} \le
\| \lambda - \mathcal{Q}_{h,i}\lambda\|_{\Gamma_i} + \|\mathcal{Q}_{h,i}(\lambda - \lambda_h)\|_{\Gamma_i}
\le \| \lambda - \mathcal{Q}_{h,i}\lambda\|_{\Gamma_i}
+ \|\lambda - \lambda_h\|_{\Gamma_i},
$$
and the use of the approximation property \eqref{eq: approx prop Q}.
\end{proof}

The estimates from Lemma~\ref{lem: mortar error}
can be further developed by invoking the
approximation properties \eqref{eqs: approx props} and bounding the
consistency error $\mathcal{E}_c$ as in Sections \ref{ssub: err_projection_sharp} and \ref{ssub: err_projection_flat}. In their presented form,
however, these results emphasize that a half order loss in convergence
of the mortar variable is expected compared to the velocity.

\section{Reduction to an interface problem}
\label{sec:reduction_to_an_interface_problem}

We continue by presenting an iterative solution method for the
flux-mortar method \eqref{eq: discrete problem}. For that, we note
that the decomposition \eqref{eq: definition V_h} of $V_h$ into
interior and interface degrees of freedom allows us to reformulate the
method as an equivalent problem only in the flux mortar variable
$\lambda_h$. We recall that the method \eqref{eq: discrete problem}
can be written equivalently in the domain decomposition form
\eqref{dd-formulation}. Equation \eqref{dd-2} enforces weak pressure
continuity on the interface and is the basis for the interface
problem.  In order to set up this reduced problem, we first solve two
subproblems that incorporate the source term $f$ and provide the
right-hand side for the problem. Next, the reduced problem is set up
and solved. Finally, a post-processing step is necessary to obtain the
full solution to the original problem \eqref{eq: discrete
  problem}. For notational brevity, we omit the subscript $h$ on all
functions in this section, keeping in mind that all functions are
discrete. In the solution process we will utilize a generic extension
$\tilde\myR_h\mu \in \bigoplus V_{h,i}$ such that $\tilde\myR_{h,i}\mu
= \mathcal{Q}_{h,i}\mu$ on $\Gamma_i$. In practice,
$\tilde\myR_{h,i}\mu$ can be simply chosen to have all degrees of
freedom not associated with $\Gamma_i$ equal to zero. Recall that
$V_h = \bigoplus(V_{h,i}^0 \oplus \myR_{h,i}\Lambda)$ with $\myR_{h,i}$
the discrete extension \eqref{eq: R_h problem}. Since
$\tilde\myR_{h,i}\mu = \myR_{h,i}\mu + \bm{v}_{\mu,i}^0$ for some
$\bm{v}_{\mu,i}^0 \in V_{h,i}^0$, the spaces $\bigoplus(V_{h,i}^0 \oplus \myR_{h,i}\Lambda)$
and $\bigoplus(V_{h,i}^0 \oplus \tilde\myR_{h,i}\Lambda)$ are the same.
We will also utilize the orthogonal decomposition
$\Lambda_h = \Lambda_h^0 \oplus \overline\Lambda_h$, where
\begin{align} \label{eq: def Lambda^0}
\Lambda_h^0 := \{ \mu \in \Lambda_h : \ b(\tilde\myR_h \mu, s) = 0, \ \forall s \in S_H \}.
\end{align}
Let $B:
\Lambda_h \to S_H$ be defines as: $\forall \, \mu \in \Lambda_h, (B
\mu, s)_\Omega := b(\tilde\myR_h\mu,s) \ \forall \, s \in S_H$.  We note
that $\ker B = \Lambda_h^0$. Using \cite[Proposition~7.4.1]{QV-NumPDE}, the inf-sup
condition from Corollary~\ref{cor: inf-sup S_H} implies that $B^T$ is an isomorphism
from $S_H$ to the polar set of $\ker B$, $\{g \in \Lambda_h:
(g,\mu)_\Omega = 0 \ \forall \, \mu \in \ker B\}$, which is exactly
$\overline\Lambda_h$. Equivalently, $B$ is an isomorphism from
$\overline\Lambda_h$ to $S_H$.

The first step aims to capture the mean influence of the source term
$f$ on each interior subdomain using this space $S_H$, c.f. \eqref{eq:
  definition S_Hi darcy}. We solve the following global coarse
problem: Find $\overline\lambda_f \in \overline\Lambda_h$ such that
\begin{align}
  b(\tilde\myR_h \overline\lambda_f, s) = (f, s)_\Omega,
  \quad \forall s &\in S_H. \label{preproblem}
\end{align}
This problem has the form $B\overline\lambda_f = \bar f$ in $S_H$, thus it
has a unique solution, since $B: \overline\Lambda_h \to S_H$ is an isomorphism.

Second, we use $\overline\lambda_f$ to solve independent, local subproblems to
capture the remaining influence of the source term $f$: Find
$(\bm{u}_f^0, p_f^0, r_f) \in V_h^0 \times W_h \times S_H$ such that
\begin{subequations}\label{pre-2}
\begin{align}
	a(\bm{u}_f^0, \bm{v}^0) - b(\bm{v}^0, p_f^0)
	&=
	- a(\tilde\myR_h \overline\lambda_f, \bm{v}^0)
	,
	& \forall \bm{v}^0 &\in V_h^0, \\
	b(\bm{u}_f^0, w) - (r_f, w)_{\Omega}
	&= - b(\tilde\myR_h \overline\lambda_f, w) + (f, w)_{\Omega} ,
	& \forall w &\in W_h,
	\\
	(p_f^0, s)_{\Omega}
	&= 0,
	& \forall s &\in S_H.
\end{align}
\end{subequations}
Here, we enforce $p_f^0 \perp S_H$ with the use of a Lagrange
multiplier $r_f$. The well posedness
of \eqref{pre-2} follows from the argument for solvability of the
discrete extension problem \eqref{eq: R_h problem} given in
Lemma~\ref{lem: R_h is well-posed}. We further note that, setting $w =
r_f \in S_H$ and using \eqref{preproblem} implies that $r_f =
0$. Therefore, the velocity $\bm{u}_f := \bm{u}_f^0 + \tilde\myR_h
\overline\lambda_f$ satisfies the mass conservation equation $b(\bm{u}_f, w) =
(f, w)_{\Omega}$ for all $w \in W_h$.

The next step is to satisfy the Darcy equation \eqref{eqs: discrete
  problem eq1} by updating $\bm{u}_f$ with a divergence-free function
$\myR_h \lambda^0$. This is done by solving the reduced interface problem:
Find $\lambda^0 \in \Lambda_h^0$ such that
\begin{align}\label{eqs: reduced problem}
a(\myR_h \lambda^0, \tilde\myR_h \mu^0) - b(\tilde\myR_h \mu^0, p^{\lambda^0}) &=
- a(\bm{u}_f, \tilde\myR_h \mu^0) + b(\tilde\myR_h \mu^0, p_f^0), &
\forall \mu^0 \in \Lambda_h^0,
\end{align}
in which the pair $(\myR_h \lambda^0, p^{\lambda^0})$ solves the
discrete extension problem \eqref{eq: R_h problem}. The solvability of
\eqref{eqs: reduced problem} is established in Lemma~\ref{lem: SPD}
below. We note that $\nabla\cdot \myR_{h, i} \lambda^0 \in S_{H,i}$,
see \eqref{Rh-stab}, therefore $\nabla\cdot \myR_{h, i} \lambda^0 = 0$,
since $\lambda^0 \in \Lambda_h^0$.

After solving the reduced problem, we require one more step to obtain
the correct mean pressure in each interior subdomain. Thus, we 
construct $\overline{p}_\lambda \in S_H$ such that
\begin{align} \label{eq: postproblem}
b(\tilde\myR_h \mu, \overline{p}_\lambda)
&= a(\bm{u}_f + \myR_h \lambda^0, \tilde\myR_h \mu) - b(\tilde\myR_h \mu, p^{\lambda^0} + p_f^0), &
\forall \mu \in \Lambda_h.
\end{align}
Note that this equation is trivial for $\mu \in \Lambda_h^0$ due to
\eqref{eqs: reduced problem}. Hence we restrict ourselves to
$\overline\mu \in \overline\Lambda_h$, which results in a coarse grid
problem of the form $B^T \overline{p}_\lambda = g$ in
$\overline\Lambda_h$. Since $B^T: S_H \to \overline\Lambda_h$ is an isomorphism,
the problem has a unique solution.

We now have all the necessary ingredients to construct:
\begin{align}\label{eqs: construct solution}
  \bm{u} := \bm{u}_f + \myR_h \lambda^0 = 
  \bm{u}_f^0 + \myR_h \lambda^0 + \tilde\myR_h\overline\lambda_f, \quad
  p := p_f^0 + p^{\lambda^0}
  + \overline{p}_\lambda.
\end{align}
It is elementary to check that $(\bm{u}, p) \in V_h \times W_h$ indeed
solves \eqref{eq: discrete problem}. The corresponding mortar flux is
$\lambda = \lambda^0 + \overline\lambda_f$. We next show the solvability of
\eqref{eqs: reduced problem}.
\begin{lemma} \label{lem: SPD}
The bilinear form of the reduced problem \eqref{eqs: reduced problem}, given by
\begin{align}\label{a-Gamma}
a_\Gamma(\lambda, \mu) := a(\myR_h \lambda, \tilde\myR_h \mu) - b(\tilde\myR_h \mu, p^{\lambda}),
\end{align}
is symmetric and positive definite in $\Lambda_h^0 \times \Lambda_h^0$.
\end{lemma}
\begin{proof}
Using that $\tilde\myR_h \mu = \myR_h \mu + \bm{v}_\mu^0$, we have
\begin{align*}
  a_\Gamma(\lambda, \mu) & = a(\myR_h \lambda, \myR_h \mu) - b(\myR_h \mu, p^{\lambda})
  + a(\myR_h \lambda, \bm{v}_\mu^0) - b(\bm{v}_\mu^0, p^{\lambda}) \\
  & = a(\myR_h \lambda, \myR_h \mu) - b(\myR_h \mu, p^{\lambda})
  = a(\myR_h \lambda, \myR_h \mu).
\end{align*}
Here we used \eqref{Rh-eq1} in the second equality. For the last
equality we used that $\nabla \cdot \myR_{h, i} \mu = 0$ in
$\Omega_i$, since $\mu \in \Lambda_h^0$, c.f. \eqref{Rh-eq2} stated
for $\mu$. We therefore conclude that $a_\Gamma(\lambda, \mu)$ is
symmetric and positive semidefinite in $\Lambda_h^0 \times \Lambda_h^0$. Moreover,
if $\myR_h \lambda = 0$, then $\mathcal{Q}_{h}\lambda = 0$, thus $\lambda = 0$,
due to \eqref{eq: flat mortar condition} or \eqref{eq: sharp mortar condition}. Hence
$a_\Gamma$ is positive definite.
\end{proof}

The main implication of Lemma~\ref{lem: SPD} is that the interface
problem \eqref{eqs: reduced problem} can be solved using iterative
methods such as the Conjugate Gradient (CG) method. An important
observation is that mass is conserved locally, by construction, even
if the iterative solver is terminated before
convergence. Specifically, the component $\bm{u}_f$ is computed a
priori and the divergence-free update defined by $\lambda^0$ only
improves the accuracy of the solution with respect to Darcy's law.

\begin{remark}
  The implementation of problems \eqref{preproblem} and \eqref{eq:
    postproblem} requires solving a system with the same coarse
  matrix. The same system also occurs in the computation of the
  projection onto $\Lambda_h^0$ required in \eqref{eqs: reduced
    problem}. We refer the reader to \cite{VasWangYot} for an algebraic
  formulation for solving a global saddle problem with singular
  subdomain problems of a type similar to \eqref{eq: discrete
    problem}, which is based on the FETI method \cite{Toselli-Widlund}. We note
  that the incorporation of the coarse problem results in convergence
  of the interface iterative solver that is independent of the
  subdomain size.
\end{remark}

\section{Generalization to saddle point problems}
\label{sec:generalization_to_elliptic_problems}

In this section, we extend the concepts developed for the Darcy model
to a more general setting and introduce the flux-mortar MFE method for
a wider class of saddle point problems. We apply the general theory
to the coupled Stokes-Darcy problem in Section~\ref{sub:coupling_stokes_flow}.

Given a pair of function spaces $\tilde{V}
\times \tilde{W}$ on $\Omega$ and a non-overlapping
decomposition $\Omega = \bigcup_{i \in I_\Omega} \Omega_i$, we consider the problem:
Find $(\bm{u}, p) \in \tilde V \times \tilde W$ such that
\begin{subequations} \label{eq: general form}
	\begin{align}
		\sum_i a_i(\bm{u}_i, \bm{v}_{i})
		- \sum_i b_i(\bm{v}_{i}, p_i)
		&= \sum_i (\bm{g}, \bm{v}_i)_{\Omega_i},
		& \forall \bm{v} &\in \tilde V, \\
		\sum_i b_i(\bm{u}_i, w_i)
		&= \sum_i (f, w_i)_{\Omega_i},
		& \forall w &\in \tilde W.
	\end{align}
\end{subequations}
We note that this formulation allows for both essential and natural,
homogeneous boundary conditions on $\partial \Omega$. Essential
boundary conditions are incorporated in the definition of $\tilde{V}
\times \tilde{W}$. Extensions to other boundary conditions can readily
be made.

Let $V_i$ and $W_i$ be the respective restrictions of $\tilde{V}$ and
$\tilde{W}$ to subdomain $\Omega_i$. The composite spaces are defined as
\begin{align} \label{eq: composite spaces 2}
V := \bigoplus_i V_i, \quad W := \bigoplus_i W_i,
\end{align}
endowed with the norms $\| \bm{v} \|_V := \sum_i \| \bm{v}_{i} \|_{V_i}$ and
$\| w \|_W := \sum_i \| w_i \|_{W_i}$.

For each $V_i$, the trace operator $\Tr_i$ onto $\Gamma_i$ is then
defined such that the following alternative characterization of
$\tilde{V}$ holds:
\begin{align} \label{eq: characterization V_tilde}
\tilde{V} = \left\{ \bm{v} \in V:\ 
	\Tr_i \bm{v}_i = \Tr_j \bm{v}_j \ \text{on each } \Gamma_{ij}
	\right\}.
\end{align}
These continuities allow us to define a single-valued
global trace operator $\Tr \tilde{V}$ on $\Gamma$ and
introduce the interface space $\Lambda \subseteq \Tr \tilde{V}$  with a
suitable norm $\| \cdot \|_\Lambda$. Moreover, we define the subspace $V_i^0
:= \{\bm{v}_{i}^0 \in V_i :\ \Tr_i \bm{v}_{i}^0 = 0 \}$ for each $i \in
I_\Omega$.

We next construct the discretization of \eqref{eq: general form}. For
each $i \in I_\Omega$, let $\Omega_{i, h}$ be the tessellation of
$\Omega_i$ on which we define $V_{h, i} \times W_{h, i} \subset V_i
\times W_i$ as a stable mixed finite element pair for the
corresponding subproblem. Let $V_{h, i}^0 = V_{h, i} \cap V_i^0$, let
$\Lambda_h \subset \Lambda$ be the discretization of the interface
space, and let $S_{H, i}$ be the following null-space:
\begin{align} \label{eq: definition S_H}
	S_{H, i} :=
	\{ w_i \in W_i :\
	b(\bm{v}_{h, i}^0, w_i) = 0, \
	\forall \bm{v}_{i}^0 \in V_i^0 \}.
\end{align}
This allows us to define the discrete extension operator
$\myR_{h, i}: \Lambda \to V_{h, i}$ as the solution to the following problem for a given
$\lambda \in \Lambda$: 
Find $(\myR_{h, i} \lambda, p_{h, i}^\lambda, r_i) \in V_{h, i} \times W_{h, i} \times S_{H, i}$ such that
\begin{subequations} \label{eqs: R problem general}
	\begin{align}
		a_i(\myR_{h, i} \lambda, \bm{v}_{h, i}^0)
		- b_i(\bm{v}_{h, i}^0, p_{h, i}^\lambda)
		&= 0,
		& \forall \bm{v}_{h, i}^0 &\in V_{h, i}^0, \label{Rh-Vh} \\
		b_i(\myR_{h, i} \lambda, w_{h, i})
		- (r_i, w_{h, i})_{\Omega_i}
		&= 0,
		& \forall w_{h, i} &\in W_{h, i}, \label{eq: R problem general W_h} \\
		(p_{h, i}^\lambda, s_i)_{\Omega_i}
		&= 0,
		& \forall s_i &\in S_{H, i}, \label{eq: R problem general S_H} \\
		\Tr_i \myR_{h, i} \lambda
		&= \mathcal{Q}_{h, i} \lambda
		, & \text{ on } &\Gamma_i. \label{Rh-BC-general}
	\end{align}
\end{subequations}
Here, $\mathcal{Q}_{h, i}: \Lambda \to \Tr_i V_{h, i}$ is a chosen projection operator that maps interface data to the trace space of $V_{h, i}$. 


\begin{remark} \label{rem: introduction of S_H}
	The introduction of the space $S_{H, i}$ serves two purposes. First, it ensures that the subproblem is solvable by enforcing the Lagrange multiplier $r_i$ to act as compatible data. Second, the final equation ensures that the auxiliary variable $p_{h, i}^\lambda$ is uniquely defined, i.e. orthogonal to $S_{H, i}$.
\end{remark}

In turn, we define the composite spaces $V_h$ and $W_h$ as in \eqref{eq: definition V_h}:
\begin{align} \label{eq: V decomposition}
	V_h &:= \bigoplus_i V_{h, i}^0 \oplus \myR_{h, i} \Lambda_h
	= V_h^0 \oplus \myR_h \Lambda_h, &
	W_h &:= \bigoplus_i W_{h, i}.
\end{align}

We are now ready to set up the discretization of 
problem \eqref{eq: general form}: Find $(\bm{u}_h, p_h) \in V_h \times W_h$ such that
\begin{subequations} \label{eq: general form-h}
	\begin{align}
		\sum_i a_i(\bm{u}_{h,i}, \bm{v}_{h, i})
		- \sum_i b_i(\bm{v}_{h, i}, p_{h,i})
		&= \sum_i (\bm{g}, \bm{v}_{h,i})_{\Omega_i},
		& \forall \bm{v}_h &\in V_h, \\
		\sum_i b_i(\bm{u}_{h,i}, w_{h,i})
		&= \sum_i (f, w_{h,i})_{\Omega_i},
		& \forall w_h &\in W_h.
	\end{align}
\end{subequations}

Let $\Pi^V$, $\Pi^W$, and $\Pi^\Lambda$ denote interpolants onto
$V_h$, $W_h$, and $\Lambda_h$, respectively. We assume that these
interpolants have suitable approximation properties in the sense of
\eqref{eqs: approx props} and \eqref{eq: approx prop v}.

The main result concerning the analysis of the discrete problem \eqref{eq: general form-h}
is presented in the following theorem.

\begin{theorem} \label{thm: general error estimate}
Assume that:
\begin{enumerate}[label=A\arabic*., ref=A\arabic*]
\item The discrete spaces $V_h \times W_h$ are defined as in \eqref{eq: V decomposition}.

\item Problem \eqref{eqs: R problem general} has a unique solution and
  the resulting extension operator $\myR_h: \Lambda \to V_h$ is
  continuous, i.e. $\| \myR_h \lambda \|_V \lesssim \| \lambda \|_\Lambda$
  $\forall \, \lambda \in \Lambda$.

\item The four inequalities from Lemma~\ref{lem: Brezzi conditions} hold, i.e.,
  the bilinear forms $a(\cdot,\cdot)$ and $b(\cdot,\cdot)$ are continuous,
  $a(\cdot,\cdot)$ is coercive, and the finite element pair $V_h\times W_h$ is
  inf-sup stable.
  
\item The interpolant $\Pi^V$ is b-compatible in the sense of Lemma~\ref{lem: B-compatible}.
	\end{enumerate}
Then the discrete problem \eqref{eq: general form-h} admits a unique
solution that depends
continuously on the data $f$. Moreover,
the following \emph{a priori} error estimate holds:
	\begin{align} \label{eq: Error estimate general}
		\| \bm{u} - \bm{u}_h \|_V
		+
		\| p - p_h \|_W
		\lesssim
		\| \Pi^V \bm{u} - \bm{u} \|_V
		+
		\| \Pi^W p - p \|_W
		+
		\mathcal{E}_c,
	\end{align}
	with the consistency error defined as
\begin{align}\label{consist-error-general}
\mathcal{E}_c := \sup_{\bm{v}_h \in V_h}
\frac{a(\bm{u}, \bm{v}_h) - b(\bm{v}_h, p) - (\bm{g}, \bm{v}_h)_{\Omega}}{\| \bm{v}_h \|_V}.
\end{align}
\end{theorem}
\begin{proof}
The well-posedness of the discrete problem follows from the standard saddle point theory
\cite{boffi2013mixed} in a way similar to Theorem~\ref{thm: well-posedness discrete}.
For the error estimate, we follow the same steps as in the beginning
of Section~\ref{sub:error_estimates}. In short, we form the error equations
as in \eqref{eqs: step1 errors}, choose test
functions $(\bm{v}_h, p_h) \in V_h \times W_h$ as in \eqref{eq: test
  functions}, and bound the different terms in the error equations as in
\eqref{eqs: component bounds} to obtain the error estimate \eqref{eq:
  intermediate}, which is exactly \eqref{eq: Error estimate general}.
\end{proof}

\begin{lemma} \label{lem: uniqueness mortar}
Assume, in addition to assumptions A1-4, that:
	\begin{enumerate}[label=A\arabic*., ref=A\arabic*, start=5]
		\item The mortar condition $\| \mu_h \|_\Gamma \lesssim \sum_i \| \mathcal{Q}_{h, i} \mu_h \|_{\Gamma_i}$ holds for all $\mu_h \in \Lambda_h$.
	\end{enumerate}
	Then a unique mortar variable $\lambda_h \in \Lambda_h$ exists such that $\bm{u}_h = \bm{u}_h^0 + \myR_h \lambda_h$ with $\bm{u}_h^0 \in V_h^0$.
\end{lemma}
\begin{proof}
	See Theorem~\ref{thm: well-posedness discrete}.
\end{proof}

\begin{remark}
A condition similar to A5 may be necessary at an earlier stage,
e.g. in the definition of the discrete extension operator $\myR_h$
(Lemma~\ref{lem: sharp projection well-posed}), to establish the
approximation property of $\Pi^V$ (Lemma~\ref{lem: approximation prop
  Pi}), or to bound the consistency error (Section~\ref{ssub:
  err_projection_flat}). However, we emphasize that A5 is not
necessary in general to ensure uniqueness of the discrete solution
$(\bm{u}_h, p_h)$.
\end{remark}

\subsection{Coupling Stokes and Darcy flows}
\label{sub:coupling_stokes_flow}

As an example, we consider a coupled system of porous medium flow and
Stokes flow and follow all the steps from the
previous section to formulate and analyze the corresponding
flux-mortar MFE method. For this setting, let $\Omega_S$ and
$\Omega_D$ form a disjoint decomposition of $\Omega$ into regions of Stokes
and Darcy flow, respectively. For ease of presentation, we assume that
both $\Omega_S$ and $\Omega_D$ are simply connected domains. More general
configurations can also be treated, see, e.g. \cite{GirVasYot}.
Let the Stokes-Darcy interface be given by $\Gamma_{SD} := \partial \Omega_S
\cap \partial \Omega_D$. Let $\Gamma_S = \partial\Omega\cap\partial\Omega_S$
and $\Gamma_D = \partial\Omega\cap\partial\Omega_D$.
Denoting the restriction of a function to
$\Omega_S$ or $\Omega_D$ by a subscript $S$ or $D$, respectively, the
governing equations of the coupled Stokes-Darcy problem are
\cite{LSY}:
\begin{subequations} \label{eqs: Stokes-Darcy}
\begin{align}
  \sigma_S &:= \tilde\mu \epsilon(\bu_S) - p_S I, &
  & &
  \mbox{ in } &\Omega_S, \\
  - \nabla \cdot \sigma_S &= \bm{g}_S, &
  \nabla\cdot \bu_S &= f_S &
  \mbox{ in } &\Omega_S, \label{eq: mom balance SD}&
  \\
  \bu_D &= -K \nabla p_D, &
  \nabla\cdot \bu_D &= f_D &
  \mbox{ in } &\Omega_D, \label{eq: Darcy SD} \\
  \bm{\nu} \times (\sigma_S \bm{\nu}) &= - \bm{\nu} \times (\beta \bu_S), &
  \bm{\nu} \cdot \bu_S &= \bm{\nu} \cdot \bu_D &
  \mbox{ on } &\Gamma_{SD}, \label{eq: interface SD 1}\\
  \bm{\nu} \cdot (\sigma_S \bm{\nu}) &= - p_D &
  &
  & \mbox{ on } &\Gamma_{SD}, \label{eq: interface SD 2}\\
  \bm{u}_S &= 0
  \quad \mbox{ on } \Gamma_S, &
  p_D &= 0 &
  \mbox{ on } &\Gamma_D. \label{eq: BC SD}
\end{align}
\end{subequations}
Here, $\tilde\mu$ represents the viscosity, $\bm{g}_S$ is a body force,
$f$ is the mass source, 
$\beta$ is the Beavers-Joseph-Saffman (BJS) constant, $\bm{\nu}$
is the unit normal to $\Gamma_{SD}$ oriented outward with respect to
$\Omega_S$, $\bm{\nu} \times \bm{v}$ is the cross product if $n = 3$
and $\bm{\nu} \times \bm{v} = \bm{\nu}^\perp \cdot \bm{v}$ for $n = 2$
with $\perp$ denoting a rotation of $\pi/2$.  Moreover, $\varepsilon$
denotes the symmetric gradient, i.e. $\varepsilon(\bm{v}) := \frac12
(\nabla \bm{v} + (\nabla \bm{v})^T)$. 

Let us continue by defining the function spaces $\tilde{V} \times \tilde{W}$:
\begin{align*}
	\tilde{V} &:= \left\{ 
	\bm{v} \in H(\div, \Omega) :\ 
	\bm{v}_S \in (H^1(\Omega_S))^n, \quad 
	\bm{v}_S|_{\Gamma_S} = 0 
	\right\}, &
	\tilde{W} &:= L^2(\Omega).
\end{align*}
Next, we introduce index sets $I_S$ and $I_D$ to further decompose
$\Omega_S = \bigcup_{i \in I_S} \Omega_i$ and $\Omega_D = \bigcup_{i
  \in I_D} \Omega_i$ and we define $I_\Omega = I_S \cup I_D$. The
interfaces internal to $\Omega_S$ and $\Omega_D$ are denoted by
$\Gamma_{SS}$ and $\Gamma_{DD}$, respectively. Let $\Gamma = \Gamma_{DD} \cup
\Gamma_{SS} \cup \Gamma_{SD}$ and $\Gamma_i := \Gamma \cap \partial
\Omega_i$.

The variational formulation of problem \eqref{eqs: Stokes-Darcy}
obtains the form \eqref{eq: general form} by defining the bilinear
forms $a_i$ and $b_i$ per subdomain as follows \cite{LSY,GirVasYot}:
\begin{subequations}  \label{eqs: bilinear forms SD}
\begin{align}
a_i(\bm{u}_i, \bm{v}_{i}) &:=
(K^{-1} \bm{u}_i, \bm{v}_{i})_{\Omega_i}
, &
i &\in I_D, \\
a_i(\bm{u}_i, \bm{v}_{i}) &:=
(\tilde\mu \varepsilon(\bm{u}_i), \varepsilon(\bm{v}_{i}))_{\Omega_i}
+
(\beta \bm{\nu}_i \times \bm{u}_i, \bm{\nu}_i \times \bm{v}_{i})_{\Gamma_i \cap \Gamma_{SD}}
, &
i &\in I_S, \\
b_i(\bm{u}_i, w_i) &:= (\nabla \cdot \bm{u}_i, w_i)_{\Omega_i}, & i &\in I_\Omega.
\end{align}
\end{subequations}
It is shown in \cite{LSY,GirVasYot} that this variational formulation
has a unique solution.

Let $V_i$ and $W_i$ be the respective restrictions of $\tilde{V}$ and
$\tilde{W}$ to subdomain $\Omega_i$. The composite spaces $V$ and $W$
are defined as in \eqref{eq: composite spaces 2} and we associate the
following norms:
\begin{align*}	
	\| \bm{v} \|_V &:= 
	\sum_{i} \| \bm{v}_{i} \|_{V_i}
	= \sum_{i \in I_S} \| \bm{v}_{i} \|_{1, \Omega_i} +
	\sum_{i \in I_D} \| \bm{v}_{i} \|_{\div, \Omega_i}, \\
	\| w \|_W &:= 
	\sum_i \| w_i \|_{W_i} = 
	\sum_i \| w_i \|_{\Omega_i}.
\end{align*}

Next, we define the local trace operators $\Tr_i$. For $i \in I_D$,
let $\Tr_i$ be the normal trace operator on $\Gamma_i$, as in the
Darcy model problem of Section~\ref{sec:model_problem}. For $i \in
I_S$, on the other hand, let $\Tr_i \bm{v}_i$ be the trace of all
components of the vector function $\bm{v}_i$ onto $\Gamma_i$.
Note that this leads to a discrepancy on
$\Gamma_{SD}$ because $\Tr_i \bm{v}_i$ is scalar-valued for $i \in
I_D$ but vector-valued for $i \in I_S$. We therefore make a slight
alteration to the characterization \eqref{eq: characterization
  V_tilde}:
\begin{align*}
\tilde{V} = \{ \bm{v} \in V :\ 
	&\Tr_i \bm{v}_i = \bm{\nu} \cdot \Tr_j \bm{v}_j \ \text{on each } \Gamma_{ij} \text{ with } (i, j) \in I_D \times I_S , \\
	&\Tr_i \bm{v}_i = \Tr_j \bm{v}_j \ \text{on each } \Gamma_{ij} \not \subseteq \Gamma_{SD}
	\}.
\end{align*}
In turn, let the global trace operator on $\tilde{V}$, denoted by
$\Tr$, be the normal trace on $\Gamma_{DD}$ and the
full trace on $\Gamma_{SS} \cup \Gamma_{SD}$.
We then define the trace space
\begin{align*} 
	\Lambda := \left\{
	 \mu \in \Tr \tilde{V}: \mu|_{\Gamma_i} \in L^2(\Gamma_i) \ \ \forall \, i \in I_D
	\right\}.
\end{align*}
Let $\Lambda_i := \{\mu|_{\Gamma_i}, \mu \in \Lambda\}$, where the meaning
of the restriction on $\Gamma_i \cap
\Gamma_{SD}$ is either the full vector $\bmu|_{\Gamma_i \cap \Gamma_{SD}}$ for $i \in I_S$
or the normal
component $\bnu\cdot \bmu|_{\Gamma_i \cap \Gamma_{SD}}$ for $i \in I_D$.
Here, and in the following, we use a boldface $\bm{\mu}_i$ to denote
vector-valued components of $\mu$.
The space $\Lambda$ is 
endowed with the norm $\| \mu \|_\Lambda := \sum_{i} \| \mu_i \|_{\Lambda_i}$, with
\begin{align*}
  \|\mu_i \|_{\Lambda_i} := \| \mu_i \|_{\Gamma_i} \ \ \forall \, i \in I_D, \quad
  \|\bmu_i \|_{\Lambda_i} := \left\{ \begin{array}{ll}
    \| \bm{\mu}_i \|_{\frac12,\partial\Omega_i}, &
    \partial\Omega_i \cap \Gamma_S = \emptyset \\
    \|E_{i,0} \bm{\mu}_i\|_{\frac12,\partial\Omega_i}, & \partial\Omega_i \cap \Gamma_S \ne
    \emptyset  
  \end{array} \right. \forall i \in I_S,
  \end{align*}
where $E_{i,0}$ is the extension by zero to
$\partial\Omega_i$.

We end this subsection with a statement of a version of Korn's inequality
\cite[(1.8)]{Brenner}, which will be used to establish coercivity of the bilinear form
$a(\cdot,\cdot)$. Let $\mathcal{O} \subset \mathbb{R}^n$, $n = 2,3$ be a connected
bounded domain and let $\mathcal{G}$ with $|\mathcal{G}| > 0$ be a section of its
boundary. Then, for all
$\bv \in (H^1(\mathcal{O}))^n$,
\begin{equation}\label{Korn}
  |\bv|_{1,\mathcal{O}} \lesssim \Big( \|\epsilon(\bv)\|_{\mathcal{O}} +
  \sup_{\begin{array}{c}
      \bm{m} \in \RM(\mathcal{O}) \\
      \|\bm{m}\|_\mathcal{G} = 1, \, \int_\mathcal{G} \bm{m} \, ds = 0 \end{array}}
(\bv,\bm{m})_\mathcal{G} = 0 \Big),
\end{equation}
where $\RM(\mathcal{O})$ is the space of rigid body motions on $\mathcal{O}$. Combined
with Poincar\'e inequality \cite{quarteroni1999domain}, for all
$\bv \in (H^1(\mathcal{O}))^n$ with $\int_\mathcal{G} \bv \, ds = 0$,
$\|\bv\|_{\mathcal{O}} \lesssim |\bv|_{1,\mathcal{O}}$, \eqref{Korn} implies that for all
$\bv \in (H^1(\mathcal{O}))^n$ with $(\bv,\bm{m})_\mathcal{G} = 0
\ \forall\, \bm{m} \in \RM(\mathcal{O})$,
\begin{equation}\label{Korn-2}
    \|\bv\|_{1,\mathcal{O}} \lesssim \|\epsilon(\bv)\|_{\mathcal{O}}.
\end{equation}

\subsubsection{Discretization}

For each $i \in I_\Omega$, let $\Omega_{h,i}$ be a shape-regular mesh, allowing for
non-matching grids along the interfaces. We choose a finite element pair $V_{h, i}
\times W_{h, i} \subset V_i \times W_i$ such that it is stable for the Darcy subproblem if $i
\in I_D$ and for the Stokes subproblem if $i \in I_S$. Stable MFE pairs for
the Stokes subproblems, see e.g. \cite[Chapter 8]{boffi2013mixed},
include the Taylor-Hood pair, the MINI mixed finite element, and the
Bernardi-Raugel pair. Note that the essential boundary condition
$\bu_S = 0$ on $\Gamma_S$ is built in $V_{h, i}$.

We next define the discrete flux space $\Lambda_h \subset \Lambda$. On
$\Gamma_{DD}$, the discrete space $\Lambda_{h,D} \subset
L^2(\Gamma_{DD})$ is defined interface by interface as described in
Section~\ref{sec:discretization}. On $\Gamma_{SS} \cup \Gamma_{SD}$ we
consider a globally conforming shape-regular mesh. Such mesh can be
obtained as the trace of a mesh $\tilde \Omega_{h,S}$ on $\Omega_S$
that is aligned with the domain decomposition. Let $\tilde V_{h,S}
\subset \tilde V|_{\Omega_S}$ be a conforming Lagrange finite element
space on $\tilde \Omega_{h,S}$. We define the discrete flux space on
$\Gamma_{SS} \cup \Gamma_{SD}$ as $\Lambda_{h,S} := \Tr \tilde
V_{h,S}$.  To ensure the mortar condition A5, the space
$\Lambda_{h,S}$ must be defined on a sufficiently coarse mortar
grid, see \cite{GirVasYot} for specific examples.

Due to the boundary condition \eqref{eq: BC SD}, we redefine $I_{int}
:= I_S \cup \{ i \in I_D :\ \partial \Omega_i \subseteq \Gamma \}$. In
turn, the space $S_H$, defined by \eqref{eq: definition S_H}, is given
explicitly by \eqref{eq: definition S_Hi darcy}. Let $W_{h,i}^0 :=
W_{h, i} \cap S_{H, i}^\perp$.  We emphasize that the following
inf-sup condition holds for all $i \in I_\Omega$:
\begin{align} \label{ineq: b_i infsup SD}
  \forall w_{h, i}^0 \in W_{h, i}^0, \ \exists \, 0 \ne \bm{v}_{h, i}^0 \in V_{h, i}^0
  \text{ such that }
b_i(\bm{v}_{h, i}^0, w_{h, i}^0) \gtrsim \| \bm{v}_{h, i}^0 \|_{V_i} \| w_{h, i}^0 \|_{W_i}.
\end{align}

We continue with the definition of the operator $\mathcal{Q}_{h, i} :
\Lambda \to \Tr_i V_{h, i}$. For $i \in I_D$, recall that the space
$\Lambda$ has a different number of components on $\Gamma_{DD}$ and
$\Gamma_{SD}$.  On $\Gamma_i \cap \Gamma_{SD}$, let $\mathcal{Q}_{h,
  i} \bm{\lambda}$ be the $L^2$-projection of $\bm{\nu} \cdot \bm{\lambda}$
onto the normal trace space $(\Tr_i V_{h, i})|_{\Gamma_i \cap \Gamma_{SD}}$.
On $\Gamma_i \cap \Gamma_{DD}$, let
$\mathcal{Q}_{h, i}$ be the $L^2$-projection
$\mathcal{Q}_{h,i}^\flat$
from Section~\ref{sub: disc_projection_from_the_mortar_space}.

Now consider $i \in I_S$. The operator $\mathcal{Q}_{h, i}$ needs to
satisfy
\begin{equation}\label{prop-Qh}
(\bnu_i\cdot (\mathcal{Q}_{h, i} \bm{\lambda} - \bm{\lambda}),1)_{\Gamma_i} = 0,  \quad i \in I_S,
\end{equation}
which is needed for inf-sup stability, cf. Lemma~\ref{lem: A3 SD}, and
b-compatibility of the interpolant $\Pi^V$, cf. Lemma~\ref{lem: A4 SD}.
The $L^2$-projection onto $\Tr_i V_{h, i}$ does not satisfy
\eqref{prop-Qh}, since the space $\Tr_i V_{h, i}$ is continuous on
$\Gamma_i$, but the normal vector $\bnu_i$ is discontinuous at the
corners of the subdomains. We therefore need a different construction.
Let $\myI_{\Gamma_i}: \Lambda_i \to \Tr_i V_{h, i}$ be a suitable interpolant or
projection with optimal approximation properties. Specific choices of $\myI_{\Gamma_i}$
will be discussed below. Since
$\myI_{\Gamma_i}$ may not satisfy \eqref{prop-Qh}, we correct it
on each flat face $F$ of $\Gamma_{i}$. We assume that,
given $\bm{\lambda} \in \Lambda_i$, there exists
$\bm{c}^F_{h,i} \in \Tr_i V_{h, i}|_F \cap
(H^1_0(F))^n$ such that
\begin{equation}\label{Stokes-mortar}
  (\bm{c}^F_{h,i},\bchi_{h,i})_F
  = (\bm{\lambda} - \myI_{\Gamma_i} \bl,\bchi_{h,i})_F \ \ \forall \bchi_{h,i} \in V^F_{h,i},
  \quad \|\bm{c}^F_{h,i}\|_{\frac12,F} \lesssim \|\bl - \myI_{\Gamma_i}\bl\|_{\frac12,F},
\end{equation}
where $V^F_{h,i}$ is a suitably defined finite element space on $F$ such that
$\bnu_i|_F \in V^F_{h,i}$.
We refer to \cite[Appendix]{GirVasYot} for examples of spaces and
constructions of $\bm{c}^F_{h,i}$. In particular, in two dimensions, assuming
that $\bl \in \mathcal{C}^0(\Gamma_i)$, we can take $\myI_{\Gamma_i}\bl$ to be the Lagrange
interpolant and use the constructions from \cite[Section~7.1]{GirVasYot}. Alternatively,
in both two and three dimensions we can take $\myI_{\Gamma_i}$ to be the $L^2$-projection
onto $\Tr_i V_{h, i}$ and use the construction from \cite[Section~7.2]{GirVasYot}.
We then define
$$
\mathcal{Q}_{h, i} \bl := \myI_{\Gamma_i} \bl + \sum_{F \subset \Gamma_i} \bm{c}^F_{h,i},
$$
which satisfies for each face $F$, 
\begin{equation}\label{Qh-orth}
  (\mathcal{Q}_{h, i} \bl - \bl, \bchi_{h,i})_F = 0 \quad \forall \bchi_{h,i} \in V^F_{h,i}.
\end{equation}
Since $\bnu_i|_F \in V^F_{h,i}$, then \eqref{prop-Qh} holds.
A scaling argument similar to the one in \cite[Lemma~5.1]{GirVasYot}
shows that that $\mathcal{Q}_{h, i}$ is stable and has optimal
approximation properties in $\|\cdot\|_{\Lambda_i}$. We further note that
the approximation property of the space $V^\Gamma_{h,i}$ on $\Gamma_i$,
$V^\Gamma_{h,i}|_F := V^F_{h,i}$, does not affect
the approximation property of $\mathcal{Q}_{h, i}$, but it affects the
consistency error $\mathcal{E}_c$, cf. Lemma~\ref{lem: consistency error SD}.

We now have all the ingredients to set up problem \eqref{eqs: R
  problem general} and therewith define the extension operator
$\myR_{h, i}$. In turn, the discrete spaces $V_h
\times W_h$ are defined as in \eqref{eq: V decomposition}.
The discrete Stokes-Darcy problem is then defined by \eqref{eq:
  general form-h}, posed on $V_h \times W_h$, with the bilinear forms
from \eqref{eqs: bilinear forms SD}.

\begin{remark}
The choice of a full vector $\bm{\lambda}_h$ on
$\Gamma_{SD}$ is different from previously developed pressure-mortar
methods for the Stokes-Darcy problem
\cite{LSY,GirVasYot,galvis-sarkis,VasWangYot}, where $\lambda_h$ is a
scalar on $\Gamma_{SD}$ modeling $\bm{\nu} \cdot (\sigma_S \bm{\nu}) =
- p_D$ and used to impose weakly $\bm{\nu} \cdot \bu_S = \bm{\nu}
\cdot \bu_D$. In a domain decomposition implementation, the BJS
boundary condition is incorporated into the subdomain solves
\cite{GirVasYot,VasWangYot}. In contrast, in our method, the BJS term
$(\beta \bm{\nu} \times \bm{u}_i, \bm{\nu} \times \bm{v}_{i})_{\Gamma_i \cap \Gamma_{SD}}$
is eliminated from the subdomain solves, since $\bm{v}_{h, i}^0 = 0$ on $\partial\Omega_i$
in \eqref{Rh-Vh}. The Stokes subdomain problems are of Dirichlet type
with data $\mathcal{Q}_{h, i} \bmu_h$. In turn, the BJS boundary condition
is incorporated into the coupled system \eqref{eq: general form-h}
via the BJS term in the bilinear form $a(\cdot,\cdot)$ in \eqref{eqs: bilinear forms SD}. In
the domain decomposition implementation, the BJS boundary condition
is incorporated into the interface operator, see
Section~\ref{ssub:reduction_to_an_interface_problem_SD}.
\end{remark}

\subsubsection{Interpolants}

We next define appropriate interpolants in the discrete spaces. We
define $\Pi^W$ as the $L^2$-projection onto $W_h$, $\Pi^\Lambda$ as the
$L^2$-projection onto $\Lambda_h$, and $\Pi^{V^\Gamma}_i$ as the
$L^2$-projection onto $V^\Gamma_{h,i}$.

The interpolant $\Pi^V$ is constructed according to the
following steps. First, for $i \in I_D$, let $\Pi_i^V$ be the
b-compatible interpolant associated with $V_{h, i}$ introduced in
Section~\ref{sub:interpolation_operators} and satisfying properties
\eqref{eq: commutativity}--\eqref{eq: Pi-trace}. Note that \eqref{eq:
  Pi-trace} implies $\Tr_i \Pi_i^V = \mathcal{Q}_{h, i} \Tr_i$ on $\Gamma_i$,
which is used in the construction of the composite
interpolant $\Pi^V$. However, canonical b-compatible interpolants for
Stokes finite elements do not typically satisfy this property. For
this reason, in the Stokes region we define $\Pi_i^V$ as a suitable
Stokes elliptic projection. More precisely, for $i \in I_S$,
given $\bm{u}_i$, we consider the discrete Stokes problem:
Find $(\Pi_i^V \bm{u}_i, p_{h, i}^u) \in V_{h, i} \times W^0_{h, i}$
such that
\begin{subequations} \label{Pi-Stokes}
\begin{align}
(\nabla(\Pi_i^V \bm{u}_i), \nabla\bm{v}_{h,i}^0)_{\Omega_i}
- (\nabla\cdot\bm{v}_{h, i}^0, p_{h,i}^u)_{\Omega_i}
&= (\nabla\bm{u}_i, \nabla\bm{v}_{h,i}^0)_{\Omega_i},
& \forall \bm{v}_{h, i}^0 &\in V_{h, i}^0, \\
(\nabla\cdot\Pi_i^V \bm{u}_i, w_{h, i})_{\Omega_i}
&= (\nabla\cdot\bm{u}_i, w_{h, i})_{\Omega_i},
& \forall w_{h, i} &\in W_{h, i}, \label{Pi-Stokes W_h} \\
\Pi_i^V \bm{u}_i
&= \mathcal{Q}_{h, i} \Tr_i \bm{u}_i,
& \text{ on } &\Gamma_i. \label{Pi-Stokes-Tr}
\end{align}
\end{subequations}
The well-posedness of the above problem and optimal approximation properties of $\Pi_i^V$
follows from standard Stokes finite element analysis \cite{boffi2013mixed}.

Let $\lambda = \Tr \bm{u}$. Note that, by construction, we have $\Tr_i
\Pi_i^V = \mathcal{Q}_{h, i} \Tr_i$ on $\Gamma_i$ for both
$i \in I_D$ and $i \in I_S$. Therefore $\Pi_i^V
(\bm{u}_i - \myR_{h,i} \lambda) \in V_{h, i}^0$.
Using this observation, the interpolant $\Pi^V$ onto $V_h$ is defined similarly to
\eqref{Pi-flat}:
\begin{align*}
	\Pi^V \bm{u} := \myR_h \Pi^\Lambda \lambda
        + \bigoplus_i \Pi_i^V (\bm{u}_i - \myR_{h,i} \lambda)
= \myR_h (\Pi^\Lambda\lambda - \lambda) + \bigoplus_i \Pi_i^V \bm{u}_i.
\end{align*}
The continuity of $\myR_h$, which will be established in Lemma~\ref{lem: A2 SD}, implies
the following approximation property of $\Pi^V$:
\begin{align} \label{eq: approximation PiV SD}
	\| \bm{u} - \Pi^V \bm{u} \|_V &\lesssim
	\sum_{i \in I_S} \| \bm{u} - \Pi_i^V \bm{u} \|_{1, \Omega_i}
		+ \sum_{i \in I_D} \| \bm{u} - \Pi_i^V \bm{u} \|_{\div, \Omega_i}
		+ \| \lambda - \Pi^\Lambda \lambda \|_\Lambda.
\end{align}

\subsubsection{Stability and error analysis}

\begin{theorem}
  The discrete Stokes-Darcy problem \eqref{eq: general form-h}
  has a unique solution $(\bm{u}_h,
p_h) \in V_h \times W_h$. If A5 holds, then there is a 
unique mortar solution $\lambda_h \in
\Lambda_h$. Moreover, the following error estimate holds with respect
to the solution $(\bm{u}, p)$ of \eqref{eq: general form}:
\begin{align*}
  & \| \bm{u} - \bm{u}_h \|_V + \| p - p_h \|_W \\
& \quad  \lesssim
\sum_{i \in I_S} \| \bm{u} - \Pi_i^V \bm{u} \|_{1, \Omega_i}
+ \sum_{i \in I_D} \| \bm{u} - \Pi_i^V \bm{u} \|_{\div, \Omega_i}
+ \| \lambda - \Pi^\Lambda \lambda \|_\Lambda
+ \sum_{i \in I_\Omega} \| p - \Pi_i^W p \|_{W_i} \\
& \qquad 
+ h^{-1/2} \sum_{i \in I_D}
\| p_D - \mathcal{Q}_{h, i} p_D \|_{\Gamma_i} + 
\sum_{i \in I_S} 
\|  \sigma_S \bm{\nu} - \Pi^{V^\Gamma}_i (\sigma_S \bm{\nu}) \|_{\Gamma_i}. 
\end{align*}
\end{theorem}
\begin{proof}
The proof is based on Theorem~\ref{thm: general error estimate}.  We
consider its assumptions A1-4. A1 is satisfied by construction, A2 is
shown in Lemma~\ref{lem: A2 SD}, A3 in Lemma~\ref{lem: A3 SD}, and A4
in Lemma~\ref{lem: A4 SD}. We then invoke Theorem~\ref{thm: general
  error estimate} to obtain existence and uniqueness of $(\bm{u}_h, p_h)$. The
uniqueness of $\lambda_h$ under A5 follows from Lemma~\ref{lem:
  uniqueness mortar}. Finally, we obtain the error estimate
by combining \eqref{eq: Error estimate general}, the
approximation property \eqref{eq: approximation PiV SD}, and the
estimate on the consistency error from Lemma~\ref{lem: consistency
  error SD}.
\end{proof}

\begin{lemma}[A2] \label{lem: A2 SD}
Problem \eqref{eqs: R problem general} has a unique solution and the resulting extension operator $\myR_h: \Lambda \to V_h$ is continuous, i.e. $\| \myR_h \lambda \|_V \lesssim \| \lambda \|_\Lambda$ $\forall \, \lambda \in \Lambda$.
\end{lemma}
\begin{proof}
For $i \in I_D$, the unique solvability of \eqref{eqs: R problem
  general} and the continuity of $\myR_{h, i}$ hold by Lemma~\ref{lem:
  R_h is well-posed}.  For $i \in I_S$,
we consider uniqueness by setting $\lambda = 0$.
Setting $w_{h, i} = 1$ in \eqref{eq: R problem general W_h}
and using the divergence theorem and \eqref{Rh-BC-general}, we 
obtain $r_i = 0$. Next, setting the test functions in
\eqref{Rh-Vh}--\eqref{eq: R problem general W_h}
to $(\myR_{h, i} \lambda, p_{h, i}^\lambda)$ and summing the 
equations gives us $\myR_{h, i} \lambda = 0$, using Korn's inequality
\eqref{Korn-2}. Moreover, we have $p_{h,i}^\lambda \perp S_{H, i}$
from \eqref{eq: R problem general S_H},
so we use the inf-sup condition
\eqref{ineq: b_i infsup SD}  and \eqref{Rh-Vh} to derive that $p_{h, i}^\lambda = 0$.

It remains to show continuity for $i \in I_S$. The first step is to
obtain a bound on $p_{h, i}^\lambda$. Since $p_{h, i}^\lambda \perp S_{H, i}$,
we use $\bm{v}_{h, i}^0$ from the inf-sup condition \eqref{ineq: b_i
  infsup SD} as a test function and use the continuity of $a_i(\cdot,\cdot)$ to
obtain
\begin{align*}
		\| \bm{v}_{h, i}^0 \|_{V_i} \| p_{h, i}^\lambda \|_{W_i}
		\lesssim 
		b_i(\bm{v}_{h, i}^0, p_{h, i}^\lambda) 
		= a_i(\myR_{h, i} \bm\lambda, \bm{v}_{h, i}^0)
		\lesssim
		\| \myR_{h, i} \bm\lambda \|_{V_i}
		\| \bm{v}_{h, i}^0 \|_{V_i}.
\end{align*}
Thus, $\| p_{h, i}^\lambda \|_{W_i} \lesssim \| \myR_{h, i} \bm\lambda \|_{V_i}$.

Next, let $\myR_{h, i}^\star \bm{\lambda} \in V_{h, i}$ be a continuous
discrete extension operator \cite[Theorem~4.1.3]{quarteroni1999domain}
satisfying $\Tr_i \myR_{h, i}^\star
\bm{\lambda} = \mathcal{Q}_{h, i} \bm{\lambda}$ on $\Gamma_i$ and
\begin{align*}
\| \myR_{h, i}^\star \bm{\lambda} \|_{V_i} \lesssim 
\| \mathcal{Q}_{h, i} \bm{\lambda} \|_{\Lambda_i}
\lesssim \| \bm{\lambda} \|_{\Lambda_i}.
\end{align*}
We take as test functions in \eqref{eqs: R problem general}
$\bv_{h,i}^0 = \bm\varphi_{h,i}^0 := (\myR_{h, i} - \myR_{h, i}^\star) \bm\lambda \in
V_{h, i}^0$, $w_{h,i} = p_{h,i}^\lambda$, $s_i = r_i$, and combine the equations.
Using Korn's inequality \eqref{Korn-2}, 
the continuity of $a_i$ and $b_i$, Young's inequality, and the bounds on $p_{h,
  i}^\lambda$ and $\myR_{h, i}^* \bm\lambda$, we derive
\begin{align*}
\|\bm\varphi_{h,i}^0\|_{V_i}^2 
&\lesssim a_i(\bm\varphi_{h,i}^0,\bm\varphi_{h,i}^0)
= - a_i(\myR_{h, i}^\star \bm\lambda,\bm\varphi_{h,i}^0) + 
b_i(-\myR_{h, i}^\star \bm\lambda, p_{h, i}^\lambda)\\
&\lesssim \| \myR_{h, i}^\star \bm\lambda \|_{V_i} 
( \|\bm\varphi_{h,i}^0\|_{V_i} +  \| p_{h, i}^\lambda \|_{W_i})
\lesssim \| \bm\lambda \|_{\Lambda_i}^2
+ \epsilon (\|\bm\varphi_{h,i}^0\|_{V_i}^2 +  \| \myR_{h, i} \bm\lambda \|_{V_i}^2).
\end{align*}
Combining this bound with
$\| \myR_{h, i} \bm\lambda \|_{V_i}^2 \lesssim \|\bm\varphi_{h,i}^0\|_{V_i}^2
+ \| \myR_{h, i}^\star \bm\lambda \|_{V_i}^2 \lesssim
\|\bm\varphi_{h,i}^0\|_{V_i}^2 + \| \bm\lambda \|_{\Lambda_i}^2$ and taking $\epsilon$ small
enough, we obtain $\|\bm\varphi_{h,i}^0\|_{V_i} \lesssim \| \bm\lambda \|_{\Lambda_i}$,
which implies $\| \myR_{h, i} \bm\lambda \|_{V_i} \lesssim \| \bm\lambda \|_{\Lambda_i}$
for $i \in I_S$, concluding the proof.
\end{proof}

\begin{lemma}[A3] \label{lem: A3 SD}
	The four inequalities from Lemma~\ref{lem: Brezzi conditions} hold for $V_h \times W_h$.
\end{lemma}
\begin{proof}
Let us consider the inequalities \eqref{ineqs: Brezzi
  conditions}. First, $b_i$ from \eqref{eqs: bilinear forms SD} is
continuous due to the Cauchy-Schwarz inequality. The same holds for
$a_i$ with $i \in I_D$. For $i \in I_S$, we additionally use a trace
inequality to bound $ \beta \| \bm{\nu} \times \bm{u}_{h, i}
\|_{\Gamma_{SD}} \| \bm{\nu} \times \bm{v}_{h, i} \|_{\Gamma_{SD}}
\lesssim \| \bm{u}_{h, i} \|_{1, \Omega_i} \| \bm{v}_{h, i} \|_{1,
  \Omega_i}$.  Third, the coercivity of $a_i$ for $i \in I_D$ is shown
in Lemma~\ref{lem: Brezzi conditions}. For $i \in I_S$, Korn's
inequality \eqref{Korn} cannot be applied locally, since the velocity
is not restricted on subdomain boundaries. To that end, we recall that
$\Lambda_{h,S} = \Tr \tilde V_{h,S}$, where $\tilde V_{h,S}
\subset \tilde V|_{\Omega_S}$ is a conforming Lagrange finite element
space on a mesh $\tilde \Omega_{h,S}$ that is aligned with the domain
decomposition. Let $\tilde V_{h,i} = \tilde V_{h,S}|_{\Omega_i}$, $i \in I_S$.
We can write $\tilde V_{h,i} =
\tilde V_{h,i}^0 \oplus \mathcal{E}_{h,i}\Lambda_h$, where
$\tilde V_{h,i}^0 = \{\tilde\bv_{h,i} \in \tilde V_{h,i}:
\Tr_i \tilde\bv_{h,i} = 0 \text{ on } \Gamma_i\}$
and $\mathcal{E}_{h,i}: \Lambda_h \to \tilde V_{h,i}$ is a discrete
extension operator such that $\mathcal{E}_{h,i} \bmu_h = \bmu_h$ on
$\Gamma_i$ and
\begin{equation}\label{Eh}
  a_i(\mathcal{E}_{h,i} \bmu_h, \tilde\bv_{h,i}^0) = 0, \quad \forall \,
  \tilde\bv_{h,i}^0 \in \tilde V_{h,i}^0.
\end{equation}
Problem \eqref{Eh} is well posed, since, due to
\eqref{Korn-2}, $a_i(\cdot,\cdot)$ is coercive on $\tilde V_{h,i}^0$. Now, given
$\bu_{h,i} = \bu_{h,i}^0 + \myR_{h, i} \bm{\lambda}_h$, consider the local problem:
Find $\tilde\bu_{h,i} = \tilde\bu_{h,i}^0 + \mathcal{E}_{h, i} \bm{\lambda}_h
\in \tilde V_{h,i}$ such that
\begin{equation}\label{uh-tilde}
  a_i(\tilde\bu_{h,i}^0 + \mathcal{E}_{h, i} \bm{\lambda}_h,
  \tilde\bv_{h,i}^0 + \mathcal{E}_{h, i} \bm{\lambda}_h) =
  a_i(\bu_{h,i}, \tilde\bv_{h,i}^0 + \mathcal{E}_{h, i} \bm{\lambda}_h), \quad
  \forall \, \tilde\bv_{h,i}^0 \in \tilde V_{h,i}^0.
  \end{equation}
Note that $\bu_{h,i}$ and $\bm{\lambda}_h$ are given data. Problem \eqref{uh-tilde}
is well posed, since, using \eqref{Eh},
$a_i(\tilde\bu_{h,i}^0 + \mathcal{E}_{h, i} \bm{\lambda}_h,
\tilde\bv_{h,i}^0 + \mathcal{E}_{h, i} \bm{\lambda}_h) =
a_i(\tilde\bu_{h,i}^0,\tilde\bv_{h,i}^0) + a_i(\mathcal{E}_{h, i} \bm{\lambda}_h,
\mathcal{E}_{h, i} \bm{\lambda}_h)$, and the coercivity follows from \eqref{Korn-2}.
We further note that \eqref{uh-tilde} implies that
$a_i(\bu_{h,i} - \tilde\bu_{h,i},\tilde\bu_{h,i}) = 0$.
Also, \eqref{Qh-orth} implies that, for all $\bm{m} \in \RM(\Omega_i)$,
$(\bu_{h,i} - \tilde\bu_{h,i},\bm{m})_{\Gamma_i} =
(\mathcal{Q}_{h, i} \bm{\lambda}_h - \bm{\lambda}_h,\bm{m})_{\Gamma_i} = 0$.
Hence, Korn's inequality \eqref{Korn-2} on $\Omega_i$ gives
$\|\bu_{h,i} - \tilde\bu_{h,i}\|^2_{1,\Omega_i} \lesssim
a_i(\bu_{h,i} - \tilde\bu_{h,i},\bu_{h,i} - \tilde\bu_{h,i})$.
Then, with $\tilde\bu_{h,S} \in \tilde V_{h,S}$ defined as
$\tilde\bu_{h,S}|_{\Omega_i} = \tilde\bu_{h,i}$, we have
\begin{align*}
  \sum_{i \in I_S} a_i(\bu_{h,i},\bu_{h,i})
  & = \sum_{i \in I_S} a_i(\bu_{h,i} - \tilde\bu_{h,i},\bu_{h,i} - \tilde\bu_{h,i})
  + \sum_{i \in I_S} a_i(\tilde\bu_{h,i},\tilde\bu_{h,i})\\
  & \gtrsim \sum_{i \in I_S} \|\bu_{h,i} - \tilde\bu_{h,i}\|^2_{1,\Omega_i}
  + \|\tilde\bu_{h,S}\|^2_{1,\Omega_S}
  \gtrsim \sum_{i \in I_S} \|\bu_{h,i}\|^2_{1,\Omega_i},
\end{align*} 
where in the first inequality we used Korn's inequality \eqref{Korn-2} applied
globally on $\Omega_S$. This completes the proof of the coercivity of $a(\cdot,\cdot)$ on
$V_h$.

Next, we prove the inf-sup condition \eqref{ineq: b_infsup} by
constructing $\bm{v}_h \in V_h$ for a given $w_h \in W_h$.
We follow the approach from Lemma~\ref{lem: Brezzi conditions}
and consider a global divergence problem on $\Omega$, cf. \eqref{global-div}
to construct $\bm{v}^w \in (H^1(\Omega))^n$ with the properties
\begin{align*}
  \nabla \cdot \bm{v}^w &= w_h \text{ in } \Omega, \quad \bm{v}^w = 0 \text{ on }
  \Gamma_S, \quad \| \bm{v}^w \|_{1, \Omega} \lesssim \| w_h \|_\Omega.
\end{align*}
The construction of $\bm{v}_h$ in $\Omega_D$ is presented in
Lemma~\ref{lem: Brezzi conditions}. We now consider the construction
of $\bm{v}_h$ in $\Omega_S$. The approach used in $\Omega_D$ to construct
$\mu_h \in \Lambda_h$ does not work in $\Omega_S$, due
the global continuity of $\Lambda_h$. Instead, we consider a discrete Stokes
problem in $\Omega_S$ based on the finite element pair $\tilde V_{h,S} \times W_{H,S}$,
where we recall that $\Lambda_{h,S} = \Tr \tilde V_{h,S}$ and we define $W_{H,S}$ to
be the space of piecewise constants on the partition form by the subdomains $\Omega_i$,
$i \in I_S$. Assuming that there is at least one interior vertex in each $\Gamma_{ij}$,
the pair $\tilde V_{h,S} \times W_{H,S}$ is inf-sup stable, see
\cite[Lemma~3.3]{Stenberg84}. Let
$\tilde \bu^w_{h,S} \in \tilde V_{h,S}$ be a discrete Stokes projection of $\bm{v}^w$
in $\Omega_S$ based on solving the problem: Find
$(\tilde \bu^w_{h,S},p_{H,S}^w) \in \tilde V_{h,S} \times W_{H,S}$ such that
\begin{subequations} \label{global-Stokes}
\begin{align}
&  (\nabla\tilde \bu^w_{h,S}, \nabla\tilde\bv_{h,S})_{\Omega_S}
- (\nabla\cdot\tilde\bv_{h,S}, p_{H,S}^w)_{\Omega_S}
= (\nabla\bv^w, \nabla\tilde\bv_{h,S})_{\Omega_S}, \ \
\forall \, \tilde\bv_{h,S} \in \tilde V_{h, S}, \\
& (\nabla\cdot\tilde\bu^w_{h,S}, w_{H,S})_{\Omega_S}
= (\nabla\cdot\bv^w, w_{H,S})_{\Omega_S}, \ \
 \forall \, w_{H,S} \in W_{H,S}. \label{global-Stokes-WH}
\end{align}
\end{subequations}
The continuity of the Stokes finite element approximation implies
$\|\tilde \bu^w_{h,S}\|_{1,\Omega_S} \lesssim \|\bv^w\|_{1,\Omega_S}$. We now define
$\bmu_h := \Tr \tilde \bu^w_{h,S} \in \Lambda_{h,S}$. The trace inequality implies
$$
\sum_{i \in I_S} \| \bm{\mu}_h \|_{\Lambda_i}
\lesssim \sum_{i \in I_S} \|\tilde \bu^w_{h,S}\|_{1,\Omega_i}
\lesssim \|\bv^w\|_{1,\Omega_S} \lesssim \|w_h \|_\Omega.
$$
Moreover, \eqref{global-Stokes-WH} gives
$$
(\bnu_i\cdot\bm{\mu}_h,1)_{\Gamma_i} = (\nabla\cdot\tilde\bu^w_{h,S},1)_{\Omega_i}
= (\nabla\cdot\bv^w,1)_{\Omega_i} = (w_h,1)_{\Omega_i}, \quad \forall \, i \in I_S.
$$
Now, using \eqref{Rh-BC-general} and \eqref{prop-Qh}, we obtain
$$
(\nabla\cdot\myR_{h, i} \bmu_h, 1)_{\Omega_i}
= (\bnu_i\cdot \mathcal{Q}_{h, i} \bmu_h,1)_{\Gamma_i}
= (\bnu_i\cdot \bmu_h,1)_{\Gamma_i} = (w_h,1)_{\Omega_i}, \quad i \in I_S.
$$
Using the discrete inf-sup condition \eqref{ineq: b_i infsup SD},
we construct $\bm{v}_{h, i}^0 \in V_{h, i}^0$ such that
$$
\nabla \cdot \bm{v}_{h, i}^0 = w_{h, i} - \nabla \cdot \myR_{h, i} \bmu_h \
\text{ in } \ \Omega_i,
\quad \|\bm{v}_{h, i}^0 \|_{1,\Omega_i} \lesssim
\|w_{h, i} - \nabla \cdot \myR_{h, i} \bmu_h\|_{\Omega_i},
$$
and set $\bm{v}_{h,i} = \bm{v}_{h,i}^0 + \myR_{h,i} \bmu_h$. We have
\begin{align*}
  & \sum_{i \in I_S} b_i(\bm{v}_{h,i}, w_{h, i}) = \| w_{h} \|_{\Omega_S}^2, \\
  & \sum_{i \in I_S} \| \bm{v}_{h,i} \|_{V_i} \lesssim
  \sum_{i \in I_S} \| \myR_{h,i} \bmu_h \|_{V_i} + \| w_h \|_{\Omega_S}
  \lesssim \| w_h \|_W,
\end{align*}
using Lemma~\ref{lem: A2 SD} in the last inequality. Combined with the construction
in $\Omega_D$ from Lemma~\ref{lem: Brezzi conditions}, this implies
the inf-sup condition \eqref{ineq: b_infsup}.
\end{proof}

\begin{lemma}[A4] \label{lem: A4 SD}
	The interpolation operator $\Pi^V$ has the property
	\begin{align}
		b(\bm{u} - \Pi^V \bm{u}, w_h) = 0, \quad \forall \, w_h \in W_h.
	\end{align}	
\end{lemma}
\begin{proof}
We first note that $\Pi_i^V$ is b-compatible for the pair $V_{h, i}
\times W_{h, i}$ for $i \in I_D$. For $i \in I_S$, b-compatibility of
$\Pi_i^V$ is ensured by \eqref{Pi-Stokes W_h}. The arguments from
Lemma~\ref{lem: B-compatible} now provide the result.
\end{proof}	

\begin{lemma} \label{lem: consistency error SD}
If A5 holds, then the consistency error $\mathcal{E}_c$ satisfies
\begin{align*}
\mathcal{E}_c \lesssim&\ 
\sum_{i \in I_S} 
\|  \sigma_S \bm{\nu} - \Pi^{V^\Gamma}_i (\sigma_S \bm{\nu}) \|_{\Gamma_i}
+ h^{-1/2} \sum_{i \in I_D} \| p_D - \mathcal{Q}_{h, i} p_D \|_{\Gamma_i} 
\end{align*}
\end{lemma}

\begin{proof}
We consider the term in the numerator of the definition
\eqref{consist-error-general} of $\mathcal{E}_c$. We recall the definitions
of the bilinear forms in \eqref{eqs: bilinear forms SD} and apply
integration by parts. Since $(\bm{u}, p)$ is the solution to
\eqref{eqs: Stokes-Darcy}, we substitute the momentum balance
\eqref{eq: mom balance SD}, Darcy's law \eqref{eq: Darcy SD}, the
BJS interface condition in \eqref{eq: interface SD 1}, and the boundary
conditions \eqref{eq: BC SD} to derive
\begin{align*} 
  & \sum_{i \in I_S} \big( a_i(\bm{u}, \bm{v}_h)
  - b_i(\bm{v}_h, p) - (\bm{g}, \bm{v}_h)_{\Omega_i} \big)
+ \sum_{i \in I_D} \big( a_i(\bm{u}, \bm{v}_h) - b_i(\bm{v}_h, p) \big) \\
& \quad = 
\sum_{i \in I_S} \big( (\sigma_S \bm{\nu}_i, \bm{v}_{h, i})_{\Gamma_i}
+ (\beta \bm{\nu}_i \times \bm{u}_i, \bm{\nu}_i \times \bm{v}_{h,i})_{\Gamma_i \cap \Gamma_{SD}} \big)
+ \sum_{i \in I_D}- (p_D, \bm{\nu}_i \cdot \bm{v}_{h, i})_{\Gamma_i} \\
& \quad = 
\sum_{i \in I_S} \big( (\sigma_S \bm{\nu}_i, \bm{v}_{h, i})_{\Gamma_i \cap \Gamma_{SS}}
+ (\bnu_i\cdot\sigma_S \bm{\nu}_i, \bnu_i\cdot\bm{v}_{h, i})_{\Gamma_i \cap \Gamma_{SD}} \big)
+ \sum_{i \in I_D} - (p_D, \bm{\nu}_i \cdot \bm{v}_{h, i})_{\Gamma_i}.
\end{align*}
The terms on $\Gamma_{DD}$ are bounded in Section~\ref{ssub: err_projection_flat}.
For the terms on $\Gamma_{SS}$ we proceed in a similar way. Let
$\bv_{h,i} = \bv_{h,i}^0 + \myR_{h,i} \bmu_{h}$.  Using the orthogonality
property \eqref{Qh-orth}, the continuity of $\sigma_S$ on $\Gamma_{SS}$,
and condition A5, we obtain
\begin{align*}
& \sum_{i \in I_S} (\sigma_S \bm{\nu}_i, \bm{v}_{h, i})_{\Gamma_i \cap \Gamma_{SS}} 
= \sum_{i \in I_S} (\sigma_S \bm{\nu}_i, \myQ_{h,i} \bmu_h)_{\Gamma_i \cap \Gamma_{SS}} \\
  & \quad = \sum_{i \in I_S} \big( (\sigma_S \bm{\nu}_i - \Pi^{V^\Gamma}_i(\sigma_S \bm{\nu}_i) ,
  \myQ_{h,i} \bmu_h)_{\Gamma_i \cap \Gamma_{SS}}
  + (\Pi^{V^\Gamma}_i(\sigma_S \bm{\nu}_i) , \bmu_h)_{\Gamma_i \cap \Gamma_{SS}} \big) \\
  & \quad = \sum_{i \in I_S} \big( (\sigma_S \bm{\nu}_i - \Pi^{V^\Gamma}_i(\sigma_S \bm{\nu}_i) ,
  \myQ_{h,i} \bmu_h)_{\Gamma_i \cap \Gamma_{SS}}
  + (\Pi^{V^\Gamma}_i(\sigma_S \bm{\nu}_i) - \sigma_S \bm{\nu}_i, \bmu_h)_{\Gamma_i \cap \Gamma_{SS}} \big) \\
  & \quad \lesssim \sum_{i \in I_S}
  \|\sigma_S \bm{\nu}_i - \Pi^{V^\Gamma}_i(\sigma_S \bm{\nu}_i)\|_{\Gamma_i \cap \Gamma_{SS}}
  \|\myQ_{h,i} \bmu_h\|_{\Gamma_i \cap \Gamma_{SS}} \\
  & \quad = \|\sigma_S \bm{\nu}_i - \Pi^{V^\Gamma}_i(\sigma_S \bm{\nu}_i)\|_{\Gamma_i \cap \Gamma_{SS}}
  \|\bv_{h,i}\|_{\Gamma_i \cap \Gamma_{SS}} \\
  & \quad \lesssim \sum_{i \in I_S}
  \|\sigma_S \bm{\nu}_i - \Pi^{V^\Gamma}_i(\sigma_S \bm{\nu}_i)\|_{\Gamma_i \cap \Gamma_{SS}}
  \| \bm{v}_{h, i} \|_{1, \Omega_i},
\end{align*}
using the trace inequality, for all $i \in I_S$,
$\| \bm{v}_{h, i} \|_{\Gamma_i} \lesssim \| \bm{v}_{h, i} \|_{1, \Omega_i}$.

It remains to bound the terms on $\Gamma_{SD}$. Note that there are contributions
from $\Omega_S$ and $\Omega_D$. For $i \in I_S$, we first note that the locality of
the orthogonality \eqref{Qh-orth} for each flat face $F$ implies that
$(\bnu_i\cdot(\mathcal{Q}_{h, i} \bl - \bl), \bnu_i\cdot\bchi_{h,i})_{\Gamma_i\cap\Gamma_{SD}} = 0$
$\forall \bchi_{h,i} \in V^F_{h,i}$.
Using this, the term
$(\bnu_i\cdot\sigma_S \bm{\nu}_i, \bnu_i\cdot\bm{v}_{h, i})_{\Gamma_i \cap \Gamma_{SD}}$
is manipulated as in the above argument, while the Darcy term
$-(p_D, \bm{\nu}_i \cdot \bm{v}_{h, i})_{\Gamma_i\cap\Gamma_{SD}}$ is manipulated as
in Section~\ref{ssub: err_projection_flat}. The two expressions are combined using
the interface condition \eqref{eq: interface SD 2}, resulting in the bound
\begin{align*}
& \sum_{i \in I_S}
(\bnu_i\cdot\sigma_S \bm{\nu}_i, \bnu_i\cdot\bm{v}_{h, i})_{\Gamma_i \cap \Gamma_{SD}}
+ \sum_{i \in I_D} - (p_D, \bm{\nu}_i \cdot \bm{v}_{h, i})_{\Gamma_i \cap \Gamma_{SD}} \\
& 
\quad \lesssim \sum_{i \in I_S}
\|\sigma_S \bm{\nu}_i - \Pi^{V^\Gamma}_i(\sigma_S \bm{\nu}_i)\|_{\Gamma_i \cap \Gamma_{SD}}
\| \bm{v}_{h, i} \|_{1, \Omega_i} \\
& \qquad
+ h^{-1/2}\sum_{i \in I_D}\| p_D - \mathcal{Q}_{h, i} p_D \|_{\Gamma_i\cap\Gamma_{SD}}\| \bm{v}_{h, i} \|_{\Omega_i}.
\end{align*}
The proof is completed by collecting the bounds on $\Gamma_{DD}$,
$\Gamma_{SS}$, and $\Gamma_{SD}$.
\end{proof}

\subsubsection{Reduction to an Interface Problem}
\label{ssub:reduction_to_an_interface_problem_SD}

The coupled Stokes-Darcy problem can be reduced to a flux-mortar
interface problem following the four steps
\eqref{preproblem}--\eqref{eq: postproblem} from
Section~\ref{sec:reduction_to_an_interface_problem}. To that end,
we introduce the
following preliminary definitions. Let $\tilde\myR_h: \Lambda_h \to V_h$ be a
generic extension operator such that $\Tr_i \tilde\myR_{h,i}\mu =
\mathcal{Q}_{h,i}\mu$. For implementation reasons, we choose
$\tilde\myR_h$ to have minimal support. Let $B: \Lambda_h \to S_H$
be such that $(B \mu_h, s_H)_\Omega := b(\tilde\myR_h \mu_h, s_H)$ for all $(\mu_h,
s_H) \in \Lambda_h \times S_H$.  Next, let $\Lambda_h^0 := \mathrm{Ker} B
\subseteq \Lambda_h$ and let $\overline\Lambda_h$ be its orthogonal
complement. We note the following corollary to Lemma~\ref{lem: A3 SD}.
\begin{corollary} \label{cor: inf-sup S_H SD}
	The following inf-sup condition holds for the spaces $\Lambda_h \times S_H$:
	\begin{align*}
	\forall & s_H \in S_H, \ \exists 0 \ne \mu_h \in \Lambda_h \text{ such that }
	b(\myR_h \mu_h, s_H) \gtrsim \| \mu_h \|_\Lambda \| s_H \|_W.
	\end{align*}
\end{corollary}
\begin{proof}
	Setting $w_h := s_H \in S_H \subseteq W_h$ in the proof of Lemma~\ref{lem: A3 SD}
	leads to a pair $(\bm{v}_h^0, \mu_h)$ with $\bm{v}_h^0 = 0$, $\| \mu_h \|_\Lambda \lesssim \| s_H \|_W$, and $b(\myR_h \mu_h, s_H) = \| s_H \|_W^2$.
\end{proof}

Using this corollary, it follows that $B$ is an isomorphism from
$\overline\Lambda_h$ to $S_H$ by the same arguments as in
Section~\ref{sec:reduction_to_an_interface_problem}. Let $\bar f \in S_H$
be the mean value of $f$ on each interior subdomain. This allows us to perform
the first step, namely to solve a coarse problem for $\overline\lambda_f \in
\overline\Lambda_h$ such that
\begin{align}
  b(\tilde\myR_h \overline\lambda_f, s) = (f, s)_\Omega,
  \quad \forall s &\in S_H, \label{preproblem-SD}
\end{align}
or equivalently, $B \overline\lambda_f = \bar f$.

The second step consists of solving independent, local subproblems in the following form: 
Find
$(\bm{u}_f^0, p_f^0, r_f) \in V_h^0 \times W_h \times S_H$ such that
\begin{subequations}\label{pre-2 SD}
\begin{align}
	a(\bm{u}_f^0, \bm{v}^0) - b(\bm{v}^0, p_f^0)
	&=
	- a(\tilde\myR_h \overline\lambda_f, \bm{v}^0)
	+ (\bm{g}_S, \bm{v}^0)_{\Omega_S}
	,
	& \forall \bm{v}^0 &\in V_h^0, \\
	b(\bm{u}_f^0, w) - (r_f, w)_{\Omega}
	&= - b(\tilde\myR_h \overline\lambda_f, w) + (f, w)_{\Omega} ,
	& \forall w &\in W_h,
	\\
	(p_f^0, s)_{\Omega}
	&= 0,
	& \forall s &\in S_H.
\end{align}
\end{subequations}
We remark that the velocity $\bm{u}_f := \bm{u}_f^0 + \tilde\myR_h
\overline\lambda_f$ satisfies the mass conservation equation
$b(\bm{u}_f, w) = (f, w)_{\Omega}$ for all $w \in W_h$. This
function therefore needs to be updated with a divergence-free
velocity in order to satisfy the remaining equations.

The reduced interface problem now forms the third step:
Find $\lambda^0 \in \Lambda_h^0$ such that
\begin{align}\label{eqs: reduced problem SD}
a(\myR_h \lambda^0, \tilde\myR_h \mu^0) - b(\tilde\myR_h \mu^0, p^{\lambda^0}) &=
- a(\bm{u}_f, \tilde\myR_h \mu^0) + b(\tilde\myR_h \mu^0, p_f^0)
+ (\bm{g}_S, \tilde\myR_h \mu^0)_{\Omega_S}, 
\end{align}
for all $\mu^0 \in \Lambda_h^0$. Here, the pair $(\myR_h \lambda^0,
p^{\lambda^0})$ solves the discrete extension problem \eqref{eqs: R
  problem general}. Using the same arguments as in Lemma~\ref{lem:
  SPD}, it follows that \eqref{eqs: reduced problem SD} corresponds to
a symmetric, positive definite operator. Hence, the problem admits a
unique solution that can be obtained through the use of iterative
schemes such as the CG method. Each CG iteration 
requires solving Dirichlet subdomain problems with data
$\mathcal{Q}_{h, i} \lambda^0$ in both the Stokes and Darcy regions.

In analogy with Section~\ref{sec:reduction_to_an_interface_problem},
the velocity $\myR_h \lambda^0$ updates $\bm{u}_f$ such that Darcy's law in
$\Omega_D$ and the momentum balance equations in $\Omega_S$ are
satisfied for test functions in $\myR_h \Lambda_h^0$. Additionally,
this update enforces the BJS condition on the
interface $\Gamma_{SD}$, cf. \eqref{eq: interface SD 1}, due to the
definition of the bilinear form $a(\cdot,\cdot)$ in \eqref{eqs:
  bilinear forms SD}.

It remains to enforce the momentum balance equations and Darcy's law
for test functions in $\overline\Lambda_h$. We perform the
fourth and final step: Find $\overline{p}_\lambda \in S_H$ such that
\begin{align} \label{eq: postproblem SD}
	b(\tilde\myR_h \overline\mu, \overline{p}_\lambda)
	&= a(\bm{u}_f + \myR_h \lambda^0, \tilde\myR_h \overline\mu) - b(\tilde\myR_h \overline\mu, p^{\lambda^0} + p_f^0) + (\bm{g}_S, \tilde\myR_h \overline\mu)_{\Omega_S}, &
	\forall \overline\mu \in \overline\Lambda_h.
\end{align}
Note that this is a coarse problem of the form $B^T
\overline{p}_\lambda = g$ and, since $B$ is an isomorphism,
$\overline{p}_\lambda$ exists uniquely.

Finally, the solution $(\bm{u}, p)$ to the variational formulation \eqref{eq: general form}
of \eqref{eqs: Stokes-Darcy} is obtained by setting
$\bm{u} := \bm{u}_f + \myR_h \lambda^0 = \bm{u}_f^0 + \myR_h \lambda^0 + \tilde\myR_h\overline\lambda_f$
and
$p := p_f^0 + p^{\lambda^0} + \overline{p}_\lambda$.

\section{Numerical results}
\label{sec:numerical_results}

\iftrue 

In this section, we return to the model problem describing porous
medium flow and test the theoretical results from
Section~\ref{sec:a_priori_analysis} with the use of a numerical
experiment. The numerical code, implemented in DuMu$^\text{X}$
\cite{Kochetal2020Dumux,flemisch2011dumux}, is available for download at
\href{https://git.iws.uni-stuttgart.de/dumux-pub/boon2019a}{\texttt{git.iws.uni-stuttgart.de/dumux-pub/boon2019a}}. The
lowest order Raviart-Thomas mixed finite element method reduced to a
finite volume scheme with a two-point flux approximation (TPFA) is
applied in each subdomain and we solve the problem using the iterative
scheme described in
Section~\ref{sec:reduction_to_an_interface_problem}. On the mortar
grids, we investigate two options, namely the use of piecewise
constant functions ($\mathcal{P}_0$) and linear Lagrange basis
functions ($\mathcal{P}_1$). Moreover, both the projection operators
$\mathcal{Q}_h^\flat$ and $\mathcal{Q}_h^\sharp$ are considered in
order to cover all results from Section~\ref{sub:error_estimates}.

The set-up of the test is as follows. Let the domain $\Omega = \left[ 0, 1
  \right] \times \left[ 0, 2 \right]$, the permeability $K = 1$, and
the pressure and velocity be given by:

\begin{subequations}
\begin{align}
    p \left( x, y \right)
		&= y^2 \left( 1 - \frac{y}{3} \right) + x \left( 1-x \right) y \sin \left( 2 \pi x \right),
    	\label{eq:example_p_exact} \\
    \bm{u} \left( x, y \right) &=
            - \begin{bmatrix}
                y \left( \left(1-2x \right) \sin \left(2\pi x \right) - 2 \pi \left(x-1\right) x \cos \left( 2 \pi x \right) \right) \\
                \left( 2-y \right) y + x \left( 1 - x \right) \sin \left( 2 \pi x \right)
            \end{bmatrix}.
    \label{eq:example_u_exact}
\end{align}
\end{subequations}
We prescribe the pressure on the boundary $\partial \Omega$ and define the source function $f := \nabla \cdot \bm{u}$ to match with these chosen distributions.

We partition the domain into four subdomains by introducing interfaces
along the lines $x = 0.5$ and $y = 1$.  In order to investigate the
convergence rates from Section~\ref{sub:error_estimates}, we test a
sequence of refinements by a factor two.
Each subdomain is meshed with a rectangular grid such that the meshes
are non-matching at each of the four interfaces. The mortar grids are
generated such that each interface has the same number of elements.
We refer to the coarsest mesh size of the horizontally aligned mortar
grids as $h_\Gamma^0$ and we consider the two cases $h_\Gamma^0 \in
\{1/4, 1/6\}$. The initial discretization therefore has
either $2$ or $3$ elements on each interface $\Gamma_{ij}$. For an
illustration of the grid and the solution $(\bm{u}, p)$, we refer to
Figure~\ref{fig:resultPic}.

\begin{figure}
	\centering
	\begin{subfigure}{0.328\textwidth}
		\centering
		\includegraphics[height=3.0in]{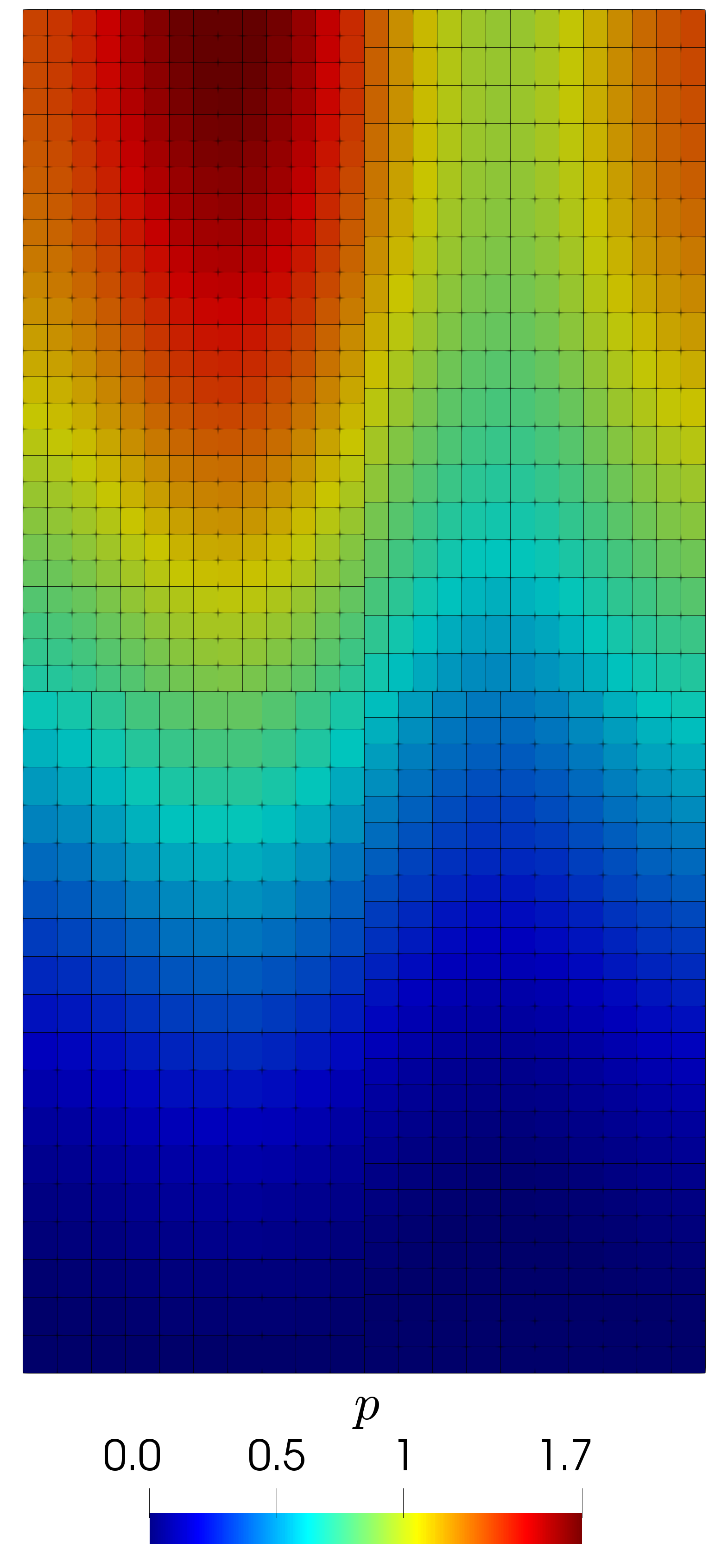}
	\end{subfigure}
	\begin{subfigure}{0.328\textwidth}
		\centering
		\includegraphics[height=3.0in]{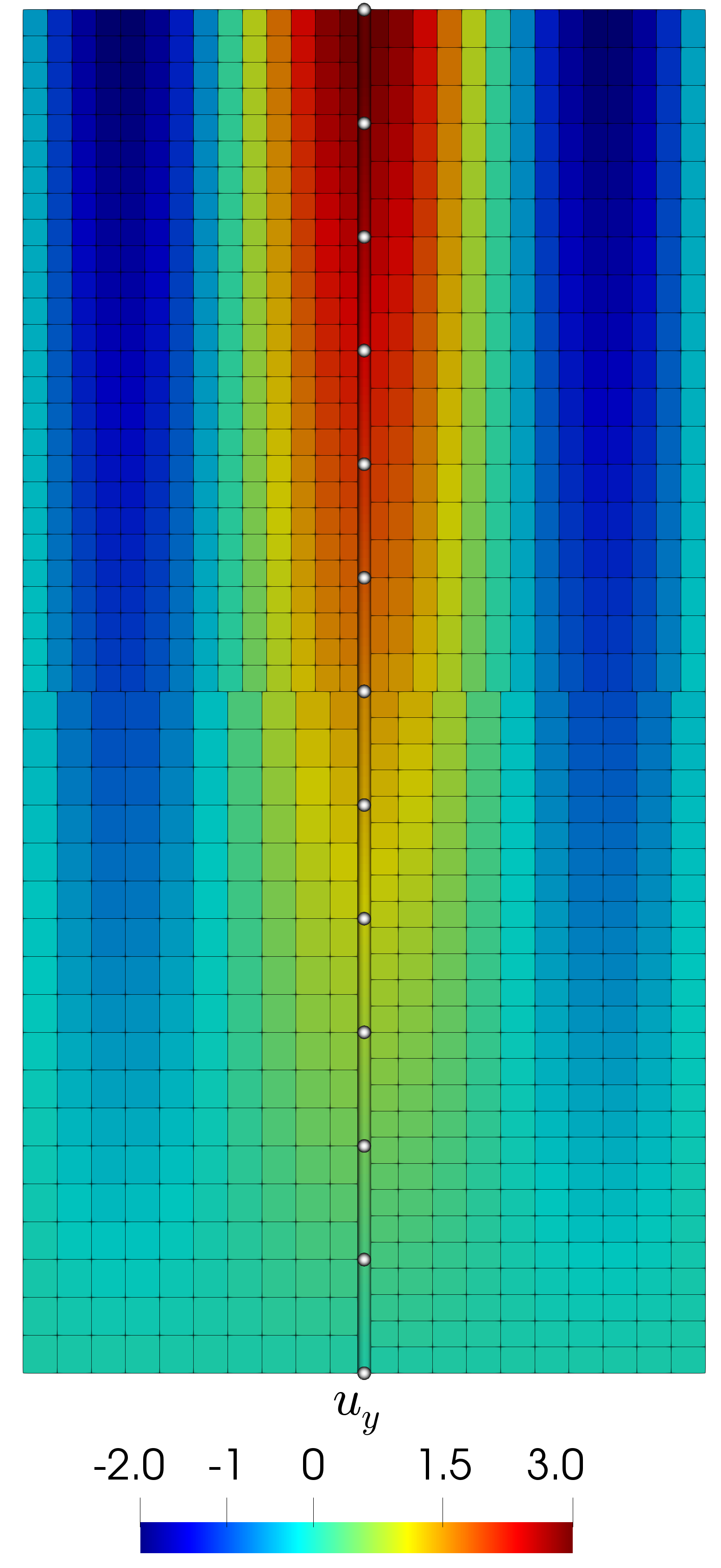}
	\end{subfigure}
	\begin{subfigure}{0.328\textwidth}
		\centering
		\includegraphics[height=3.0in]{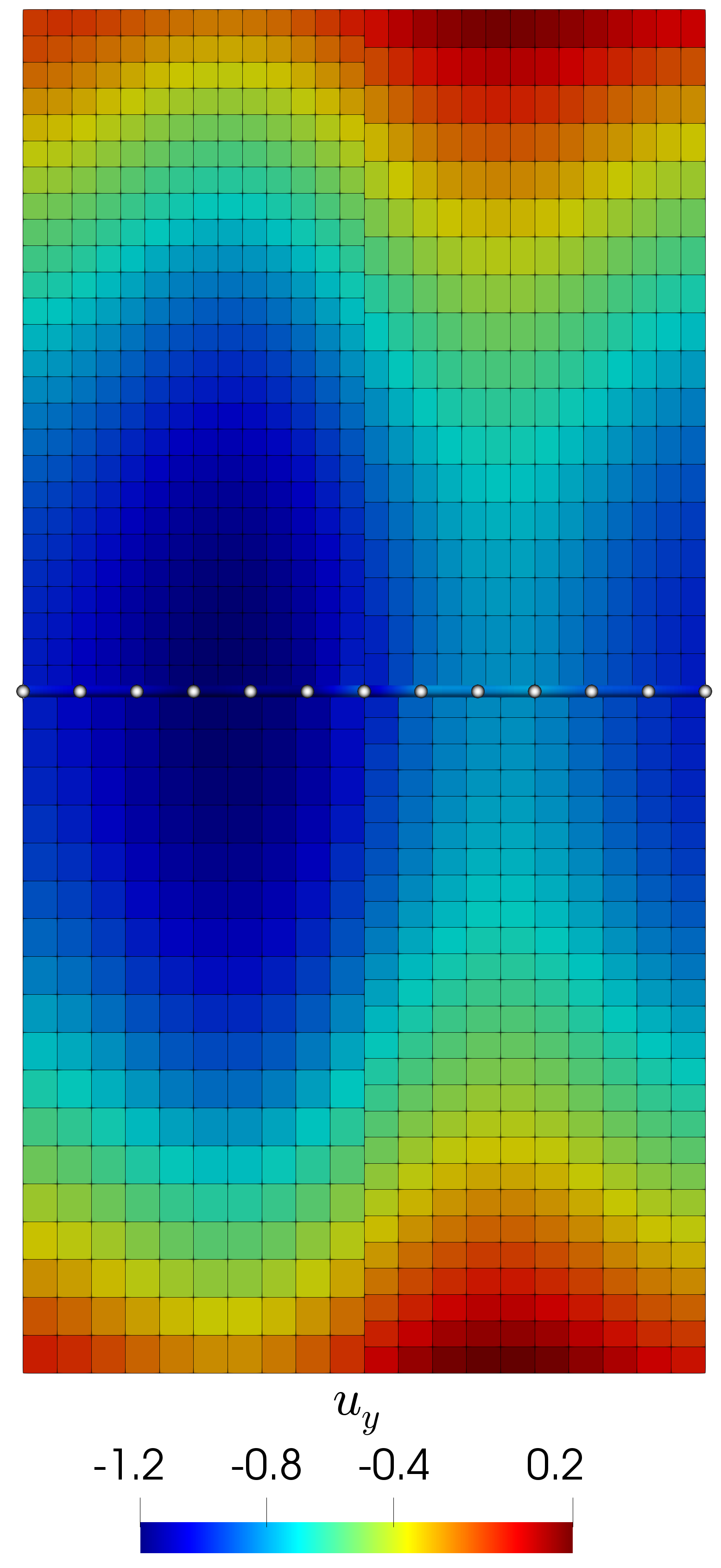}
	\end{subfigure}
	\caption{Pressure (left) and velocity (center and right)
          distributions computed after the first refinement using
          continuous, piecewise linear mortars ($\mathcal{P}_1$) and
          initial mesh size $h_\Gamma^0 = 1/6$. The vertical (center)
          and horizontal (right) mortar grids are visualized as tubes
          with white circles indicating the vertices.}
	\label{fig:resultPic}
\end{figure}

We analyze the decrease of the $L^2$ errors of the velocity $e_u :=
\|\bu - \bu_h\|_\Omega$ , the pressure $e_p := \|p - p_h\|_\Omega$,
the flux-mortar $e_\lambda := \|\lambda - \lambda_h\|_\Gamma$, and the
projected flux-mortar $e_{\myQ\lambda} := \|\lambda -
\myQ_h\lambda_h\|_\Gamma$. Convergence results for $h_\Gamma^0 = 1/4$
are presented in Tables \ref{tab:errorsRatesCG1_2cell} and
\ref{tab:errorsRatesDG0_2cell} for $\mathcal{P}_1$ and $\mathcal{P}_0$
mortars, respectively. The mortar grid is sufficiently coarse
and the mortar conditions \eqref{eq: flat mortar condition} and
\eqref{eq: sharp mortar condition} are satisfied. The rates $r_u$ and
$r_p$ indicate first order convergence for the velocity and pressure
for both projectors $\mathcal{Q}_h^\flat$ and $\mathcal{Q}_h^\sharp$
and both $\mathcal{P}_1$ and $\mathcal{P}_0$ mortars.  We note that
the theory predicts $O(h)$ convergence only for $\mathcal{Q}_h^\sharp$
in the case of $\mathcal{P}_1$ mortars, while $O(h^{1/2})$ is predicted
in the other cases. The results indicate that consistency error
$\mathcal{E}_c$, cf. Sections \ref{ssub: err_projection_sharp} and
\ref{ssub: err_projection_flat}, does not have a
noticeable influence at these mesh sizes. For the mortar variable, we
observe that the rates $r_\lambda$ and $r_{\myQ\lambda}$ are lower by
approximately one half compared to $r_u$ and $r_p$. This is in
agreement with Lemma~\ref{lem: mortar error}.

The most striking observation in both of these tables is that the two
extension operators $\myR_h^\flat$ and $\myR_h^\sharp$ produce nearly
indistinguishable solutions. However, we have verified numerically
that $\myR_h^\flat$ does not produce velocity fields with weakly
continuous fluxes across the interfaces, so it is indeed different
from $\myR_h^\sharp$. The closeness of the results is another
indication that the interface consistency error $\mathcal{E}_c$ is
dominated by the subdomain discretization error.

\begin{table}[tbhp]
    {\footnotesize
      \caption{Errors and convergence rates for $h_\Gamma^0 = 1/4$ and
        $\mathcal{P}_1$ mortars.}
   \label{tab:errorsRatesCG1_2cell}
    \begin{center}
    \begin{tabular}{l |l l |l l |l l |l l}
        \toprule
        $\mathcal{P}_1$ & $e_u^\flat$          & $r_u^\flat$ & $e_p^\flat$   & $r_p^\flat$ &   $e_\lambda^\flat$  &  $r_\lambda^\flat$  &  $e_{\mathcal{Q}\lambda}^\flat$    &  $r_{\mathcal{Q}\lambda}^\flat$ \\
        \midrule
        0               & {7.05e-2}            &             & {4.43e-2}     &             & {3.78e-2}            &                     & {1.11e-1}                          &       	\\
        1               & {2.76e-2}            & {1.35}      & {2.18e-2}     & {1.02}      & {1.78e-2}            & {1.08}              & {5.22e-2}                          & {1.01}	\\
        2               & {1.26e-2}            & {1.14}      & {1.08e-2}     & {1.01}      & {1.13e-2}            & {0.65}              & {2.79e-2}                          & {0.91}	\\
        3               & {6.11e-3}            & {1.04}      & {5.42e-3}     & {1.00}      & {7.91e-3}            & {0.52}              & {1.59e-2}                          & {0.81}	\\
        4               & {3.03e-3}            & {1.01}      & {2.71e-3}     & {1.00}      & {5.58e-3}            & {0.50}              & {9.62e-3}                          & {0.72}	\\
		5               & {1.51e-3}            & {1.00}      & {1.35e-3}     & {1.00}      & {3.95e-3}            & {0.50}              & {6.15e-3}                          & {0.64}	\\
		\toprule
         & $e_u^\sharp$        & $r_u^\flat$ & $e_p^\sharp$ & $r_p^\sharp$   &   $e_\lambda^\sharp$    &     $r_\lambda^\sharp$   &    $e_{\mathcal{Q}\lambda}^\sharp$    &  $r_{\mathcal{Q}\lambda}^\sharp$ \\
        \midrule
        0       &  {7.05e-2}   &             &  {4.43e-2}   &                &   {3.78e-2}             &                          & {1.05e-1}                             &         	\\
        1       &  {2.76e-2}   &    {1.35}   &  {2.18e-2}   &    {1.04}      &   {1.79e-2}             &      {1.08}              & {5.23e-2}                             &   {1.01}	\\
        2       &  {1.29e-2}   &    {1.14}   &  {1.08e-2}   &    {1.01}      &   {1.14e-2}             &      {0.65}              & {2.79e-2}                             &   {0.91}	\\
        3       &  {6.11e-3}   &    {1.04}   &  {5.42e-3}   &    {1.00}      &   {7.92e-3}             &      {0.52}              & {1.59e-2}                             &   {0.81}	\\
        4       &  {3.03e-3}   &    {1.01}   &  {2.71e-3}   &    {1.00}      &   {5.59e-3}             &      {0.50}              & {9.62e-3}                             &   {0.72}	\\
		5       &  {1.51e-3}   &    {1.01}   &  {1.35e-3}   &    {1.00}      &   {3.95e-3}             &      {0.50}              & {6.15e-3}                             &   {0.64}	\\
        \bottomrule
    \end{tabular}
    \end{center}
    }
\end{table}

\begin{table}[tbhp]
    {\footnotesize
    \caption{Errors and convergence rates for $h_\Gamma^0 = 1/4$ and $\mathcal{P}_0$ mortars.}
        \label{tab:errorsRatesDG0_2cell}
    \begin{center}
    \begin{tabular}{l |l l |l l |l l |l l}
        \toprule
        $\mathcal{P}_0$ & $e_u^\flat$   & $r_u^\flat$ & $e_p^\flat$   & $r_p^\flat$ &   $e_\lambda^\flat$  &  $r_\lambda^\flat$  &  $e_{\mathcal{Q}\lambda}^\flat$    &  $r_{\mathcal{Q}\lambda}^\flat$ \\
        \midrule
        0               &  {1.37e-1}    &             &  {4.48e-2}    &             &   {3.41e-1}          &                     & {4.20e-1}                          &         	\\
        1               &  {4.78e-2}    &    {1.51}   &  {2.18e-2}    &    {1.04}   &   {1.70e-1}          &    {1.01}           & {2.06e-1}                          &   {1.03}	\\
        2               &  {1.85e-2}    &    {1.37}   &  {1.08e-2}    &    {1.01}   &   {8.56e-2}          &    {0.99}           & {1.03e-1}                          &   {0.99}	\\
        3               &  {7.91e-3}    &    {1.23}   &  {5.42e-3}    &    {1.00}   &   {4.49e-2}          &    {0.93}           & {5.41e-2}                          &   {0.93}	\\
        4               &  {3.72e-3}    &    {1.09}   &  {2.71e-3}    &    {1.00}   &   {2.62e-2}          &    {0.77}           & {3.18e-2}                          &   {0.76}	\\
		5               &  {1.92e-3}    &    {0.96}   &  {1.35e-3}    &    {1.00}   &   {1.88e-2}          &    {0.48}           & {2.33e-2}                          &   {0.45}	\\
		\toprule
          & $e_u^\sharp$        & $r_u^\flat$ & $e_p^\sharp$ & $r_p^\sharp$   &   $e_\lambda^\sharp$    &     $r_\lambda^\sharp$   &    $e_{\mathcal{Q}\lambda}^\sharp$    &  $r_{\mathcal{Q}\lambda}^\sharp$ \\
        \midrule
        0          &  {1.37e-1}   &           &  {4.48e-2}     &                &   {3.41e-1}             &                        & {4.19e-1}                             &         	\\
        1          &  {4.78e-2}   &    {1.52} &  {2.18e-2}     &    {1.04}      &   {1.70e-1}             &      {1.01}            & {2.05e-1}                             &   {1.03}	\\
        2          &  {1.85e-2}   &    {1.37} &  {1.08e-2}     &    {1.01}      &   {8.56e-2}             &      {0.99}            & {1.03e-1}                             &   {0.99}	\\
        3          &  {7.91e-3}   &    {1.23} &  {5.42e-3}     &    {1.00}      &   {4.49e-2}             &      {0.93}            & {5.40e-2}                             &   {0.93}	\\
        4          &  {3.72e-3}   &    {1.09} &  {2.71e-3}     &    {1.00}      &   {2.62e-2}             &      {0.77}            & {3.18e-2}                             &   {0.76}	\\
		5          &  {1.92e-3}   &    {0.96} &  {1.35e-3}     &    {1.00}      &   {1.88e-2}             &      {0.48}            & {2.33e-2}                             &   {0.45}	\\
        \bottomrule
    \end{tabular}
    \end{center}
    }
\end{table}

In Table~\ref{tab:errorsRatesCG1_3cell}, we show the errors and
convergence rates in the case of finer mortar grids with $h_\Gamma^0 =
1/6$.  The results are only shown for piecewise linear mortars and the
projection operator $\mathcal{Q}_h^\flat$, since the other cases
produce similar errors and rates. We observe a deterioration in the
rates $r_\lambda$ and $r_{\myQ\lambda}$. To illustrate this effect, we
show in Figure~\ref{fig:mortarPlot} the mortar solution $\lambda_h$
obtained on refinement level 5 with the coarser mortar grid
$h_\Gamma^0 = 1/4$ and the finer mortar grid $h_\Gamma^0 = 1/6$.  We
first note that in both cases an oscillation appears at the junction
of the two mortar grids. It is likely due to the Gibbs phenomenon at
the end points of the interfaces, since we allow for discontinuity
from one interface to another. This oscillation is localized and it
does not affect the global accuracy.
However, in the finer mortar grid
case, an oscillation is also observed along the entire interface.
This indicates that the mortar condition \eqref{eq: flat mortar
  condition} may be violated in this case.  On the other hand, the
variables $u_h$ and $p_h$ appear unaffected by these oscillations and
exhibit first order convergence in
Table~\ref{tab:errorsRatesCG1_3cell}.

\begin{table}[tbhp]
    {\footnotesize
      \caption{Errors and convergence rates for $h_\Gamma^0 = 1/6$ and $\mathcal{P}_1$ mortars.}
        \label{tab:errorsRatesCG1_3cell}
    \begin{center}
    \begin{tabular}{l |l l |l l |l l |l l}
        \toprule
        $\mathcal{P}_1$ & $e_u^\flat$   & $r_u^\flat$ & $e_p^\flat$   & $r_p^\flat$ &   $e_\lambda^\flat$  &  $r_\lambda^\flat$  &  $e_{\mathcal{Q}\lambda}^\flat$    &  $r_{\mathcal{Q}\lambda}^\flat$ \\
        \midrule
        0               &  {7.08e-2}    &             &  {4.43e-2}    &             &   {4.51e-2}          &                     & {1.10e-1}                          &         	\\
        1               &  {2.82e-2}    &    {1.33}   &  {2.18e-2}    &    {1.03}   &   {3.37e-2}          &    {0.42}           & {6.37e-2}                          &   {0.79}	\\
        2               &  {1.30e-2}    &    {1.13}   &  {1.08e-2}    &    {1.01}   &   {2.41e-2}          &    {0.48}           & {3.90e-2}                          &   {0.71}	\\
        3               &  {6.31e-3}    &    {1.03}   &  {5.42e-3}    &    {1.00}   &   {1.74e-2}          &    {0.47}           & {2.55e-2}                          &   {0.61}	\\
        4               &  {3.15e-3}    &    {1.00}   &  {2.71e-3}    &    {1.00}   &   {1.29e-2}          &    {0.43}           & {1.78e-2}                          &   {0.52}	\\
		5               &  {1.60e-3}    &    {0.98}   &  {1.35e-3}    &    {1.00}   &   {9.91e-3}          &    {0.38}           & {1.34e-2}                          &   {0.42}	\\
        \bottomrule
    \end{tabular}
    \end{center}
    }
\end{table}

\begin{figure}
	\centering
	\begin{subfigure}{0.49\textwidth}
		\centering
		\includegraphics[width=\textwidth]{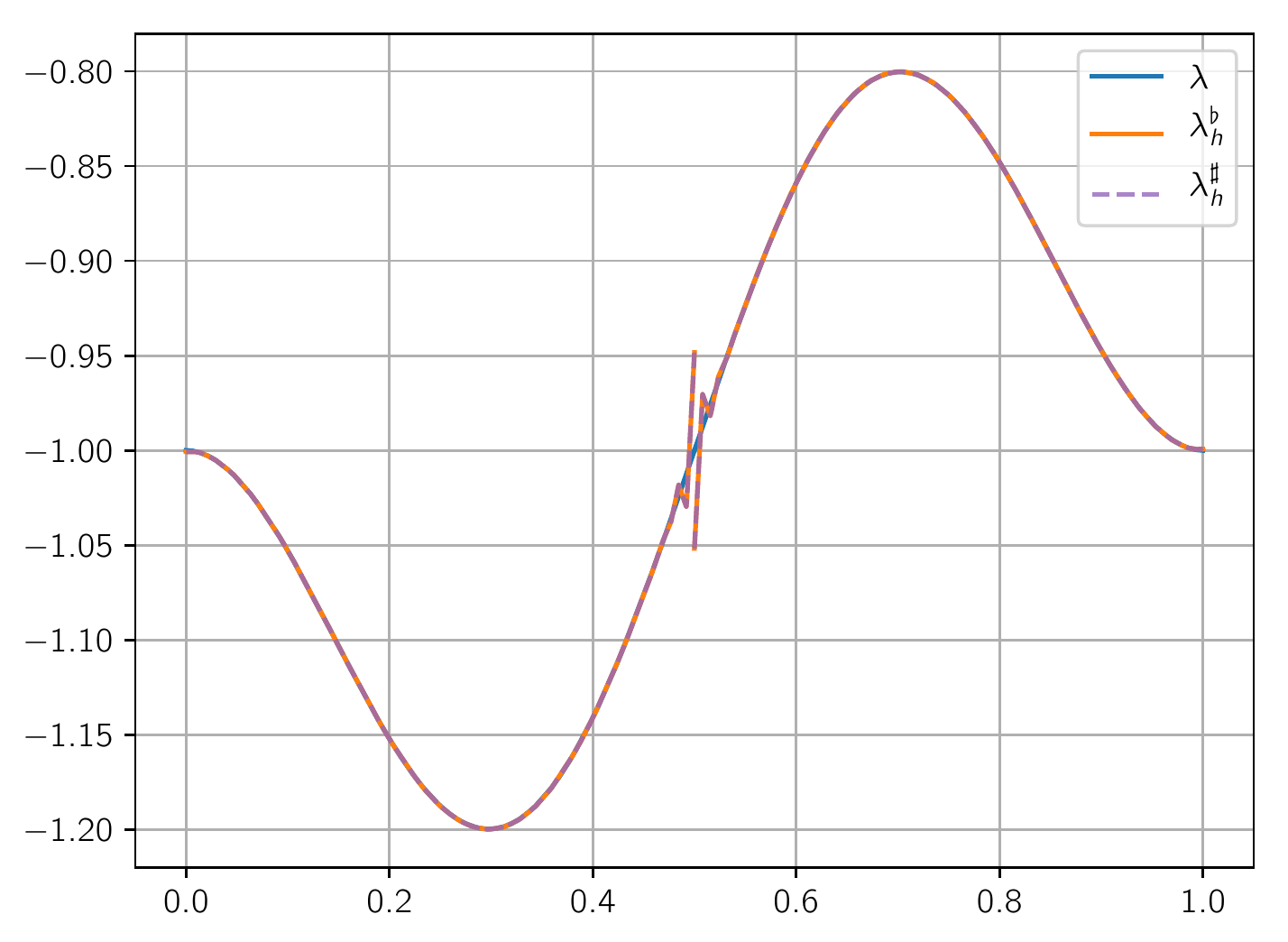}
	\end{subfigure}
	\begin{subfigure}{0.49\textwidth}
		\centering
		\includegraphics[width=\textwidth]{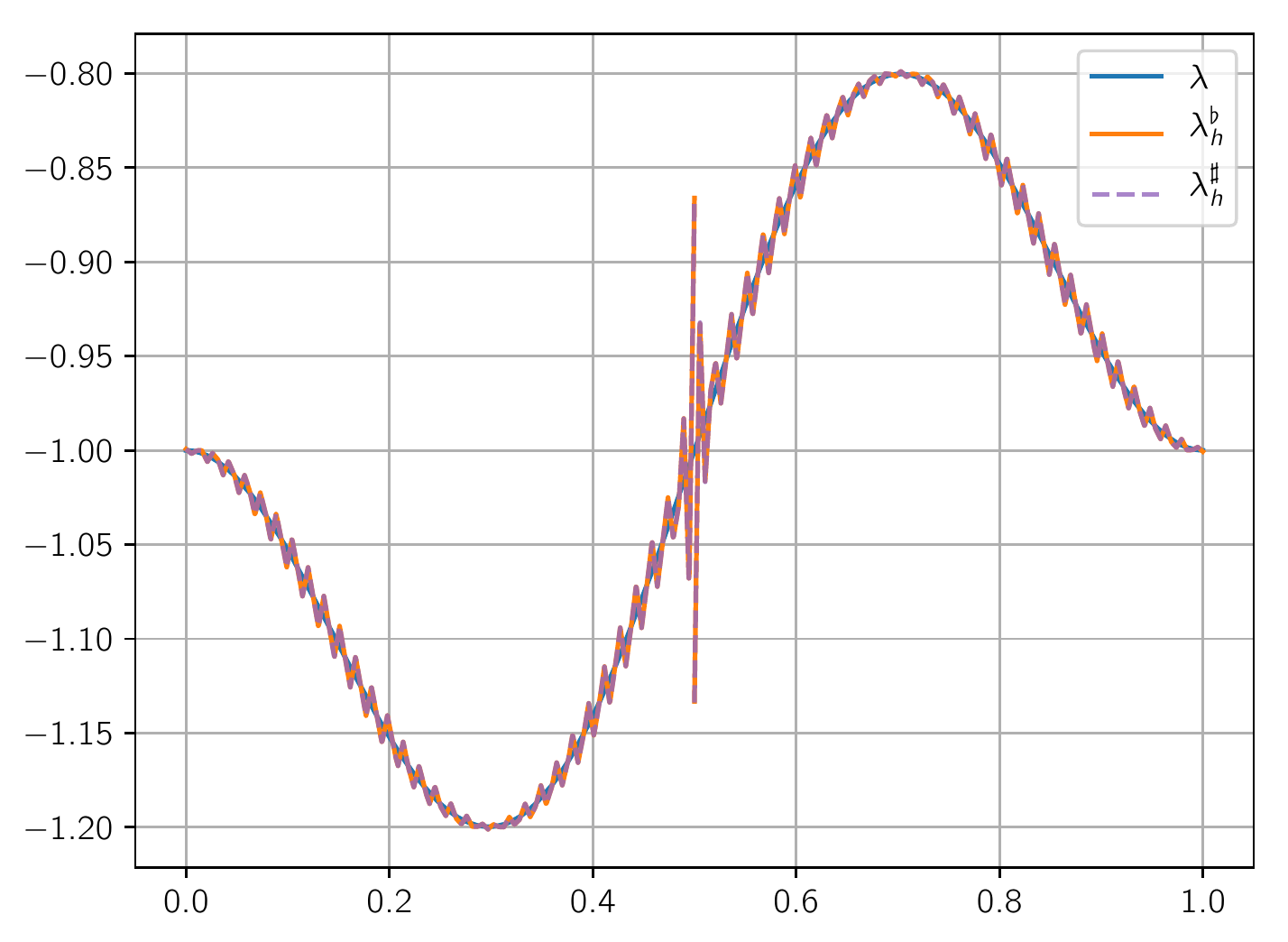}
	\end{subfigure}
	\caption{Plot of the true flux $\lambda$ and the discrete
          flux-mortar solution $\lambda_h$ along the line $y = 1$, on
          refinement level 5 with initial mortar mesh size of $h_\Gamma^0 = 1/4$ (left)
          and $h_\Gamma^0 = 1/6$ (right).}
        \label{fig:mortarPlot}
\end{figure}

\fi

\bibliographystyle{siamplain}
\bibliography{biblio}

\end{document}